\documentclass[]{article}

\usepackage[left=2.5cm, right=2.5cm, top=3cm]{geometry}
\usepackage{amssymb}
\usepackage{amsmath}
\usepackage{amsfonts}
\usepackage{amsthm}
\usepackage{mathtools}
\usepackage{listings}
\usepackage{latexsym}
\usepackage{booktabs}
\usepackage{multirow}
\usepackage{graphicx}
\usepackage{xcolor}
\usepackage{epsfig}
\usepackage{epstopdf}
\usepackage{url}
\usepackage{relsize}
\usepackage{hyperref}
\usepackage{textcomp}
\usepackage{indentfirst}
\usepackage{tipa}
\usepackage{subcaption}
\usepackage{dsfont}
\usepackage{float}
\usepackage[mathscr]{euscript}
\usepackage{enumitem}
\usepackage[title]{appendix}
\usepackage{bm}

\usepackage[backend=bibtex,style=numeric,citestyle=numeric,giveninits=true,sortcites=true,maxbibnames=2]{biblatex}

\usepackage{lineno}
\usepackage{color}
\usepackage{tikz}

\usepackage{pgfplots}
\usepackage{commath}
\usepackage{xifthen}
\usepackage[absolute,overlay]{textpos}

\definecolor{DarkBlue}{rgb}{0.00,0.00,0.55}
\definecolor{Black}{rgb}{0.00,0.00,0.00}
\hypersetup{
	linkcolor = DarkBlue,
	anchorcolor = DarkBlue,
	citecolor = DarkBlue,
	filecolor = DarkBlue,
	colorlinks  = true,
	urlcolor = Black
}

\usepackage[capitalise,noabbrev]{cleveref}
\crefname{equation}{}{}
\Crefname{equation}{Equation}{Equations}

\newcommand\blfootnote[1]{%
	\begingroup
	\renewcommand\thefootnote{}\footnote{#1}%
	\addtocounter{footnote}{-1}%
	\endgroup
}

\graphicspath{ {./figures/},{./} }
\DeclareGraphicsExtensions{.eps,.pdf,.png,.jpg}

\newcommand{\R}{\mathbb{R}}

\newcommand{\N}{\mathbb{N}}
\newcommand{\eps}{\varepsilon}
\DeclareMathOperator{\uu}{\mathfrak{u}}

\newcommand{\by}{\bm{y}}
\newcommand{\hby}{\hat{\bm{y}}}

\newcommand{\bd}{\bm{d}}
\newcommand{\hbd}{\hat{\bm{d}}}
\newcommand{\hbv}{\hat{\bm{v}}}
\newcommand{\bg}{\bm{g}}
\newcommand{\bv}{\bm{v}}
\newcommand{\be}{\bm{e}}

\renewcommand{\bf}{\bm{f}}
\newcommand{\hbf}{\hat{\bm{f}}}
\newcommand{\bu}{\bm{u}}
\newcommand{\bh}{\bm{h}}

\newcommand{\Dt}{\Delta t}
\newcommand{\oz}{\omega_0}
\newcommand{\ou}{\omega_1}
\newcommand{\fs}{\bf_S}
\newcommand{\fe}[1][]{  \bf_{\eta \ifthenelse{\isempty{#1}}{}{,#1}} }
\newcommand{\ye}{\by_\eta}
\newcommand{\rhos}{\rho_S}
\newcommand{\rhof}{\rho_F}
\newcommand{\rhoe}{\rho_\eta}
\newcommand{\bfe}[1][]{  \overline{\bf}_{\eta \ifthenelse{\isempty{#1}}{}{,#1}} }
\newcommand{\hbfe}{\hat{\bf}_{\eta}}
\newcommand{\tbfe}{\tilde{\bf}_{\eta}}
\newcommand{\ff}[1][]{  {\bf}_{F \ifthenelse{\isempty{#1}}{}{,#1}} }
\newcommand{\hfs}{\hat{\bf}_S}
\newcommand{\hff}{\hat{\bf}_F}
\newcommand{\hbh}{\hat{\bh}}
\newcommand{\tbh}{\tilde{\bh}}

\newcommand{\vz}{\upsilon_0}
\newcommand{\vu}{\upsilon_1}

\newcommand{\barc}{\bar{c}}
\newcommand{\barm}{\bar{m}}
\newcommand{\Cddf}{C_{\bf''}}

\newcommand{\Cdj}{C_{\bd}}

\definecolor{myred}{RGB}{0, 0, 0}
\newcommand{\moda}[1]{\textcolor{black}{#1}}
\newcommand{\modb}[1]{\textcolor{black}{#1}}
\newcommand{\modc}[1]{\textcolor{black}{#1}}

\newtheorem{theorem}{Theorem}[section]
\newtheorem{lemma}[theorem]{Lemma}

\theoremstyle{definition}
\newtheorem{definition}{Definition}[section]
\newtheorem{example}{Example}[section]
\theoremstyle{remark}
\newtheorem{remark}{Remark}[section]
\theoremstyle{plain}
\newtheorem{assumption}{Assumption}[section]

\crefname{assumption}{Assumption}{Assumptions}
\Crefname{assumption}{Assumption}{Assumptions}
\crefname{remark}{Remark}{Remarks}
\Crefname{remark}{Remark}{Remarks}
\crefname{example}{Example}{Examples}
\Crefname{example}{Example}{Examples}
\crefname{lemma}{Lemma}{Lemmas}
\Crefname{lemma}{Lemma}{Lemmas}

\usepackage{filecontents}

\immediate\write18{bibtex \jobname}
\begin{document}

\title{Mixed-precision explicit stabilized Runge--Kutta methods for single- and multi-scale differential equations\blfootnote{\textbf{Funding:} This research is supported by the ICONIC EPSRC Programme Grant (EP/P020720/1) and by the Swiss National Science Foundation, under grant No. $200020\_172710$.}
}

\author{M.~Croci\thanks{Oden Institute, University of Texas at Austin, Austin, TX, USA. (\textbf{\url{matteo.croci@austin.utexas.edu}})} \and G.~Rosilho de Souza\thanks{ ANMC, Institute of Mathematics, \'Ecole Polytechnique F\'ed\'erale de Lausanne, Lausanne, Switzerland. (\textbf{\url{giacomo.rosilhodesouza@epfl.ch}})}
}

\date{}

\maketitle

\begin{abstract}
Mixed-precision algorithms combine low- and high-precision computations in order to benefit from the performance gains of reduced-precision without sacrificing accuracy. In this work, we design mixed-precision Runge--Kutta--Chebyshev (RKC) methods, where high precision is used for accuracy, and low precision for stability.
Generally speaking, RKC methods are low-order explicit schemes with a stability domain growing quadratically with the number of function evaluations. For this reason, most of the computational effort is spent on stability rather than accuracy purposes.
In this paper, we show that a na\"ive mixed-precision implementation of any Runge--Kutta scheme can harm the convergence order of the method and limit its accuracy, and we introduce a new class of mixed-precision RKC schemes that are instead unaffected by this limiting behaviour. 
We present three mixed-precision schemes: a first- and a second-order RKC method, and a first-order multirate RKC scheme for multiscale problems.
These schemes perform only the few function evaluations needed for accuracy (1 or 2 for first- and second-order methods respectively) in high precision, while the rest are performed in low precision. We prove that while these methods are essentially as cheap as their fully low-precision equivalent, they retain the stability and convergence order of their high-precision counterpart. Indeed, numerical experiments confirm that these schemes are as accurate as the corresponding high-precision method.
\end{abstract}

\begin{textblock*}{\textwidth}(5.92cm,7.2cm)
	Dedicated to the memory of Assyr Abdulle (1971--2021).
\end{textblock*}

\begin{keywords}
Explicit stabilized Runge--Kutta methods, mixed-precision computing, rounding errors, reduced precision, floating-point arithmetic, multirate methods.\vspace{3pt}
\end{keywords}

\begin{MSC}
65L04, 65L06, 65L20, 65M12, 65M20, 65G50, 65G30, 65M15, 65Y99.
\end{MSC}

\section{Introduction}
\label{sec:intro}

Recent years saw the return of hardware-supported low-precision arithmetic, with a drastic increase in the number of chips (GPUs, CPUs, and chips designed for machine learning) supporting the fp16 and bfloat16 half-precision floating-point formats. As a consequence, the design and analysis of algorithms that perform all or part of the computations in reduced precision has now become an active field of investigation \cite{abdelfattah2021survey}. A popular technique is to carefully combine high- and low-precision computations so as to perform most of the heavy lifting in low precision while leaving the precision-sensitive calculations in high precision. The result is a mixed-precision algorithm\footnote{Or multi-precision algorithm, if more than two floating point formats are used.}. Mixed-precision algorithms aim to achieve the best of two worlds: perform computations that are as stable and as accurate as their fully high-precision equivalent, but with the performance benefits (in terms of speed, memory, and energy consumption) of low-precision computations. For these reasons, mixed-precision algorithms have become very popular in the numerical linear algebra \cite{abdelfattah2021survey}, machine learning \cite{das2018mixed,mellempudi2019mixed,micikevicius2017mixed}, climate and weather model simulation \cite{ackmann2021mixed,klower2020number,klower2021fluid,paxton2021climate,vavna2017single,duben2017study}, and in the numerical integration literature \cite{BGG21,Gra20,hairer2008achieving}.

In this paper we design mixed-precision Runge--Kutta (RK) methods for stiff differential equations
\begin{equation}\label{eq:ode}
\by'=\bf(\by),\qquad\qquad \by(0)=\by^0,
\end{equation}
where $\by(t)\in\R^n$ and $\bf:\R^n\rightarrow\R^n$ is a twice differentiable function. We also consider multirate problems 
\begin{equation}\label{eq:mrode}
\by'=\ff(\by)+\fs(\by),\qquad\qquad \by(0)=\by^0,
\end{equation}
where $\ff$ is a cheap but stiff term associated to fast (F) time-scales and $\fs$ is an expensive but mildly stiff term associated to slower (S) time-scales. We do not assume any scale separation, hence, in addition to all fast terms, $\ff$ may contain part of the slow dynamics too. For instance, $\ff$ can be associated to a discrete Laplacian.

Standard explicit Runge--Kutta schemes are exceedingly inefficient for the solution of stiff problems. Therefore, we must resort to implicit or explicit stabilized methods. Usually implicit methods are unconditionally stable at the price of solving a possibly nonlinear system at every time step, which are solved by Newton methods in conjunction with linear algebra routines; therefore, their performance strongly depends on nonlinearities, system size and efficiency of direct solvers or preconditioners when the problem size demands iterative solvers. Furthermore, convergence of Newton methods is not guaranteed for large step sizes. Explicit stabilized Runge--Kutta methods (ESRK) are a compromise between standard explicit and implicit methods. They are fully explicit, hence do not require the solution of linear systems, and their stability domain along the negative real axis grows as $s^2$ for an $s$-stage method. Due to this quadratic relation between work load and stability ESRK methods do not have any step size restriction, require few function evaluations and compete with implicit methods, especially for large nonlinear problems \cite{Abd02,AGR20,DDD13,Med98,VeS04}. A few families of ESRK methods exist, such as the DUMKA methods based on compositions of Euler steps \cite{Leb94,LeM94,Med98}, the Runge--Kutta--Chebyshev (RKC) methods based on recursive formulas for Chebyshev polynomials \cite{SSV98,HoS80,VerwerHundsdorfer1990RKC}, the orthogonal Runge--Kutta--Chebyshev (ROCK) methods based on optimal orthogonal polynomials \cite{Abd02,AbM01} and the Runge--Kutta--Legendre (RKL) methods based on Legendre polynomials \cite{Meyer2014}. \moda{More recently, multirate RKC (mRKC) methods \cite{AGR20} for \cref{eq:mrode}, stochastic versions of RKC, ROCK and mRKC \cite{AAV18,AbL08,AbR22b,AVZ13b}, methods for advection-diffusion problems \cite{Alm22,TaX20} and wave equations \cite{CHS00,GMS21} have been introduced.}

Many new mixed-precision algorithms are being developed by the numerical linear algebra community, among which algorithms for matrix factorization \cite{amestoy2021mixed,blanchard2020mixed,lopez2020mixed,yamazaki2015mixed,yang2021rounding}, iterative refinement \cite{amestoy2021five,carson2017new,carson2018accelerating}, and Krylov subspace methods \cite{agullo2020exploring,gratton2019exploiting}. For an overview of recent developments in mixed-precision computing we refer to this excellent community review \cite{abdelfattah2021survey}. The development of preconditioned iterative linear solvers is an active field of investigation due to the complications arising with loss of orthogonality of the Arnoldi/Lanczos vectors \cite{bjorck1992loss,meurant2006lanczos}. However, some new fascinating results have been obtained for mixed-precision GMRES \cite{gratton2019exploiting}, and flexible GMRES \cite{agullo2020exploring}. Mixed-precision multigrid solvers based on iterative refinement have also been developed \cite{mccormick2021algebraic,tamstorf2021discretization}. We were unable to find any work in the numerical optimization literature specific to mixed-precision nonlinear solvers. However, the work by Tisseur \cite{tisseur2001newton} and various results on inexact Newton-methods \cite{dembo1982inexact} might be applicable here. Overall, there is still much to discover about the behaviour of all the ingredients required by implicit timestepping schemes in finite precision (iterative linear and nonlinear solvers, their preconditioning, and the interplay between these). These considerations inspired our research into explicit stabilized methods. 

In this paper we design and analyze mixed-precision explicit stabilized schemes for \cref{eq:ode,eq:mrode} based on the RKC and the mRKC schemes, respectively. The schemes preserve the original order of convergence of the high-precision methods, but the number of high-precision evaluations of the right-hand side is reduced to the bare minimum. For instance, in an $s$-stage first-order RKC method (RKC1) only one function evaluation is needed for accuracy, and the remaining $s-1$ evaluations are only used to increase stability. With an appropriate reformulation of the scheme, we are able to perform only one function evaluation in high precision and the remaining $s-1$ in a low-precision format without impacting accuracy. Our methodology consists in linearizing the numerical scheme and carefully evaluating the Jacobian of the right-hand side in reduced-precision arithmetic. The mixed-precision first- and second-order RKC schemes for \cref{eq:ode} that we propose in \cref{sec:mixed-precision-ESRK} require only one or two, respectively, high-precision evaluations of the right-hand side. The first-order mixed-precision mRKC scheme for multirate problems \cref{eq:mrode} introduced in \cref{sec:multirate} requires only one high-precision evaluation of $\ff$ and $\fs$. All the function evaluations needed for stability are exclusively performed in a cheaper low-precision format.
In addition to proving that the mixed-precision schemes preserve the right order of convergence, we study the propagation of rounding errors and briefly discuss how the low-precision computations can impact stability. Rounding errors destroy any spectral relation between the integration variables and therefore we were not able to provide a rigorous stability analysis in the traditional ODE sense. Nevertheless, we provide an extensive numerical study of the stability and convergence properties of our mixed-precision schemes.

To our knowledge, the only other works on mixed-precision RK methods in the literature are by Grant \cite{Gra20}, Burnett et al.~\cite{BGG21}, and by Hairer et al.~\cite{hairer2008achieving}. However, their focus is on implicit RK methods and consequently their approach is quite different from ours. In \cite{Gra20} and \cite{BGG21}, the authors consider mixed-precision implicit RK methods where the implicit systems are solved in reduced precision. This operation impacts the order of convergence of the scheme, which is then recovered by performing additional explicit stages in high precision. In our work we instead preserve the order of convergence by performing a single stage in high precision and the remaining in low precision without altering the overall number of stages needed. The authors of \cite{Gra20} and \cite{BGG21} cast their strategy in the framework of additive RK methods and the order conditions are derived using B-series. However, stability is not addressed in general. In \cite{hairer2008achieving}, the authors employ quad precision to evaluate the coefficients of implicit RK methods, and double precision for the remaining computations. However, their focus is on the long-time integration of Hamiltonian systems and consequently the structure and objectives of their work is different from ours.

The remainder of this paper is structured as follows. In \cref{sec:preliminaries} we recall the most common floating-point formats, we introduce the rounding error model used in the paper, and we recall the first- and second-order RKC methods. In \cref{sec:mixed-precision-ESRK} we motivate mixed-precision Runge--Kutta methods, and we present the mixed-precision RKC schemes, together with a few strategies for cheap Jacobian evaluations in reduced precision that avoid the insurgence of stagnation. Later in the same section we also analyze the accuracy and stability of the mixed-precision RKC schemes. In \cref{sec:multirate} we introduce and analyze the mixed-precision multirate RKC scheme. In \cref{sec:numerical_results} we confirm numerically the accuracy and stability properties of the schemes. Finally, in \cref{sec:conclusions} we present our conclusions and final remarks.

\section{Preliminaries}
\label{sec:preliminaries}

\subsection{Floating-point formats used and rounding error model}
A mixed-precision algorithm uses a combination of high- and low-precision computations so as to maximize stability and efficiency. To set the scene, in this paper we only consider the (common) floating point number formats presented in Table \ref{tab:precision}. We typically refer to double precision as ``high precision'' and to any of the other formats in Table \ref{tab:precision} as ``low precision'', albeit our theory and algorithms are not restricted to these choices and are still perfectly valid under other combinations and formats. 

\begin{table}[h!]
	\centering
	\begin{tabular}{@{}llllcc@{}}
		\toprule
		\multicolumn{1}{c}{Format} & \multicolumn{1}{c}{$u$} & \multicolumn{1}{c}{$x_{\min}$} & \multicolumn{1}{c}{$x_{\max}$} & $t$ & exponent bits \\ \midrule
		bfloat16                   & $3.91\times 10^{-3}$    & $1.18\times 10^{-38}$          & $3.39\times 10^{38}$           & $8$         & $8$      \\
		fp16                       & $4.88\times 10^{-4}$    & $6.10\times 10^{-5}$           & $6.55\times 10^{4}$            & $11$        & $5$      \\
		fp32 (single)              & $5.96\times 10^{-8}$    & $1.18\times 10^{-38}$          & $3.40\times 10^{38}$           & $24$        & $8$      \\
		fp64 (double)              & $1.11 \times 10^{-16}$  & $2.22\times 10^{-308}$         & $1.80\times 10^{308}$          & $53$        & $11$     \\ \bottomrule
	\end{tabular}
	\caption{\textit{Overview of the floating point systems mentioned in the paper. Here $u=2^{-t}$ is the roundoff unit, $x_{\min}$ in the smallest normalized positive number, $x_{\max}$ is the largest finite number and $t$ is the precision. While bfloat16 and fp32 have the same range and exponent bits, fp16 has a smaller roundoff unit (higher precision) at the cost of a smaller range.}}
	\label{tab:precision}
	\vspace{-6pt}
\end{table}
\begin{assumption}
	The effects related to floating-point range (e.g.~underflow/overflow) are ignored here for simplicity. However, we remark that most range issues in our mixed-precision algorithms can easily be avoided by simple rescaling and a careful implementation, cf.~Remark \ref{rem:matrix_squeezing}.
\end{assumption}
Let us adopt the following standard floating point error model for round-to-nearest (cf.~Chapter 2 of \cite{higham2002accuracy}):
\begin{align}
\label{eq:std_floating_point_error_model}
\widehat{(a \text{ op } b)} = (a \text{ op } b)(1 + \delta),\quad |\delta| < u,\quad\text{op}\in\{+,-,\times,\backslash\},
\end{align}
where $u$ is the roundoff unit (cf.~\Cref{tab:precision}) and $\delta$ is called a roundoff error. Here and in the rest of the paper we use hats to denote quantities that are the result of finite precision computations. By using this model it is possible to derive \emph{a priori} rounding error bounds for a variety of different algorithms and operations \cite{higham2002accuracy}. The main result we employ in this paper is the backward error bound for matrix-vector products (cf.~Section 3.5 in \cite{higham2002accuracy}):
for a matrix $A\in\R^{\bar{n}\times n}$ with at most $\barm$ nonzero entries per row, and a vector $\bm{x}\in\R^n$ we have that computing the product $\bm{y}=A\bm{x}$ in finite precision yields instead the vector $\hat{\bm{y}}\in\R^{\bar{n}}$ satisfying\footnote{This is the same result as in Section 3.5 of \cite{higham2002accuracy}, but it accounts for the fact that multiplications by zero are performed exactly.}
\begin{align}
\label{eq:matvecs_bound}
\hat{\bm{y}} = \widehat{A\bm{x}} = (A + \Delta A)\bm{x},\quad\text{with}\quad |\Delta A|\leq \gamma_{\barm} |A|,\quad(\Delta A\in\R^{\bar{n}\times n}),\quad \gamma_{\barm} = \frac{\barm u}{1-\barm u}.
\end{align}
Here we denote by $|\cdot|$ the entrywise absolute value and for round-to-nearest $\gamma_{\barm}$ can be replaced with $\bar{\gamma}\barm=\barm u(1+4u+2u^2)\approx \barm u$. This last result is a consequence of the backward error bound for inner products \cite[Corollary 3.2]{LangeRump2017}, after accounting for possibly non-representable entries. Equation \eqref{eq:std_floating_point_error_model} straight-forwardly implies the normwise bounds $||\Delta A \bm{x}||_p\leq \gamma_{\barm}||A||_p||\bm{x}||_p$ for $p=1,\infty$. We will also need a bound for the spectral norm, which we provide in the following lemma.
\begin{lemma}[Lemma 6.6. in \cite{higham2002accuracy}]\label{lemma:boundDA}
	Let $A,B\in\R^{\bar{n}\times n}$ have at most $\barm$ nonzero entries per row and column, and satisfy $|B|\leq c|A|$ for some constant $c>0$. Then $||B||_2\leq c\min(\barm,r^{1/2})||A||_2$, where $r=\text{rank}(A)$. Setting $c=\bar{\gamma}\barm$, for $u$ sufficiently small we then have $||\Delta A||_2\leq \barc\barm^2u||A||_2$ for $\barc=1 + 4u + 2u^2\approx 1$.
\end{lemma}
\begin{proof}
	Let $||A||_{\max}=\max_{ij}|A_{ij}|$. This result is essentially Lemma 6.6 in \cite{higham2002accuracy} after accounting for the sparsity in $A$. Owing to Lemma 6.6 in \cite{higham2002accuracy} we have that $||B||_2\leq || \ |B|\ ||_2\leq c||\ |A|\ ||_2$. Since the $||\cdot{}||_2$ norm of a symmetric matrix is its spectral radius which in turn is a lower bound for any vector-induced norm, we have $||\ |A|\ ||_2=||\ |A|^T|A|\ ||_2^{1/2}\leq ||\ |A|^T|A|\ ||_\infty^{1/2} \leq (||A||_1||A||_\infty)^{1/2}\leq \barm ||A||_{\max}\leq \barm ||A||_2$. For the second inequality, we use the more traditional bound $||\ |A|\ ||_2\leq ||A||_F\leq r^{1/2}||A||_2$.
\end{proof}

\subsection{The Runge--Kutta--Chebyshev methods}\label{sec:RKCmethods}
\label{sec:rkc}
The main goal in the design of classical explicit Runge--Kutta schemes is to reach the highest possible order $p$ for the given number of stages $s$. For instance, as long as $s\leq 4$, we can achieve $p=s$. However, this design strategy leaves no room for enhancing stability. In contrast, explicit stabilized Runge--Kutta methods fix the order $p$ and use an increased number of stages $s\geq p$ to improve the stability properties of the scheme, thereby relaxing the stringent stability conditions affecting classical explicit methods.

In this paper we concentrate on first- and second-order Runge--Kutta--Chebyshev (RKC) methods \cite{SSV98,HoS80,VerwerHundsdorfer1990RKC}, which we denote with RKC1 and RKC2 respectively. We consider problems of the form \cref{eq:ode}, and we present these schemes in delta form since it leads to smaller rounding errors \cite[IV.8]{HairerWanner1996}. Let $\by^n$ be an approximation of $\by(t^n)$, where $t^n=n\Dt$ and $\Dt$ is the step size. One step, of size $\Dt$, of an $s$-stage RKC scheme is given by the recursion
\begin{equation}\label{eq:RKC}
\begin{dcases}
\bd_0 = \bm{0}, \quad \bd_1=\mu_1\Dt\bm{f}(\by^n),\\
\bd_j = \nu_j\bd_{j-1}+\kappa_j\bd_{j-2}+\mu_j\Dt \bm{f}(\by^n+\bd_{j-1})+\gamma_j\Dt\bm{f}(\by^n), \quad j=2,\ldots,s,\\
\by^{n+1}= \by^n+\bd_s.
\end{dcases}
\end{equation}
The coefficients $\mu_j,\nu_j,\kappa_j,\gamma_j$ are given by, for $j=2,\ldots,s$,
\begin{equation}\label{eq:coeffRKC}
\mu_1=b_1\ou,\qquad \mu_j=2\ou b_j/b_{j-1},\qquad \nu_j=2\oz b_j/b_{j-1},\qquad  \kappa_j=-b_j/b_{j-2}, \qquad \gamma_j=-\mu_j a_{j-1},
\end{equation}
and $a_{j}=1-b_{j}T_{j}(\oz)$ for $j=0,\ldots,s$, where $T_j(x)$ is the Chebyshev polynomial of the first kind of degree $j$, defined recursively by
\begin{equation}
T_0(x)=1,\qquad T_1(x)=x,\qquad T_j(x)=2xT_{j-1}(x)-T_{j-2}(x), \quad j\geq 2.
\end{equation}
The core coefficients $\oz,\ou$ and $b_j$ for $j=0,\ldots,s$ depend on the order $p=1,2$, the number of stages $s$, and the so-called damping parameter $\eps\geq 0$. For the first-order RKC1 method it holds
\begin{equation}\label{eq:coeffRKC1}
\oz=1+\eps/s^2,\quad \ou=T_s(\oz)/T_s'(\oz),\quad b_j=1/T_j(\oz), \quad j=0,\ldots,s
\end{equation}
and for the second-order RKC2 method
\begin{equation}\label{eq:coeffRKC2}
\oz=1+\eps/s^2,\quad \ou=T_s'(\oz)/T_s''(\oz),\quad b_0=b_1=b_2, \quad b_j=T_j''(\oz)/T_j'(\oz)^2, \quad j=2,\ldots,s.
\end{equation}
Typical damping parameters are $\eps=0.05$ for RKC1, and $\eps=2/13$ for RKC2 \cite{VerwerHundsdorfer1990RKC}. The purpose of the damping parameter is to increase stability in the imaginary direction, making RKC methods more resilient to small perturbations \cite{HairerWanner1996}. Note that for the RKC1 scheme $a_j=0$ and therefore $\gamma_j=0$. We remark that the RKC1 scheme with $s=1$ is the explicit Euler method. Let
\begin{equation}\label{eq:defc}
c_0=0,\qquad c_1=\mu_1, \qquad c_j=\nu_jc_{j-1}+\kappa_jc_{j-2}+\mu_j+\gamma_j, \quad j=2,\ldots,s.
\end{equation}
Verwer et al.~in \cite{VerwerHundsdorfer1990RKC} show that for RKC1 and RKC2 respectively we have $\bd_j=c_j\Dt\bf(\by^n)+O(\Dt^2)$ and $\bd_j=c_j\Dt\bf(\by^n)+c_j^2\Dt^2/2\bf'(\by^n)\bf(\by^n)+O(\Dt^3)$. Therefore $\by^n+\bd_j$ is respectively a first- or a second-order approximation of the exact solution at time $t^n+c_j\Dt$.

When applied to the Dahlquist test equation $y'=\lambda y$ with $\lambda\in\mathbb{C}_-$, the RKC method \cref{eq:RKC} with coefficients \cref{eq:coeffRKC} yields
\begin{equation}\label{eq:defRsRKC}
y^{n+1}=R_s(z)y^n,\quad\mbox{with}\quad R_s(z)=a_s+b_s T_s(\oz+\ou z), \quad z=\lambda\Dt.
\end{equation}
The polynomial $R_s(z)$ is called the stability polynomial of the method. Using the properties of Chebyshev polynomials, such as the fact that $|T_s(x)|$ is an even function, increasing for $x\geq 1$, that $T_s(1)=1$ and $T_s(x)\in [-1,1]$ for $x\in [-1,1]$, it is possible to show that $|R_s(z)|\leq 1$ for all $z$ such that $-\oz\leq \oz+\ou z\leq \oz$, i.e. for all $z\in [-\ell_s^\eps,0]$, where $\ell_s^\eps=2\oz/\ou$ \cite{VerwerHundsdorfer1990RKC}. We call $\ell_s^\eps$ the real stability boundary of the method.
Let 
\begin{equation}\label{eq:defbetas}
\beta^1(s,\eps)=\left(2-\frac{4}{3}\eps\right)s^2,\qquad\qquad \beta^2(s,\eps)=\frac{2}{3}\left(1-\frac{2}{15}\eps\right)(s^2-1).
\end{equation}
In \cite{VerwerHundsdorfer1990RKC} it is shown that for RKC1 $\ell_s^\eps\geq\beta^1(s,\eps)$ and for RKC2 $\ell_s^\eps\geq\beta^2(s,\eps)$, therefore the stability domain of both methods grows \emph{quadratically}, with respect to the number of function evaluations $s$, along the negative real axis.
Moreover, for $z\in\mathbb{R}_-$, $|z|\leq \beta^p(s,\eps)$, $p=1,2$, is a sufficient condition for stability. Note that the real stability boundary $\ell^\eps_s$ of RKC2 grows slower than for RKC1 (the constant in \cref{eq:defbetas} is smaller).

For more general right-hand sides, as in \cref{eq:ode}, the number of stages $s$ is chosen at each time step so that $\Dt\rho\leq\beta^p(s,\eps)$, where $\rho$ is the spectral radius of the Jacobian of $\bm{f}$ evaluated at $\by^n$. Note that $\rho$ can cheaply be approximated using nonlinear power methods \cite{Lin72,Ver80}.
We note that RKC methods do not have any step size restriction since for any given $\Dt$ it is sufficient to take $s$ large enough to guarantee stability. We also remark that for RKC methods the number of function evaluations is proportional to $\sqrt{\Dt\rho}$, instead of $\Dt\rho$ as for classical explicit RK methods such as, e.g.~RK4 or DOPRI45.

\begin{remark}\label{rem:nonautonomous}
	In this paper we only consider autonomous problems for simplicity and economy of notation. Let us note that the mixed-precision schemes we introduce can straight-forwardly be extended to nonautonomous problems after applying this simple modification:  for nonautonomous problems $\by'=\bf(t,\by)$, we simply replace $\bf(\by^n+\bd_j)$, $j=0,\ldots,s-1$, in \cref{eq:RKC} with $\bf(t^n+c_j\Dt,\ \by^n+\bd_j)$, and $c_j$ as in \cref{eq:defc}.
\end{remark}

\section{Order-preserving mixed-precision RKC methods}
\label{sec:mixed-precision-ESRK}

Let us first define what we mean when we write that a mixed-precision integrator scheme for \cref{eq:ode} is order-preserving. We first introduce our main working assumption; which we consider implicitly to hold true for the remaining of the paper.
\begin{assumption}
	\label{assumption:high_precision_exact}
	Computations performed in high precision are exact.
\end{assumption}
Here by ``high precision'' we indicate the highest precision used in the mixed-precision scheme (typically double or single precision). In a mixed-precision RK scheme, computations performed in low precision produce large roundoff errors and a na\"ive implementation may lead to an order-reduction phenomenon or even stagnation. When this happens, the mixed-precision scheme has a convergence order $q$, where $q$ is smaller than the convergence order $p$ of the original scheme. This motivates the following definition:
\begin{definition}[Order-preserving mixed-precision scheme]
	\label{def:order-preserving}
	Consider a $p$-th order timestepping scheme. A mixed-precision implementation of the same scheme is order-preserving up to order $q\in\N$, $1\leq q\leq p$ (or $q$-order-preserving), if it converges with order $q$ under Assumption \ref{assumption:high_precision_exact}. If a mixed-precision implementation does not converge as $\Dt\rightarrow 0$ (i.e.~the error stagnates or blows up as $\Dt\to 0$), then it is not order-preserving.
\end{definition}

Throughout this section it will be clearer why a mixed-precision scheme implemented na\"ively might not be order-preserving and is thus unable to reduce the error below the machine precision of the low-precision format.
To our knowledge, the methods we present in this paper are the first explicit mixed-precision order-preserving methods to be presented in the literature. We remark that implicit order-preserving methods are instead available \cite{BGG21,Gra20}, although the approach used for these is considerably different.

\subsection{A heuristic introduction to mixed-precision explicit Runge--Kutta schemes}
To set the scene, we start by considering linear problems. Let us first consider a generic $s$-stage order $p$ explicit RK method and take $\bm{f}(\by)=A\by$. We then know that the exact solution to \eqref{eq:ode}  and one step of the numerical scheme in exact arithmetic are respectively given by
\begin{align}
\label{eq:linear_exactapprox}
\bm{y}(t^{n+1}) = \exp(\Dt A)\bm{y}(t^n) = \sum\limits_{j=0}^\infty \frac{(\Dt A)^j}{j!}\bm{y}(t^n), \quad\quad
\bm{y}^{n+1} = R(\Delta t A)\bm{y}^n,
\end{align}
where $R(z)$, a polynomial of degree $s$, is the stability function of the method. The method is then of order $p$ if $|\exp(z)-R(z)| = O(z^{p+1})$, i.e.~if the coefficients of the $p+1$ lowest-degree terms of $R(z)$ match the first $p+1$ terms in the exponential series. The second equation in \eqref{eq:linear_exactapprox} can then be written as
\begin{align}
\label{eq:linear_approx}
\bm{y}^{n+1} = \sum\limits_{j=0}^{p} \frac{(\Delta t A)^j}{j!}\bm{y}^n + \sum\limits_{j=p+1}^{s} (1+a_j) \frac{(\Delta t A)^j}{j!}\bm{y}^n,
\end{align}
where $a_j$ for $j=p+1,\dots,s$ are coefficients which are typical of the method. After setting $\bm{y}^n=\bm{y}(t^n)$ and subtracting \eqref{eq:linear_approx} from the first equation in \eqref{eq:linear_exactapprox} it is then clear that the RK method has a truncation error $O(\Delta t^{p+1})$ and a convergence rate of $O(\Delta t^p)$. However, the argument ceases to be valid when computations are affected by rounding errors. In this case we have something that looks like\footnote{The order in which computations are performed matters little for the sake of our argument here.}
\begin{align}
\label{eq:linear_rounding_errs}
\hat{\bm{y}}^{n+1} = \bm{y}^n + \sum\limits_{j=1}^{p} \frac{\Delta t^j}{j!}\left(\prod_{k=1}^j \hat{A}_k\right)\bm{y}^n + \sum\limits_{j=p+1}^{s} (1+a_j) \frac{(\Delta t)^j}{j!}\left(\prod_{k=1}^j \hat{A}_k\right)\bm{y}^n + \eps_n.
\end{align}
Here $\hat{A}_k = A + \Delta A_k$ for $k=1,\dots,s$, and the $\{\Delta A_k\}_{k=1}^{s}$ terms satisfy $|\Delta A_k|\leq \gamma_{\barm} |A|$ for all $k$, and contain the rounding errors in the matrix-vector products with $A$. The term $\eps_n$ instead is of order $O(u)$ and contains the rounding errors from all other computations (vector multiplication by a scalar and additions). It is now immediately clear that this scheme is not of order $p$ anymore. In fact, it is not even convergent as the local error is $O(u+ \Delta t^{p+1})$ and therefore the global error blows up with a rate that is $O(u\Delta t^{-1} + \Delta t^p)$.
This is a classical result (see e.g.~\cite{Henrici1962,Henrici1963}), but it is often overlooked when working in double precision as $u$ is extremely small and makes the $u\Delta t^{-1}$ term negligible. If computations are performed using lower precisions (fp16, bfloat16, and possibly fp32), however, this term becomes significant and the method stops converging \cite{CrociGilesSR2020}.

A simple mixed-precision approach for RK methods is to perform all expensive matrix-vector products in low precision and all less expensive vector computations (as additions) in high precision. Under Assumption \ref{assumption:high_precision_exact}, the $\eps_n$ term in equation \eqref{eq:linear_rounding_errs} then vanishes and rounding errors stop causing the global error to grow like $O(u\Delta t^{-1})$. Nevertheless, things are still not entirely satisfactory: using the same argument as before we obtain a convergence rate of $O(u + \Delta t^p)$ since
\begin{align}
\Delta t^{-1}||\hat{\bm{y}}^{n+1}-\bm{y}(t^{n+1})||_{2} = ||\Delta A_1 \bm{y}^n ||_{2} + O(u\Delta t + \Delta t^p) = O(u + \Delta t^p).
\end{align}
It is therefore clear that standard mixed-precision RK methods are not order-preserving, i.e.~they are unable to reduce the approximation error below a threshold proportional to the machine precision of the low-precision format used.

The idea of our new mixed-precision RK methods is to instead compute the first $p$ matrix vector products exactly so that $\Delta A_k = 0$ for all $k=1,\dots,p$, and the final convergence rate is
\begin{equation}
\Delta t^{-1}||\hat{\bm{y}}^{n+1}-\bm{y}(t^{n+1})||_{2} = \left\lVert\frac{1+a_{p+1}}{(p+1)!}\Delta A_{p+1} (\Dt A)^p\bm{y}^n\right\rVert_{2} + \left\lVert\frac{a_{p+1}}{(p+1)!}\Delta t^pA^{p+1}\bm{y}_0\right\rVert_{2} + O(\Delta t^{p+1}) = O\left(u\Delta t^p + \Dt^p\right),
\end{equation}
which is the same as for the method in exact arithmetic, albeit with a slightly perturbed constant. This method is $p$-order-preserving. More generally, one might afford to only perform $q$ matrix-vector products in high precision, yielding a $q$-order-preserving method with a final convergence rate of $O(u\Delta t^q + \Delta t^p)$. In this scenario, the method will initially converge at a rate $p$ up until $\Delta t \propto u^{1/(p-q)}$, after which the order will decay to $q$. We note that depending on the problem, accurate enough solutions might be obtainable before this lower-order regime kicks in and choices of $q\ll p$ for high-order methods might be feasible.

\begin{remark}
	Given an $s$-stage order-$p$ RK method and $q\in \{1,\dots,p\}$, it is always possible to construct a $q$-order-preserving mixed-precision equivalent as
	\begin{align}
	\label{eq:q-order-pres-RK}
	\hat{\bm{y}}^{n+1} = \left(\sum\limits_{j=0}^{q} \frac{(\Delta t A)^j}{j!}\hat{\bm{y}}^n\right) + \left(\sum\limits_{j=q+1}^{s} (1+a_j) \frac{(\Delta t)^j}{j!}\left(\prod_{k=q+1}^j \hat{A}_k\right)A^q\hat{\bm{y}}^n\right) = \bm{u}_q + \bm{c}_s,
	\end{align}
	where $a_j = 0$ for $j=q+1,\dots,\min(p+1,s)$. Here $\bm{u}_q$ is the explicit method of order $q$ that matches the first $q+1$ terms in the exponential series and is computed exactly, while $\bm{c}_s$ is a stabilising $O((\Delta t A)^{q+1})$ correction term that is computed in low precision.
\end{remark}

Assuming that matrix-vector products dominate the computations, by using a $q$-order-preserving mixed-precision scheme we would reduce the cost by a factor
\begin{align}
\label{eq:cost_reduction_factor_linear}
\varrho = 1 - \dfrac{(s-q) + qr}{sr} = \dfrac{(s-q)(r-1)}{sr},
\end{align}
where $r$ is the ratio between the costs of performing a matrix-vector product in high and in low precision. For instance, if we choose $q=2$ in the classical RK4 method, and we use a combination of fp64 (double) and fp16, we have $r=4$ for a sparse matrix and $r=16$ for a dense matrix, yielding $\varrho = 37.5\%$ and $\varrho\approx47\%$. Furthermore RK4 will converge with order $4$ up until $\Delta t\propto u^{1/(p-q)}=0.022$, after which it will converge with order $2$. If we instead take $q=1$ and consider a $64$-stages RKC2 method we have $\varrho\approx 74\%$ and $\varrho\approx92\%$ for a sparse and dense matrix respectively, and that the scheme will retain its second order until $\Delta t\propto u^{1/(p-q)}=2^{-11}$. We remark that these are only rough estimates and that in practice $r$ might be larger if computations are memory-bound.

There are three complications to the idea presented in this section: 1) In most traditional RK methods we have that $s$ is not much larger than $p$ and, for instance if $s=p$, the whole scheme must be run in high precision to retain the full order. 2) The presence of nonlinearities disrupts the argument we just presented and the mixed-precision scheme must be constructed more carefully. 3) Performing some computations in low precision might disrupt the numerical stability of the method.

For ESRK, point 1) is not problematic because usually $s\gg p$.
For all other explicit methods we simply advocate that using $q < p$ might still bring some computational advantage, especially for large $\Delta t$. As far as points 2) and 3) are concerned, in the remaining of this section and in Section \ref{sec:multirate} we explain how to implement the mixed-precision schemes so as to deal with nonlinearities and we derive under which conditions these schemes are still numerically stable. However, there are some limitations:
\begin{remark}
	We are currently unable to develop efficient mixed-precision ESRK schemes based on three-term recurrence relations that are more than second-order preserving.
\end{remark}

\subsection{Mixed-precision RKC schemes for nonlinear problems}\label{sec:mpRKC}
In order to construct a mixed-precision version of method \cref{eq:RKC} that is order-preserving up to order $q=p$ we must ensure that all $p+1$ lowest-order terms are computed exactly (i.e.~in high precision). The resulting mixed-precision methods therefore vary according to the value of $p$.
In this section we consider problem \cref{eq:ode}, problem \cref{eq:mrode} is considered in \cref{sec:multirate}.

\color{myred}
Let $\hby^n$ be an approximation to $\by(t^n)$ computed with the mixed-precision scheme and $s\in\mathbb{N}$ such that \moda{$\Dt\rho\leq\beta^p(s,\eps)$}, where $\rho$ is the spectral radius of the Jacobian of $\bm{f}$ evaluated in $\hby^n$ and $\beta^p(s,\eps)$ is given in \cref{eq:defbetas}. One step of the 1-order-preserving mixed-precision RKC scheme is given by:
\begin{align}
\label{eq:first-order-scheme}
\begin{dcases}
\hbd_0 = \bm{0},\quad \hbd_1 = \mu_1\Dt \bm{f}(\hby^n),\\
\hbd_j = \nu_j \hbd_{j-1} + \kappa_j \hbd_{j-2} +\mu_j\Dt(\bm{f}(\hby^n)+\hat\Delta\bm{f}_{j-1}) +\gamma_j\Dt \bm{f}(\hby^n) ,\quad j=2,\dots,s,\\
\hby^{n+1} = \hby^n + \hbd_s,
\end{dcases}
\end{align}
where the $\{\hat\Delta\bm{f}_{j}\}_{j=1}^{s-1}$ are quantities evaluated in low precision satisfying $\hat\Delta\bm{f}_{j}\approx \Delta \bm{f}_j=\bm{f}(\hby^n+\hbd_j)-\bm{f}(\hby^n)$. Note that if in \cref{eq:first-order-scheme} we replace $\hat\Delta\bm{f}_j$ with $\Delta\bm{f}_j$ we obtain the original RKC scheme \cref{eq:RKC}. However, evaluating this difference in high precision is expensive, and for this reason we instead compute an approximation in low precision. The accuracy of this approximation together with the choice of coefficients will set the effective convergence order of the mixed-precision RKC scheme \cref{eq:first-order-scheme}. Method \cref{eq:first-order-scheme} is reminiscent of the linearized RKC method of \cite{Ver82}, where $\bm{f}(\by^n+\bd_j)$ is replaced with $\bm{f}(\by^n)+\bm{f}'(\by^n)\bd_j$. Indeed, in what follows $\hat\Delta\bm{f}_{j}\approx \bm{f}'(\by^n)\bd_j$, which is crucial to maintain stability.
\color{black}

\paragraph{\moda{First-order-preserving RKC schemes}}
\mbox{}\\
\moda{We consider here a 1-order preserving scheme ($q=1$) for the first- and second-order version of \cref{eq:first-order-scheme}, hence with coefficients given by \cref{eq:coeffRKC,eq:coeffRKC1} for $p=1$ or \cref{eq:coeffRKC,eq:coeffRKC2} for $p=2$.}
Our 1-order-preserving mixed-precision RKC method needs only one high-precision evaluation of the right-hand side $\bm{f}$, to preserve accuracy. The remaining $s-1$ evaluations are for stability and can be performed in low precision. 

The 1-order-preserving method is given by \cref{eq:first-order-scheme} with $\{\hat\Delta\bm{f}_{j}\}_{j=1}^{s-1}$ satisfying as $\Dt\rightarrow 0$,
\begin{equation}
\label{eq:requirementRKC1}
\hat\Delta\bm{f}_{j} = \Delta\bm{f}_j + O(\epsilon\Dt),\quad\forall j,
\end{equation}
where $\epsilon$ is a small positive constant. The $O(\Dt)$ accuracy of the approximation is the central ingredient that is required to obtain a 1-order-preserving method, while the small constant $\epsilon$ ensures that $\hat\Delta\bm{f}_{j}$ is close to $\Delta\bm{f}_j=O(\Dt)$, which in turn is an $O(\Dt^2)$ approximation of $\bm{f}'(\hby^n)\hbd_{j}$ and brings stability.
The challenge here is that a na\"ive low-precision evaluation of $\hat\Delta\bm{f}_j$ leads to rounding errors that in general are not $O(\Dt)$, but only $O(u)$, thus impacting the limiting accuracy of the scheme, and for this reason the $\hat\Delta\bm{f}_{j}$ terms must be carefully implemented. There are multiple ways of computing $\hat\Delta\bm{f}_{j}$ so as to satisfy \eqref{eq:requirementRKC1}, but we refer to \cref{sec:evalhDfj} for a discussion on the available options. \textcolor{myred}{The impact of \cref{eq:requirementRKC1} on the accuracy and stability of the method is studied \cref{sec:mpRKCanalysis}. We will show that the constants in $O(\epsilon\Dt)$ play an important role and must be relatively small to preserve the internal stability of the method.}

\color{myred}
Summarizing, the 1-order-preserving RKC method is given by
\begin{equation}\label{eq:1-order-preserving-scheme}
\text{Method }\cref{eq:first-order-scheme} \text{ with coefficients } \cref{eq:coeffRKC},\,\cref{eq:coeffRKC1} \text{ or } \cref{eq:coeffRKC},\,\cref{eq:coeffRKC2} \text{ and }\hat\Delta\bm{f}_j \text{ as in } \cref{eq:requirementRKC1}.
\end{equation}

\color{black}
\begin{remark}
	A na\"ive approach for designing a mixed-precision RKC scheme could be to consider scheme \cref{eq:RKC} and perform all $\bf$ evaluations in low precision except in the first stage, where high precision is employed. However, this technique would lead to stagnation. Indeed, in a Taylor expansion of $\by^{n+1}$ in \cref{eq:RKC}, with respect to $\Dt$, not only $\bf(\by^n)$ but also $\bf(\by^n+\bd_j)$, $j=1,\ldots,s-1$, appear in the first-order $O(\Dt)$ term.
\end{remark}

\begin{remark}
	\label{rem:1-order-pres-high-order-methods}
	\modb{We presented the mixed-precision RKC scheme \eqref{eq:1-order-preserving-scheme} as an example of how to construct a 1-order-preserving RKC1 or RKC2 method. The same strategy (use the delta form, replace $\bm{f}(\by^n+\bd_j)$ with $\bm{f}(\hby^n)+\hat\Delta\bm{f}_j$ and require \cref{eq:requirementRKC1}) can be straight-forwardly employed to construct the 1-order-preserving mixed-precision version of other higher-order schemes such as the ROCK methods \cite{Abd02,AbM01}.}
\end{remark}

\begin{remark}\label{rem:skrock}
	\modc{Scheme \cref{eq:1-order-preserving-scheme} with first-order coefficients \cref{eq:coeffRKC,eq:coeffRKC1} can be straightforwardly extended to stochastic differential equations (SDEs). It suffices to apply the same reasoning on the SK-ROCK scheme \cite{AAV18}, which is the natural extension of RKC1 to SDEs. Indeed, writing the SK-ROCK method in delta form yields
		\begin{equation}\label{eq:skrock}
		\begin{aligned}
		\bm{D}_0 &=\bm{0},\quad \bm{D}_1=\mu_1\Dt\bm{f}(\bm{X}_n+\nu_1 \bm{Q})+\kappa_1 \bm{Q},\\
		\bm{D}_j &= \nu_j\bm{D}_{j-1}+\kappa_j\bm{D}_{j-2}+\mu_j\Dt\bm{f}(\bm{X}_n+\bm{D}_{j-1}),\quad j=2,\ldots,s\\
		\bm{X}_{n+1} &= \bm{X}_n+\bm{D}_s,
		\end{aligned}
		\end{equation}
		where $\bm{Q}$ contains the diffusion terms. A mixed-precision version of \cref{eq:skrock} is obtained by replacing $\bm{f}(\bm{X}_n+\nu_1 \bm{Q})$ with $\bm{f}(\hat{\bm{X}}_n)+\hat\Delta\bm{f}_0$ and $\bm{f}(\bm{X}_n+\bm{D}_{j})$ with $\bm{f}(\hat{\bm{X}}_n)+\hat\Delta\bm{f}_j$. Accuracy is preserved if $\hat\Delta\bm{f}_{j} = \Delta\bm{f}_j + O(\epsilon\Dt)$ for $j=0,\ldots,s$, where $\Delta\bm{f}_0=\bm{f}(\hat{\bm{X}}_n+\nu_1 \bm{Q})-\bm{f}(\hat{\bm{X}}_n)$ and $\Delta\bm{f}_j=\bm{f}(\hat{\bm{X}}_n+\bm{D}_j)-\bm{f}(\hat{\bm{X}}_n)$ for $j=1,\ldots,s$. For the evaluation of $\hat\Delta\bm{f}_j$ the techniques of \cref{sec:evalhDfj} can be employed. Note that $\bm{Q}$ must be evaluated in high precision, however this is done only once as in the standard SK-ROCK method.
	}
\end{remark}

\paragraph{Second-order-preserving RKC2 scheme}
\mbox{}\\
\moda{Condition \cref{eq:requirementRKC1} is enough to obtain a 1-order-preserving method. In fact, method \cref{eq:first-order-scheme} under condition \cref{eq:requirementRKC1} has a local error of $O(\epsilon\Dt^2)$, which leads to an $O(\epsilon\Dt)$ global error. However, condition \cref{eq:requirementRKC1} on its own does not ensure second-order convergence. In order to obtain a 2-order-preserving scheme, we still employ method \cref{eq:first-order-scheme}, but we now require instead that, as $\Dt\rightarrow 0$,}
\begin{align}\label{eq:requirementRKC2}
\hat\Delta\bm{f}_{j} = \Delta\bm{f}_j + O(\Dt^2),\quad\forall j.
\end{align}
Again, multiple strategies for the evaluation of $\hat\Delta\bm{f}_j$ satisfying \cref{eq:requirementRKC2} are feasible here and we will describe them in \cref{sec:evalhDfj}. \color{myred}While condition \cref{eq:requirementRKC2} is enough to ensure second-order convergence, the size of the constants in \cref{eq:requirementRKC2} must also remain small in order to preserve the internal stability (stability within one step) of the method. Unfortunately, we could not derive second-order approximations \cref{eq:requirementRKC2} with small constants outside a convergence regime. For this reason, we instead only require condition \cref{eq:requirementRKC2} to be satisfied in a convergence regime, and use \cref{eq:requirementRKC1} otherwise. We thus propose the following hybrid scheme which detects convergence and switches from \cref{eq:requirementRKC1} to \cref{eq:requirementRKC2} and vice-versa as needed:
\begin{equation}\label{eq:hybrid-scheme}
\text{Method }\cref{eq:first-order-scheme} \text{ with coefficients } \cref{eq:coeffRKC},\,\cref{eq:coeffRKC2} \text{ and }\hat\Delta\bm{f}_j \text{ as in }
\begin{cases}
\cref{eq:requirementRKC2} & \quad\text{if}\quad \Vert \hbd_j-c_j\Dt\bm{f}(\hby^n)\Vert_2\leq\Vert\hbd_j\Vert_2,\\
\cref{eq:requirementRKC1} & \quad\text{otherwise.}
\end{cases}
\end{equation}
Under condition \cref{eq:requirementRKC1}, as $\Dt\to 0$ we have $\hbd_j=c_j\Dt\bm{f}(\hby^n)+O(\Dt^2)$, and hence $\Vert \hbd_j-c_j\Dt\bm{f}(\hby^n)\Vert_2\leq\Vert\hbd_j\Vert_2$, and we can safely switch to the stricter condition \cref{eq:requirementRKC2} which gives second-order convergence. For $\Dt$ large $\hbd_j$ is small due to stability, while $c_j\Dt\bm{f}(\hby^n)$ is large due to stiffness. Therefore, $\Vert \hbd_j-c_j\Dt\bm{f}(\hby^n)\Vert_2\approx \Vert c_j\Dt\bm{f}(\hby^n)\Vert_2$ and $\Vert \hbd_j-c_j\Dt\bm{f}(\hby^n)\Vert_2\leq\Vert\hbd_j\Vert_2$ is violated for large $\Dt$. Hence scheme \cref{eq:hybrid-scheme} chooses \cref{eq:requirementRKC2} only in a convergence regime and preserves internal stability by switching to \cref{eq:requirementRKC1} for large $\Dt$.
\color{black}

\begin{remark}
	\modb{With the same strategy as in \cref{rem:1-order-pres-high-order-methods} a mixed-precision 2-order-preserving version of ROCK4 can easily be obtained from \cref{eq:hybrid-scheme}.}
\end{remark}

\subsection{Evaluation of the \texorpdfstring{$\hat\Delta \bm{f}_j$}{Jacobian} terms}\label{sec:evalhDfj}
\moda{Our order-preserving RKC schemes require} the $\hat\Delta\bm{f}_j$ terms to be evaluated in reduced precision at the given accuracy \cref{eq:requirementRKC1} for $q=1$ or \cref{eq:requirementRKC2} for $q=2$. We now explain how this can be done in practice under different scenarios. For this purpose, suppose that $\bf$ has the form
\begin{equation}
\bf(\by)=A\by+\bg(\by),
\end{equation}
with $A\in\R^{n\times n}$, and $\bg:\R^n\rightarrow\R^n$ be any function with the same smoothness as $\bf$. Obviously, one can put $A=0$ and let $\bg$ absorb all linear terms. However, factoring out the linear terms helps in reducing rounding errors, and $A$ here can also be considered as the Jacobian of some other terms within $\bf$ for which the derivative is more readily available.
In what follows we write $\bf (\by)$, $\bm{g} (\by)$ and $\hbf (\by)$, $\hat{\bm{g}} (\by)$ to indicate function evaluations in high or low precision, respectively

\textbf{Scenario 1: the nonlinear term is much cheaper to evaluate than the linear term.} \moda{In this case, for \cref{eq:requirementRKC1} }it is simply possible to implement $\hat\Delta \bm{f}_j$ using mixed precision as
\begin{align}
\label{eq:evalDeltaf_mixed}
\hat\Delta \bm{f}_j = \hat{A}\hbd_{j} + \bm{g}(\hby^n + \hbd_{j}) - \bm{g}(\hby^n) = \Delta\bm{f}_j + O(u\Dt),
\end{align}
since $\hbd_{j}=O(\Dt)$ and the low-precision multiplication by $\hat{A}$ yields an $O(u\Dt)$ error (i.e.~$\varepsilon=u$). This strategy requires applying $A$ in high precision only to $\hby^n$ at the first stage of \cref{eq:first-order-scheme}. If possible, these high-precision matrix-vector products (matvecs) could even be performed matrix-free to avoid storing $A$ in high precision. \moda{When \cref{eq:requirementRKC2} is required we} instead compute $\hat\Delta \bm{f}_j$ as
\begin{align}\label{eq:evalDeltaf_mixed2}
\hat\Delta \bm{f}_j = \hat{A}\hbv_j + c_j\Dt A\bm{f}(\hby^n) + \bm{g}(\hby^n + \hbd_j) - \bm{g}(\hby^n) = \Delta\bm{f}_j + O(u\Dt^2),
\end{align}
\moda{where $\hbv_j=\hbd_j-c_j\Dt\bm{f}(\hby^n)$. The second equality follows from $\hbv_j=O(\Dt^2)$ (see \cref{sec:RKCmethods} or \cref{thm:convmpRKCbis}), hence multiplication by $\hat{A}$ yields an $O(u\Dt^2)$ error. This strategy only requires evaluating $\bm{f}(\hby^n)$ and $A\bm{f}(\hby^n)$ once in high precision every $s$ stages.}

\color{myred}
\begin{remark}
	While a more rigorous analysis is provided in \cref{sec:internalstab}, we now briefly comment on the behaviour of the errors in \cref{eq:evalDeltaf_mixed,eq:evalDeltaf_mixed2} outside a convergence regime for large $\Dt$. For simplicity, we assume $\bm{g}\equiv 0$ and that $\hbd_j=\bd_j=(R_j(\Dt A)-I)\hby^n$, giving $\Vert\hbd_j\Vert_2\leq 2\Vert\hby^n\Vert_2$ (the exact scheme is stable, hence $\Vert R_j(\Dt A)\Vert_2 \leq 1$). 
	Using \cref{lemma:boundDA} and $\hat{A} = A + \Delta A$, the rounding errors in \cref{eq:evalDeltaf_mixed} are bounded by
	\begin{equation}\label{eq:errorO1}
	||\Delta A\hbd_j||_2 \leq \barc\barm^2u||A||_2\Vert\hbd_j\Vert_2\leq 2\barc\barm^2u||A||_2 \Vert\hby^n\Vert_2.
	\end{equation}
	In \cref{eq:evalDeltaf_mixed2} we have $\Vert\hbv_j\Vert_2\leq\Vert\hbd_j\Vert_2+c_j\Dt\Vert A\Vert_2\Vert\hby^n\Vert_2\leq (2+c_j\Dt\Vert A\Vert_2)\Vert\hby^n\Vert_2$ and the rounding errors are instead bounded by
	\begin{equation}\label{eq:errorO2}
	||\Delta A\hbv_j||_2 \leq \barc\barm^2u||A||_2\Vert\hbv_j\Vert_2\leq \barc\barm^2u||A||_2(2+c_j\Dt\Vert A\Vert_2)\Vert\hby^n\Vert_2,
	\end{equation}
	which for large $\Dt$ is much larger than \cref{eq:errorO1} due to the term $\Dt\Vert A\Vert_2=\Dt\rho$ which is approximately $\beta^p(s,\eps)=O(s^2)$. Hence, outside of a convergence regime the error in \cref{eq:evalDeltaf_mixed2} might\footnote{\color{myred}While we are comparing upper bounds rather than lower bounds, the upper bounds are close to what we observe in practice. Note that in rounding error analysis the lower bound for the rounding error is always zero, corresponding to the possible, yet unrealistic scenario in which all computations are performed exactly.} be larger than in \cref{eq:evalDeltaf_mixed} and affect stability. To avoid stability issues but still preserve second-order the hybrid scheme \cref{eq:hybrid-scheme} should be employed. Indeed, in \cref{eq:hybrid-scheme} approximation \cref{eq:evalDeltaf_mixed2} is employed only when $\Vert\hbv_j\Vert_2$ is smaller than $\Vert\hbd_j\Vert_2$.
\end{remark}
\color{black}

\textbf{Scenario 2: both linear and nonlinear terms are expensive to evaluate.} In this case we need to implement the whole $\hat\Delta \bm{f}_j$ in low precision while still ensuring the right order of accuracy and stability. For this purpose, we employ Jacobian approximations. \moda{In \cref{eq:requirementRKC1} we have that}
\begin{align}
\Delta\bm{f}_j = \bm{f}'(\hby^n)\hbd_j + O(\Dt^2),
\end{align}
and therefore we can implement $\hat\Delta\bm{f}_j$ by approximating the action of the Jacobian $\bm{f}'(\hby^n)$ against $\hbd_j$ in low precision, since $\hbd_j=O(\Dt)$ and this leads to an $O(u\Dt)$ rounding error (again $\varepsilon=u$). If an analytic expression for the directional derivative of $\bm{f}$ is available, then the easiest option is to just evaluate the action of the derivative in low precision. When the Jacobian is not known analytically, there are various techniques available to compute the action of a Jacobian against a vector efficiently. However, we do not describe these techniques in detail here, and we only mention two. The first is automatic differentiation \cite{griewank2008evaluating}, through which we can compute the action of $\bm{f}'$ at up to roughly the same cost of a couple of evaluations of $\bm{f}$ itself (see Chapter 4 in \cite{griewank2008evaluating}).The second simply entails computing
\begin{align}
\label{eq:defDf2}
\hat\Delta\bm{f}_j \coloneqq \hat A \hbd_j+ \delta^{-1}\left(\hat{\bg}\left(\hby^n+\delta\,\hbd_j\right)-\bg\left(\hby^n\right)\right) ,\qquad\delta = \frac{\sqrt{u}}{\Dt},
\end{align}
which yields
\begin{align}
\hat\Delta\bm{f}_j =\hat A \hbd_j+\bm{g}'(\hby^n)\hbd_j+O(\sqrt{u}\Dt)
=\Delta \bm{f}_j+O(\sqrt{u}\Dt+\Dt^2).
\end{align}
Computation of $\hat\Delta\bm{f}_j$ only requires low-precision evaluations of $\bg$, as $\bg(\hby^n)$ is already known from $\bm{f}(\hby^n)$, which is needed during the first stage of \cref{eq:first-order-scheme}. Estimate \cref{eq:defDf2} yields an accurate enough approximation with $\varepsilon=\sqrt{u}$ and is proved in \cref{lemma:jacapp}. The introduction of the coefficient $\delta$ is crucial to guarantee a good Jacobian approximation. Indeed, the roundoff introduced by $\hat \bg$ is $O(u)$ and a multiplication by $\delta^{-1}$ ensures that $\hat\Delta\bm{f}_j=\bm{f}'(\hby^n)\hbd_j+O(\sqrt{u}\Dt)$. We refer the reader to \cref{lemma:jacapp} for details.

\color{myred} %if you remove this then remove also \color{black} after the red text block
In order to satisfy the second-order condition \cref{eq:requirementRKC2} we use $\hbv_j=\hbd_j-c_j\Delta t\bm{f}(\hby^n)$ and we rewrite $\hat\Delta\bm{f}_j$ as
\begin{align}
\hat\Delta\bm{f}_j = \hat A \hbv_j+c_j\Delta t A\bm{f}(\hby^n)+\hat\Delta_1\bm{g}_j + \hat\Delta_2\bm{g}_j,
\end{align}
with $\Delta_1\bm{g}_j$ and $\hat\Delta_2\bm{g}_j$ required to satisfy
\begin{align}
\hat\Delta_1\bm{g}_j =\bm{g}(\hby^n + \hbd_j) - \bm{g}(\hby^n+c_j\Delta t f(\hby^n))+ O(\Dt^2),\quad\quad \hat\Delta_2\bm{g}_j =  \bm{g}(\hby^n+c_j\Delta t f(\hby^n))- \bm{g}(\hby^n) + O(\Dt^2),
\end{align}
and thus ensuring that $\hat\Delta_1\bm{g}_j + \hat\Delta_2\bm{g}_j =  \bm{g}(\hby^n + \hbd_j) - \bm{g}(\hby^n)+O(\Delta t^2)$.
Here, $\hat\Delta_1\bm{g}_{j}$ and $\hat\Delta_2\bm{g}_{j}$ can again be obtained via Jacobian approximation/evaluation:
\begin{align}\label{eq:condDeltag}
\hat\Delta_1\bm{g}_{j} = \hat{\bm{g}}'(\hby^n + c_j\Dt\bm{f}(\hby^n))\hbv_j,\qquad
\hat\Delta_2\bm{g}_{j} = c_j\Dt\bm{g}'(\hby^n)\bm{f}(\hby^n),
\end{align}
for which we can use the same techniques mentioned before. For instance, $\hat\Delta_1\bm{g}_{j}$ can be computed analogously to \cref{eq:defDf2}:
\begin{align}\label{eq:exDeltag1}
\hat\Delta_1\bm{g}_{j}\coloneqq \delta^{-1}\left( \hat{\bm{g}}(\hby^n+c_j\Delta t\bm{f}(\hby^n)+\delta\hbv^j)-\hat{\bm{g}}(\hby^n+c_j\Delta t\bm{f}(\hby^n))\right),\qquad \delta = \frac{\sqrt{u}}{\Dt^2}.
\end{align}
and thus
\begin{align}
\hat\Delta_1\bm{g}_{j}= \bm{g}'(\hby^n+c_j\Delta t\bm{f}(\hby^n))\hbv^j+O(\sqrt{u}\Delta t^2)
=\bm{g}(\hby^n + \hbd_j) - \bm{g}(\hby^n+c_j\Delta t f(\hby^n))+ O(\sqrt{u}\Dt^2+\Delta t^4).
\end{align}
Note that the hat in the Jacobian appearing in the expression for $\hat\Delta_1\bm{g}_{j}$ in \cref{eq:condDeltag} indicates that the Jacobian approximation can be performed in low precision as in \cref{eq:exDeltag1}. In contrast, the expression for $\hat\Delta_2\bm{g}_{j}$ requires the Jacobian to be evaluated in high precision. This different choice is crucial to ensure that the overall error in the approximation is $O(\Dt^2)$. We remark that for the $\hat\Delta_2\bm{g}_{j}$ term a single high-precision evaluation of $\bm{g}'(\hby^n)\bm{f}(\hby^n)$ every $s$ stages is sufficient since the only thing that varies with $j$ is $c_j$.
\color{black}
\\

\textbf{Other scenarios.}
\begin{itemize}
	\item If the linear term is cheaper to evaluate than the nonlinear term the solution is to simply apply the strategy for Scenario 2 and evaluate the matrix-vector products in high precision.
	\item In some cases it is possible to implement differences like $\bm{g}(\bm{y} + \bm{b})-\bm{g}(\bm{y})$ in such a way that the rounding errors are automatically of the right order of accuracy. An example scenario is when there is an analytical expression for the difference of the right order.
\end{itemize}

\begin{example}
	Take the nonlinear convective term of the Navier--Stokes equations, $\bm{g}(\bm{y})=\bm{y}\nabla\bm{y}$. We then have that
	\begin{align}
	\label{eq:_exNS_1}
	\bm{g}(\bm{y} + \bm{b}) - \bm{g}(\bm{y}) = \bm{b}\nabla \bm{y} + \bm{y}\nabla\bm{b} + \bm{b}\nabla\bm{b}.
	\end{align}
	For the first-order methods we then take $\bm{y}=\hby_n$, $\bm{b}=\hbd_j=O(\Dt)$, and \eqref{eq:_exNS_1} evaluated in low precision yields an $O(u\Dt)$ error. For the second-order methods we instead take $\bm{y}=\hby_n + c_j\Dt\bm{f}(\hby^n)$ and $\bm{b}=\hbv_j=O(\Dt^2)$ for the $\hat\Delta_1\bm{f}_j$ term, yielding an $O(u\Dt^2)$ rounding error if \eqref{eq:_exNS_1} is evaluated in low precision. For the $\hat\Delta_2\bm{f}_j$ term we instead set $\bm{y}=\hby^n$ and $\bm{b}=c_j\Dt\bm{f}(\hby^n)$ yielding $c_j\Dt\left(\bm{f}(\hby^n)\nabla \hby^n + \hby^n\nabla\bm{f}(\hby^n)\right) + c_j^2\Dt^2\bm{f}(\hby^n)\nabla\bm{f}(\hby^n)$, where each term is constant across the stages (except for the scalings by $c_j$ and $c_j^2$) and they can be pre-computed once every RKC time step in high precision (i.e.~no error). In practice, the $\bm{b}\nabla\bm{b}$ term in \eqref{eq:_exNS_1} is of higher order and can possibly be dropped to save on computations. Note that by dropping higher-order terms we recover the directional derivative of $\bm{g}$. \modb{We remark that the Navier--Stokes equations are index-2 differential algebraic equations, and, as such, they can be solved by simply replacing the RKC schemes of the projection methods from \cite{Ros14,ZhP06} with their order-preserving mixed-precision counterparts presented here.}
\end{example}

\subsection{Cost analysis}\label{sec:costanalysis}
Before analysing the convergence and stability properties of algorithm \cref{eq:first-order-scheme} we first derive an expression for the computational savings resulting from our mixed-precision methods with respect to a method fully implemented in high precision. We assume that vector operations are negligible\footnote{Note that method \cref{eq:first-order-scheme} can be implemented with the same number of vector operations as their high-precision equivalents.}, and we define $r$ to be the ratio between the cost of evaluating $\bm{f}$ in high precision and the cost of computing $\hat\Delta\bm{f}_j$ with one of the strategies we just presented (the latter not including the cost of the quantities computed once every $s$-stages). We then have a cost reduction factor $\varrho$ of
\begin{align}
\label{eq:cost_reduction_factor_nonlinear}
\varrho = 1 - \dfrac{(s-q) + qr}{sr} = \dfrac{(s-q)(r-1)}{sr},\quad q \in \{1,2\}.
\\
\varrho \longrightarrow 1-\frac{1}{r},\quad\text{as}\quad s\rightarrow\infty,\qquad\text{and}\qquad \varrho \longrightarrow 1 - \frac{q}{s},\quad\text{as}\quad r\rightarrow\infty.
\label{eq:varrho_limits}
\end{align}
Note that this expression for $\varrho$ is essentially the same as in \eqref{eq:cost_reduction_factor_linear}. By looking at the limit cases for $s,r\rightarrow \infty$, we see that the best cost reduction factor we can hope for when $s$ is large (typical in ESRK methods) is $1-r^{-1}$, which for $r$ also large becomes very close to $1$. The actual value of $r$ grows as the number of bits of the low-precision format chosen decreases and is problem-dependent.

We can present a couple of examples for Scenario 2 under some simplifying assumptions: 1) We only look at flop counts and we ignore savings related to memory efficiency. 2) One flop in a format using twice or four times the number of bits costs twice or four times as much. 3) The cost of evaluating $\hat\Delta\bm{f}_j$ is roughly the same as that of evaluating $\hbf$ in the same precision\footnote{This assumption holds for evaluations in the style of \eqref{eq:defDf2}. When Jacobian matrix-vector products are instead computed via a forward pass of automatic differentiation, computing $\hat\Delta\bm{f}_j$ costs up to $2.5$ times as $\hbf$, even though in practice it might be cheaper, cf.~Chapter 4 in \cite{griewank2008evaluating}.}.
For instance, for sparse $A$ and linear-cost evaluations of the nonlinear term (e.g.~this is the case for the heat equation with a nonlinear reaction term acting entrywise on the solution) we obtain $r=2$ for double-single or single-half combinations, and $r=4$ for double-half, reducing the overall cost by half or a factor of $4$ respectively. For dense $A$ and/or quadratic-cost evaluations of $\bm{g}$ we instead get up to $r=4$ (double-single or single-half) and $r=16$ (double-half) leading to much greater savings. 

In memory-bound computations $r$ might actually be larger since most function evaluations are performed in a low-precision format, which might allow for better cache exploitation. We remark that our $q$-order-preserving schemes require the storage of $q$ additional vectors. However, they might also allow to avoid storing some of the data needed to evaluate $\bm{f}$ in high precision since high-precision evaluations of $\bm{f}$ occur less often. For instance, the matrix $A$ could be implemented matrix-free in high precision and only explicitly stored in low precision.

\subsection{Convergence and stability analysis}\label{sec:mpRKCanalysis}
In this section, we present the accuracy and stability analysis for the \moda{mixed-precision schemes \cref{eq:1-order-preserving-scheme} and \cref{eq:hybrid-scheme}} introduced in \cref{sec:mpRKC}. 
We start by introducing \cref{lemma:perturbations,lemma:coroll_perturbations} below which collects some results that are crucial for the modeling of stage perturbations as truncation or rounding errors. \moda{Then, in \cref{thm:convmpRKCbis}, we show that conditions \cref{eq:requirementRKC1,eq:requirementRKC2} are indeed sufficient to obtain 1- and 2-order preserving schemes. In \cref{thm:accuracympRKC,thm:accuracympRKC2orderpres} we study error propagation within each timestep under \cref{ass:ddf} on the internal stability of the methods. In particular we show the benefits of $\hat\Delta \bm{f}_j$ approximating the Jacobian. Finally, in \cref{thm:stabmpRKC} we study the internal stability of the methods and hence the validity of \cref{ass:ddf}.}

Henceforth, $U_j(x)$ is the Chebyshev polynomial of the second kind of degree $j$, defined recursively by
\begin{equation}
U_0(x)=1,\qquad U_1(x)=2x,\qquad U_j(x)=2xU_{j-1}(x)-U_{j-2}(x),\quad j\geq 2.
\end{equation}

\color{myred} % if you remove this then remove \color{black} below as well
\begin{lemma}\label{lemma:perturbations}
	Let $p=1,2$, $A\in\R^{n\times n}$ and $\mu_j,\nu_j,\kappa_j$ be as in \cref{eq:coeffRKC}, with $\oz,\ou,b_j$ as in \cref{eq:coeffRKC1} if $p=1$ and as in \cref{eq:coeffRKC2} if $p=2$. Let $\bm{r}_j\in\R^n$, $j=1,\ldots,s$, and
	\begin{equation}\label{eq:rk_perturbations}
	\bd_0 = \bm{0}, \qquad \bd_1=\bm{r}_1,\qquad
	\bd_j = \nu_j\bd_{j-1}+\kappa_j\bd_{j-2}+\mu_j\, \Dt A\bd_{j-1}+\bm{r}_j,\quad j=2,\ldots,s.
	\end{equation}
	Then:
	\begin{enumerate}[label=\roman*)]
		\item\label{item:perturbations} Let $I\in\R^{n\times n}$ be the identity matrix. It holds
		\begin{equation}\label{eq:sum_perturbations}
		\bd_k=\sum_{j=1}^k\frac{b_k}{b_j}U_{k-j}(\oz I+\ou \Dt A)\bm{r}_j,\qquad k=1,\ldots,s.
		\end{equation}
		\item\label{item:identity} Recall that $R_k(z)=a_k+b_k T_k(\oz+\ou z)$ is the internal stability polynomial of the RKC scheme. It holds
		\begin{align}\label{eq:defBarR}
		\overline{R}_k(z) &\coloneqq \frac{R_k(z)-1}{z} =\sum_{j=1}^k\frac{b_k}{b_j}U_{k-j}(\oz+\ou z)(\mu_j+\gamma_j), \\  \label{eq:defBarBarR}
		\widetilde{R}_k(z)&\coloneqq \frac{R_k(z)-1-c_k z}{z^2}=\sum_{j=2}^k\frac{b_k}{b_j}U_{k-j}(\oz+\ou z)\mu_j c_{j-1},
		\end{align}
		with $\gamma_1=0$. Note that $R_k(z)=1+c_k z+O(z^2)$ \cite{VerwerHundsdorfer1990RKC} and thus $\overline{R}_k(z)$, $\widetilde{R}_k(z)$ are polynomials as well.
		\item\label{item:eqr} Let $\bm{r}\in\R^n$ and $\bm{r}_j=(\mu_j+\gamma_j)\bm{r}$ for $j=1,\ldots,s$, with $\gamma_1=0$. Then $\bd_k=\overline{R}_k(\Delta t A)\bm{r}$.
	\end{enumerate}
\end{lemma}
\begin{proof}
	Point i) has been proved in \cite{VerwerHundsdorfer1990RKC} and point iii) follows by combining i) and \cref{eq:defBarR}, so we only need to prove ii). To prove \cref{eq:defBarR} we note that
	\begin{equation}\label{eq:rkc_lin_eq}
	\bd_0 = \bm{0}, \qquad \bd_1=\mu_1 \Dt A\by^n,\qquad
	\bd_j = \nu_j\bd_{j-1}+\kappa_j\bd_{j-2}+\mu_j\, \Dt A(\by^n+\bd_{j-1})+\gamma_j\Dt A\by^n,\quad j=2,\ldots,s,
	\end{equation}
	is one step of \cref{eq:RKC} with linear $\bm{f}(\by)=A\by$, hence
	\begin{equation}\label{eq:byplusdk}
	\by^n+\bd_k=R_k(\Dt A)\by^n.
	\end{equation}
	However, applying \cref{eq:sum_perturbations} to \cref{eq:rkc_lin_eq} with $\bm{r}_j=(\mu_j+\gamma_j)\Delta t A \by^n$ yields
	\begin{equation}\label{eq:dk_sumlin}
	\bd_k=\sum_{j=1}^k\frac{b_k}{b_j}U_{k-j}(\oz I+\ou \Dt A)(\mu_j+\gamma_j)\Delta t A\by^n,\qquad k=1,\ldots,s.
	\end{equation}
	Pulling \cref{eq:byplusdk,eq:dk_sumlin} together we obtain \cref{eq:defBarR} with $z=\Dt A$. To prove \cref{eq:defBarBarR} we first subtract $c_j\Delta t A\by^n$ from \cref{eq:rkc_lin_eq}. Then, by using \cref{eq:defc} and by setting $\bv_j=\bd_j-c_j\Delta t A \by^n$, we obtain
	\begin{equation}\label{eq:rkc_lin_eq_v}
	\bv_0 = \bm{0}, \qquad \bv_1=\bm{0},\qquad
	\bv_j = \nu_j\bv_{j-1}+\kappa_j\bv_{j-2}+\mu_j\, \Dt A(c_{j-1}\Delta t A\by^n+\bv_{j-1}),\quad j=2,\ldots,s,
	\end{equation}
	and we can also rewrite \cref{eq:byplusdk} as
	\begin{equation}\label{eq:byplusvk}
	\by^n+c_k\Delta t A \by^n+\bv_k=R_k(\Dt A)\by^n.
	\end{equation}
	Applying \cref{eq:sum_perturbations} to \cref{eq:rkc_lin_eq_v}, (with $\bv_j$ instead of $\bd_j$ and $\bm{r}_j=\mu_j c_{j-1}(\Delta t A)^2\by^n$) gives
	\begin{equation}\label{eq:vk_sumlin}
	\bv_k=\sum_{j=2}^k\frac{b_k}{b_j}U_{k-j}(\oz I+\ou \Dt A)\mu_jc_{j-1}(\Delta t A)^2\by^n,\qquad k=1,\ldots,s.
	\end{equation}
	Combining \cref{eq:byplusvk,eq:vk_sumlin} we obtain \cref{eq:defBarBarR}. %Finally, \ref{item:eqr} follows from \cref{eq:sum_perturbations,eq:defBarR}.
\end{proof}
\begin{lemma}\label{lemma:coroll_perturbations}
	Under the assumptions of \cref{lemma:perturbations} we suppose that $A\in\R^{n\times n}$ is a symmetric nonpositive definite matrix and $s$ is such that $\Dt\rho\leq \beta^p(s,\eps)$, where $\rho$ is the spectral radius of $A$ and $\beta^p(s,\eps)$ is as in \cref{eq:defbetas}. Let ${\bm{r}}_j\in\R^n$, $j=1,\ldots,s$, and
	\begin{align}
	\bd_k &= \sum_{j=1}^k\frac{b_k}{b_j}U_{k-j}(\oz I+\ou\Dt A)(\mu_j+\gamma_j)\bm{r}_j, &
	\tilde{\bd}_k &= \sum_{j=1}^k\frac{b_k}{b_j}U_{k-j}(\oz I+\ou\Dt A)\mu_j\bm{r}_j.
	\end{align}
	Then:
	\begin{enumerate}[label=\roman*)]
		%\item\label{item:eqr} Let $\bm{r}\in\R^n$ and $\bm{r}_j=(\mu_j+\gamma_j)\bm{r}$ for $j=1,\ldots,s$, with $\gamma_1=0$. Then $\bd_k=\overline{R}_k(\Delta t A)\bm{r}$.
		\item\label{item:boundnorm} $\Vert \bd_k\Vert_2\leq c_k \max_{j=1,\ldots,k}\Vert \bm{r}_j\Vert_2$ with $c_k$ as in \cref{eq:defc}. Note that $c_k\leq 1$ \cite{VerwerHundsdorfer1990RKC}.
		\item\label{item:boundnormbis} $\Vert \tilde \bd_k\Vert_2\leq C_k \max_{j=1,\ldots,k}\Vert \bm{r}_j\Vert_2$ with $C_k=c_k$ for $p=1$ and $C_k$ a small constant for $p=2$.
	\end{enumerate}
\end{lemma}
\begin{proof}
	
	For \ref{item:boundnorm} we use 
	\begin{equation}
	\Vert U_{k-j}(\oz I+\ou \Dt A)\Vert_2\leq \max_{-\beta^p(s,\eps)\leq z\leq 0} |U_{k-j}(\oz+\ou z)|\leq |U_{k-j}(\oz)|= U_{k-j}(\oz),
	\end{equation}
	and the fact that $b_j,\mu_j+\gamma_j\geq 0$, hence
	\begin{equation}
	\begin{aligned}
	\Vert\bd_k\Vert_2 &\leq \sum_{j=1}^k\frac{b_k}{b_j}\Vert U_{k-j}(\oz I+\ou \Delta t A)\Vert_2(\mu_j+\gamma_j) \Vert \bm{r}_j\Vert_2
	\leq \max_{j=1,\ldots,k}\Vert \bm{r}_j\Vert_2 \sum_{j=1}^k\frac{b_k}{b_j} U_{k-j}(\oz)(\mu_j+\gamma_j) = \max_{j=1,\ldots,k}\Vert \bm{r}_j\Vert_2 \overline{R}_k(0),
	\end{aligned}
	\end{equation}
	and we conclude using $\overline{R}_k(0)=R_k'(0)=c_k$ \cite{VerwerHundsdorfer1990RKC}. For \cref{item:boundnormbis} with $p=1$ we obtain the same result since $\gamma_j=0$. For $p=2$ we use $\mu_j>0$ to deduce
	\begin{equation}
	\begin{aligned}
	\Vert\tilde{\bd_k}\Vert_2 
	&\leq \max_{j=1,\ldots,k}\Vert \bm{r}_j\Vert_2 \sum_{j=1}^k\frac{b_k}{b_j} U_{k-j}(\oz)\mu_j = C_k \max_{j=1,\ldots,k}\Vert \bm{r}_j\Vert_2.
	\end{aligned}
	\end{equation}
	The constant $C_k$ is bounded using $b_k/b_j<4/3$, $U_{k-j}(\oz)\approx k-j+1$ and $\mu_j\approx 6/(s^2-1)$. When there is no damping those approximations are exact and yield $C_k=4 k(k+1)/(s^2-1)$, hence $C_k\leq C_s=4s(s+1)/(s^2-1)$, which is bounded by $8$ (note that $s\geq 2$ for $p=2$) and decreasing in $s$.
\end{proof}
\color{black}

\paragraph{Convergence analysis}
\mbox{}\\
\moda{We present here the convergence analysis in which we prove that the mixed-precision RKC schemes \cref{eq:1-order-preserving-scheme} and \cref{eq:hybrid-scheme} are 1- and 2-order-preserving, respectively.}

The order of convergence of first- and second-order explicit stabilized schemes is typically only proven for linear problems \cite{AbM01,Ver96,VerwerHundsdorfer1990RKC}, as this is sufficient to infer convergence in the nonlinear case as well \cite{HNW08}. By exploiting the internal stability properties of these methods, such an approach yields error bounds that are independent from the problem stiffness. These results are unusual for explicit methods and are akin to those obtained with B-convergence analysis for implicit methods. In contrast, in our analysis we directly consider the nonlinear case to show that the low-precision Jacobian approximations proposed in \cref{sec:evalhDfj} do not impact convergence. For this purpose, we perform a standard Taylor expansion of the mixed-precision schemes in \cref{thm:convmpRKCbis}, and verify that the order of convergence is preserved. However, with this strategy we cannot see the \moda{benefits, in terms of stability,} of $\hat\Delta\bf_j$ approximating the Jacobian. Therefore, in \cref{thm:accuracympRKC} below we adopt a stronger assumption and unveil the good stability properties brought by the $\hat\Delta\bf_j$ terms.

\moda{In \cref{thm:convmpRKCbis} we assume that we are in a convergence regime so that $\hat\Delta\bm{f}_j$ satisfies \cref{eq:requirementRKC2} in scheme \cref{eq:hybrid-scheme}.}

\begin{theorem}\label{thm:convmpRKCbis}
	Assuming \cref{eq:requirementRKC1}, the solution $\hby^{n+1}$ of the mixed-precision RKC scheme \cref{eq:first-order-scheme} with coefficients \cref{eq:coeffRKC} and \cref{eq:coeffRKC1} or \cref{eq:coeffRKC2} satisfies
	\begin{equation}\label{eq:taylormpRKC1}
	\hby^{n+1} = \hby^n+\Dt\bf(\hby^n)+O(\epsilon\Dt^2+\Dt^2).
	\end{equation}	
	Assuming \cref{eq:requirementRKC2}, the solution $\hby^{n+1}$ of the mixed-precision RKC scheme \cref{eq:first-order-scheme} with coefficients \cref{eq:coeffRKC} and \cref{eq:coeffRKC2} satisfies
	\begin{equation}\label{eq:taylormpRKC2}
	\hby^{n+1} = \hby^n+\Dt\bf(\hby^n)+\frac{1}{2}\Dt^2\bf'(\hby^n)\bf(\hby^n)+O(\Dt^3).
	\end{equation}
\end{theorem}
\color{myred}
\begin{proof}
	To obtain \cref{eq:taylormpRKC1} we apply \cref{lemma:perturbations} to \cref{eq:first-order-scheme}, with $A=0$ and $\bm{r}_{j}=\mu_j\Dt(\bf(\hby^n)+\hat\Delta\bf_{j-1})+\gamma_j\Dt\bf(\hby^n)$ for $j=1,\ldots,s$ (with $\gamma_1=0$, $\hat\Delta\bf_{0}=\bm{0}$), we obtain
	\begin{equation}\label{eq:sumfordk}
	\begin{aligned}
	\hbd_k &=\sum_{j=1}^k\frac{b_k}{b_j}U_{k-j}(\oz)(\mu_j\Dt(\bf(\hby^n)+\hat\Delta\bf_{j-1})+\gamma_j\Dt\bm{f}(\hby^n))
	=\overline{R}_k(0)\Dt \bm{f}(\hby^n)+\sum_{j=1}^k\frac{b_k}{b_j}U_{k-j}(\oz)\mu_j\Dt\hat\Delta\bf_{j-1},
	\end{aligned}
	\end{equation}
	hence using $\overline{R}_k(0)=c_k$ and $c_s=1$ we have
	\begin{equation}\label{eq:convblabla}
	\begin{aligned}
	\hby^{n+1}=\hby^n + \hbd_s 
	&= \hby^n + \Dt\bf(\hby^n)+\Dt\sum_{j=1}^s\frac{b_s}{b_j}U_{s-j}(\oz)\mu_j\hat\Delta\bf_{j-1}.
	\end{aligned}
	\end{equation}
	The result follows by applying point \ref{item:boundnormbis} of \cref{lemma:coroll_perturbations}, equation \cref{eq:requirementRKC1} and $\Delta\bm{f}_j=O(\Dt)$ to the last sum.
	We now prove \cref{eq:taylormpRKC2}. From $\hat\Delta\bm{f}_j=O(\Dt)$ and \cref{eq:sumfordk} we see that $\hbv_j=\hbd_j-c_j\Dt\bm{f}(\hby^n)=O(\Dt^2)$. Therefore, by using \cref{eq:requirementRKC2} it follows that $\hat\Delta\bf_{j-1}=\bf'(\hby^n)c_{j-1}\Dt\bf(\hby^n) +\bm{r}_j$ with $\bm{r}_j=O(\Dt^2)$.
	From \cref{eq:convblabla} we compute
	\begin{equation}
	\begin{aligned}
	\hby^{n+1}
	&= \hby^n + \Dt\bf(\hby^n)+\Dt\sum_{j=1}^s\frac{b_s}{b_j}U_{s-j}(\oz)\mu_j(c_{j-1}\Dt \bm{f}'(\hby^n)\bm{f}(\hby^n)+\bm{r}_j)\\
	&=  \hby^n + \Dt\bf(\hby^n)+\widetilde{R}_s(0)\Dt^2 \bm{f}'(\hby^n)\bm{f}(\hby^n)+\Dt\sum_{j=1}^s\frac{b_s}{b_j}U_{s-j}(\oz)\mu_j\bm{r}_j,
	\end{aligned}
	\end{equation}
	we conclude by using point \ref{item:boundnormbis} of \cref{lemma:coroll_perturbations} and the relation $\widetilde{R}_s(0)=\lim\limits_{z\to 0}\ (R_s(z)-1-z)/z^2= 1/2$.
\end{proof}
\color{black}
In the proof of \cref{thm:convmpRKCbis} we cannot infer anything about the stability of the methods. In order to \moda{investigate} stability, we require the following assumption:

\begin{assumption}\label{ass:ddf}
	\moda{Schemes \cref{eq:1-order-preserving-scheme} and \cref{eq:hybrid-scheme} satisfy $\Vert\hbd_j\Vert_2\leq\Cdj\Vert\by^n\Vert_2$ for some $\Cdj\geq 2$, and a similar assumption also holds for scheme \cref{eq:RKC}.
		Furthermore $\hat\Delta\bm{f}_j$ satisfy $\Vert\hat\Delta\bm{f}_j-\Delta\bm{f}_j\Vert_2\leq C_\Delta\Vert\hby^n\Vert_2$ for some $C_\Delta>0$.
	}
	Finally, $\bf$ is twice differentiable with $\Vert\bf''(\by)\Vert_2\leq \Cddf$, where $\Cddf>0$ is a small constant.
\end{assumption}
\noindent
\moda{\Cref{ass:ddf} is obviously satisfied as $\Dt\to 0$ since in that case $\hbd_j=c_j\Dt\bm{f}(\hby^n)+O(\Dt^2)$ and $\hat\Delta\bm{f}_j$ satisfy \cref{eq:requirementRKC1} or \cref{eq:requirementRKC2}. For large $\Dt$ it is an internal stability assumption. Note that for linear problems and scheme \cref{eq:RKC} it holds $\Vert\bd_j\Vert_2\leq 2\Vert\by^n\Vert_2$. We will discuss internal stability (i.e.~error propagation within one time step) and the validity of \cref{ass:ddf} in \cref{sec:internalstab}.}
For our analysis it is sufficient for the bound $\Vert\bf''(\by)\Vert_2\leq \Cddf$ to be satisfied in a neighborhood of the solution. Note that assuming $\Cddf$ to be small is not restrictive since the norm of $\bf''$ is not related to stiffness. To see this consider, for instance, the linear case $\bf(\by)=A\by$ where $\Cddf=0$, or also the examples of \cref{sec:numerical_results}.

\moda{\Cref{thm:accuracympRKC,thm:accuracympRKC2orderpres} below provide} estimates for the local errors \moda{of \cref{eq:1-order-preserving-scheme,eq:hybrid-scheme}} and their propagation under \cref{ass:ddf}. Note that the constants in the error estimates \cref{eq:errRKC1,eq:errRKC2,eq:errRKC2-2ordpres} do not depend on the number of stages $s$ nor on the stiffness of $\bf$.
\color{myred}
\begin{theorem}\label{thm:accuracympRKC} 
	Let \cref{ass:ddf} be satisfied and let $\bf'(\by)$ be symmetric and nonpositive definite.
	The error between the exact RKC scheme \cref{eq:RKC} and the mixed-precision RKC scheme \cref{eq:1-order-preserving-scheme}, both with first-order coefficients \cref{eq:coeffRKC,eq:coeffRKC1}, satisfies
	\begin{equation}\label{eq:errRKC1}
	\begin{aligned}
	\Vert \hby^{n+1}-\by^{n+1} \Vert_2 &\leq \Vert \hby^{n}-\by^{n} \Vert_2 +\Cddf\Vert \hby^{n}-\by^{n} \Vert_2^2\Dt + \min(\hat C\epsilon\Dt,C_\Delta\Vert\hby^n\Vert_2)\Dt\\
	&\quad+\Cddf\Cdj^2(\min(\Dt\Vert\bf(\hby^n)\Vert_2,\Vert\hby^n\Vert_2)^2+\min(\Dt\Vert\bf(\by^n)\Vert_2,\Vert\by^n\Vert_2)^2)\Dt,
	\end{aligned}
	\end{equation}
	where $\hat C$ is the error constant in \cref{eq:requirementRKC1}.
	
	The error between the exact RKC scheme \cref{eq:RKC} and the mixed-precision RKC scheme \cref{eq:1-order-preserving-scheme}, both with second-order coefficients \cref{eq:coeffRKC,eq:coeffRKC2}, satisfies
	\begin{equation}\label{eq:errRKC2}
	\begin{aligned}
	\Vert \hby^{n+1}-\by^{n+1} \Vert_2 &\leq \Vert \hby^{n}-\by^{n} \Vert_2 +\Cddf(1+2C_s)\Vert \hby^{n}-\by^{n} \Vert_2^2\Dt + C_s\min(\hat C\epsilon\Dt,C_\Delta\Vert\hby^n\Vert_2)\Dt\\
	&\quad+C_s\Cddf\Cdj^2(\min(\Dt\Vert\bf(\hby^n)\Vert_2,\Vert\hby^n\Vert_2)^2+\min(\Dt\Vert\bf(\by^n)\Vert_2,\Vert\by^n\Vert_2)^2)\Dt,
	\end{aligned}
	\end{equation}
	where $C_s$ is as in \cref{lemma:coroll_perturbations}.
\end{theorem}
\begin{proof}
	We prove first \cref{eq:errRKC1}. Let $\bm{E}^n=\hby^n-\by^n$ and $\be_j=\hbd_j-\bd_j$, subtracting \cref{eq:RKC} from \cref{eq:first-order-scheme} yields
	\begin{equation}\label{eq:erriter}
	\begin{aligned}
	\be_0 &= \bm{0}, \qquad \be_1=\mu_1\Dt(\bf(\hby^n)-\bf(\by^n)),\\
	\be_j &= \nu_j\be_{j-1}+\kappa_j\be_{j-2}+\mu_j\Dt(\bm{f}(\hby^n)+\hat\Delta\bm{f}_{j-1}-\bm{f}(\by^n+\bd_{j-1}))+\gamma_j\Dt(\bm{f}(\hby^n)-\bm{f}(\by^n)), \quad j=2,\ldots,s,\\
	\bm{E}^{n+1}&=\bm{E}^n+\be_s.
	\end{aligned}
	\end{equation}
	From \cref{eq:requirementRKC1} we have $\bm{f}(\hby^n)+\hat\Delta\bm{f}_{j}=\bm{f}(\hby^n+\hbd_j)+\bm{r}_{j}$ with $\Vert\bm{r}_{j}\Vert_2=\Vert\hat\Delta\bm{f}_j-\Delta\bm{f}_j\Vert_2\leq \min(\hat C\epsilon\Dt,C_\Delta\Vert\hby^n\Vert_2)$ (cf. \cref{ass:ddf} and \cref{eq:requirementRKC1}) and $C_\Delta,\hat C,\epsilon$ depending on the definition of $\hat\Delta\bm{f}_{j}$. Hence, using $\hby^n=\by^n+\bm{E}^n$,
	\begin{equation}\label{eq:stageerr}
	\begin{aligned}
	\bm{f}(\hby^n)+\hat\Delta\bm{f}_{j}-\bm{f}(\by^n+\bd_j) &= \bm{f}(\by^n+\bm{E}^n+\hbd_j)-\bm{f}(\by^n+\bd_j) +\bm{r}_{j} = \bm{f}'(\by^n)(\bm{E}^n+\be_j)+\bm{r}_{j}+\bm{t}_{j},
	\end{aligned}
	\end{equation}
	where $\bm{t}_{j}$ is the residual of the Taylor expansions of $\bm{f}(\by^n+\bm{E}^n+\hbd_j)$ and $\bm{f}(\by^n+\bd_j)$. From $\hbd_j=c_j\Dt\bm{f}(\hby^n)+O(\Dt^2)$ and \cref{ass:ddf} we have $\Vert\hbd_j\Vert_2\leq \Cdj\min(\Dt\Vert\bm{f}(\hby^n)\Vert_2,\Vert\hby^n\Vert_2)$, since $\Vert\bf''(\by)\Vert_2\leq \Cddf$ it follows that
	\begin{equation}
	\Vert \bm{t}_{j}\Vert_2\leq \Cddf(\Vert\bm{E}^n\Vert_2^2+\Vert\hbd_j\Vert_2^2+\Vert\bd_j\Vert_2^2)\leq \Cddf (\Vert\bm{E}^n\Vert_2^2+\Cdj^2(\min(\Dt\Vert\bf(\hby^n)\Vert_2,\Vert\hby^n\Vert_2)^2+\min(\Dt\Vert\bf(\by^n)\Vert_2,\Vert\by^n\Vert_2)^2)).
	\end{equation}
	Here we used \cref{ass:ddf}. Similarly, we have $\bf(\hby^n)-\bf(\by^n)=\bf'(\by^n)\bm{E}^n+\bm{t}_0$ with $\Vert\bm{t}_0\Vert_2\leq \Cddf\Vert\bm{E}^n\Vert_2^2$. 
	Inserting \cref{eq:stageerr} into \cref{eq:erriter} yields
	\begin{equation}
	\begin{aligned}
	\be_0 &= \bm{0}, \qquad \be_1=\mu_1\Dt(\bf'(\by^n)\bm{E}^n+\bm{t}_0),\\
	\be_j &= \nu_j\be_{j-1}+\kappa_j\be_{j-2}+\mu_j\Dt\bf'(\by^n)\be_{j-1}+\mu_j\Dt(\bf'(\by^n)\bm{E}^n+ \bm{r}_{j-1}+\bm{t}_{j-1})+\gamma_j\Dt(\bf'(\by^n)\bm{E}^n+\bm{t}_0), \quad j=2,\ldots,s,
	\end{aligned}
	\end{equation}
	Let $\bm{r}_0=\bm{0}$. Points \ref{item:perturbations} and \ref{item:eqr} of \cref{lemma:perturbations} together with the relation $R_s(z)=1+\overline{R}_s(z)z$ imply that
	\begin{equation}\label{eq:kjsbf}
	\begin{aligned}
	\bm{E}^{n+1}&=\bm{E}^n+\be_s
	=\bm{E}^n+\sum_{j=1}^s\frac{b_s}{b_j}U_{s-j}(\oz I+\ou\Dt\bf'(\by^n))\Dt((\mu_j+\gamma_j)\bf'(\by^n)\bm{E}^n+\mu_j(\bm{r}_{j-1}+\bm{t}_{j-1})+\gamma_j\bm{t}_0)\\
	&=\bm{E}^n+\overline{R}_s(\Dt\bf'(\by^n))\Dt\bf'(\by^n)\bm{E}^n + \Dt\sum_{j=1}^s\frac{b_s}{b_j}U_{s-j}(\oz I+\ou\Dt\bf'(\by^n))(\mu_j(\bm{r}_{j-1}+\bm{t}_{j-1})+\gamma_j\bm{t}_0)\\
	&=R_s(\Dt\bf'(\by^n))\bm{E}^n + \Dt\sum_{j=1}^s\frac{b_s}{b_j}U_{s-j}(\oz I+\ou\Dt\bf'(\by^n))(\mu_j(\bm{r}_{j-1}+\bm{t}_{j-1})+\gamma_j\bm{t}_0).
	\end{aligned}
	\end{equation}
	We conclude by applying point \ref{item:boundnorm} in \cref{lemma:coroll_perturbations}, by using the relation $\Vert R_s(\Dt\bf'(\by^n))\Vert_2\leq 1$, and by noting that for first-order coefficients we have $\gamma_j=0$.
	In order to prove \cref{eq:errRKC2}, we rewrite \cref{eq:kjsbf} as
	\begin{equation}
	\begin{aligned}
	\bm{E}^{n+1}
	&=R_s(\Dt\bf'(\by^n))\bm{E}^n+ \Dt\sum_{j=1}^s\frac{b_s}{b_j}U_{s-j}(\oz I+\ou\Dt\bf'(\by^n))(\mu_j(\bm{r}_{j-1}+\bm{t}_{j-1}-\bm{t}_0)+(\mu_j+\gamma_j)\bm{t}_0).
	\end{aligned}
	\end{equation}
	The thesis is then readily obtained by invoking point \ref{item:boundnormbis} of \cref{lemma:coroll_perturbations}.
\end{proof}

\begin{theorem}\label{thm:accuracympRKC2orderpres}
	Let \cref{ass:ddf} be satisfied and let $\bf'(\by)$ be symmetric and nonpositive definite.
	The error between the exact RKC scheme \cref{eq:RKC} with second-order coefficients \cref{eq:coeffRKC,eq:coeffRKC2} and the mixed-precision RKC scheme \cref{eq:hybrid-scheme} satisfies
	\begin{equation}\label{eq:errRKC2-2ordpres}
	\begin{aligned}
	\Vert \hby^{n+1}-\by^{n+1} \Vert_2 &\leq \Vert \hby^{n}-\by^{n} \Vert_2 +\Cddf(1+2C_s)\Vert \hby^{n}-\by^{n} \Vert_2^2\Dt + C_s\min(\hat C\Dt^2,C_\Delta\Vert\hby^n\Vert_2)\Dt\\
	&\quad+C_s\Cddf\Cdj^2(\min(\Dt\Vert\bf(\hby^n)\Vert_2,\Vert\hby^n\Vert_2)^2+\min(\Dt\Vert\bf(\by^n)\Vert_2,\Vert\by^n\Vert_2)^2)\Dt,
	\end{aligned}
	\end{equation}
	with $\hat C$ the error constant in \cref{eq:requirementRKC2} and $C_s$ as in \cref{lemma:coroll_perturbations}.
\end{theorem}
\begin{proof}
	The proof is analogous to \cref{thm:convmpmRKC} but with $\Vert \bm{r}_j\Vert_2\leq \min(\hat C\Dt^2,C_\Delta\Vert\hby^n\Vert_2)$.
\end{proof}

Note that due to $\hat\Delta\bf_j$ approximating the Jacobian we could use the stability polynomials of the RKC methods and show that the lower-order term $\Vert \hby^n-\by^n\Vert_2$ is not amplified. Due to the nonlinearity we also have a term $\Vert \hby^n-\by^n\Vert_2^2\Dt$ in the recursive relation, but it is of higher order.
If we had not made \cref{ass:ddf} outside of a convergence regime the error could still grow with $\Dt$ like $\hat C\eps\Dt$ (or $\hat C \Dt^2$), $\Dt\Vert\bm{f}(\hby^n)\Vert_2$, and $\Dt\Vert\bm{f}(\by^n)\Vert_2$. These terms could become very large and cause the scheme to become unstable. 
\color{black}

\subsection{Internal error propagation}\label{sec:internalstab}
\mbox{}\\
We now investigate the propagation of rounding errors within one time step and \moda{the} validity of \cref{ass:ddf}. For this purpose, we assume that the problem is linear and therefore $\bf(\by)=A\by$, where $A\in\R^{n\times n}$ is a symmetric nonpositive definite matrix. \moda{We also assume that in \cref{eq:requirementRKC1,eq:requirementRKC2} $\hat\Delta\bm{f}_j$ is computed with \cref{eq:evalDeltaf_mixed,eq:evalDeltaf_mixed2} respectively.} We stress that the next estimates are very pessimistic as they are worst-case bounds and do not take into account rounding error cancellation effects \cite{HighamMary2019}. 

\begin{theorem}\label{thm:stabmpRKC}
	\moda{Let $\hby^n$ be the solution computed by the mixed-precision RKC schemes \cref{eq:1-order-preserving-scheme} or \cref{eq:hybrid-scheme}, and let $s$ be} such that $\Dt\rho\leq \beta^p(s,\eps)$, where $\rho$ is the spectral radius of the nonpositive definite matrix $A$. Then
	\begin{equation}\label{eq:ynpu}
	\hby^{n+1}= R_s(\Dt A) \hby^n+\bm{r}_s(\hby^n),
	\end{equation}
	where $\bm{r}_s(\hby^n)$ represents the rounding errors introduced at time step $n$. 
	\color{myred}
	It holds
	\begin{equation}\label{eq:boundrs}
	\Vert \bm{r}_s(\hby^n)\Vert_2
	\leq \max_{k=1,\ldots,s-1}\Vert R_k(\Dt A)-I\Vert_2 \left((1+C(s,\eps)\Dt u)^{s-1}-1\right)\Vert\hby^n\Vert_2,
	\end{equation}
	where $\barc$, $\barm$ are as in \cref{lemma:boundDA} and
	\begin{equation}
	C(s,\eps) = 3\barc\barm^2\ou\rho \max_{j=0,\ldots,s-2}\Vert U_{j}(\oz I +\ou\Dt A)\Vert_2.
	\end{equation}
\end{theorem}
%\color{black}
\color{myred}% non so perché devo ridirgli di mettere il rosso da qui via
\begin{proof}
	Scheme \cref{eq:first-order-scheme} with $\bf(\by)=A\by$ and $\hat\Delta\bf_j$ as in \cref{eq:evalDeltaf_mixed} reads
	\begin{equation}\label{eq:recstaban}
	\begin{aligned}
	\hbd_0 &= \bm{0},\quad \hbd_1 = \mu_1 \Dt A \hby^n,\\
	\hbd_j &= \nu_j \hbd_{j-1} + \kappa_j \hbd_{j-2} + \mu_j\Delta t\hat{A}\hbd_{j-1} + (\mu_j+\gamma_j) \Dt A \hby^n\\
	&= \nu_j \hbd_{j-1} + \kappa_j \hbd_{j-2} + \mu_j\Delta t A\hbd_{j-1} +(\mu_j+\gamma_j) \Dt A \hby^n+ \mu_j \Dt\Delta A_{j-1}\hbd_{j-1}, \quad j=2,\dots,s.
	\end{aligned}
	\end{equation}
	Here we are again using the notation $\Delta A_{j-1}$ to indicate rounding errors in the matrix-vector products. \Cref{lemma:perturbations} implies
	\begin{equation}\label{eq:exprhbdk}
	\begin{aligned}
	\hbd_k=\overline{R}_k(\Dt A)\Dt A\hby^n+\bm{r}_k(\hby^n)
	= (R_k(\Dt A)-I)\hby^n+\bm{r}_k(\hby^n),
	\end{aligned}
	\end{equation}
	\color{black}
	with 
	\begin{equation}\label{eq:defrk}
	\bm{r}_k(\hby^n)=\sum_{j=1}^k\frac{b_k}{b_j}U_{k-j}(\oz I +\ou\Dt A)\mu_j\Dt \Delta A_{j-1} \hbd_{j-1}
	=2\ou\Dt\sum_{j=1}^{k-1}\frac{b_k}{b_j}U_{k-j-1}(\oz I +\ou\Dt A)\Delta A_{j} \hbd_{j},
	\end{equation}
	where for the second equality we used $\hbd_0=\bm{0}$ and \cref{eq:coeffRKC}. Equation \cref{eq:exprhbdk} and $\hby^{n+1}=\hby^n+\hbd_s$ yield \cref{eq:ynpu}. \Cref{lemma:boundDA}, \cref{eq:defrk} and \moda{$b_k/b_j\leq 3/2$} imply
	\begin{equation}\label{eq:boundtmprk}
	\Vert \bm{r}_k(\hby^n)\Vert_2
	\leq \moda{3}\ou\Dt \max_{j=0,\ldots,k-2}\Vert U_{j}(\oz I +\ou\Dt A)\Vert_2\sum_{j=1}^{k-1}\Vert \Delta A_{j} \hbd_{j}\Vert_2
	\leq C(s,\eps)\Dt u\sum_{j=1}^{k-1}\Vert\hbd_{j}\Vert_2.
	\end{equation}
	%which gives \cref{eq:boundrs1}.
	%Finally, $\overline R_k(z)z=R_k(z)-1$, \cref{eq:exprhbdk,eq:boundtmprk} imply \cref{eq:bounddk}.
	Using \cref{eq:boundtmprk} in \cref{eq:exprhbdk} we obtain
	\begin{equation}
	\Vert\hbd_k\Vert_2\leq \Vert R_k(\Dt A)-I\Vert_2\Vert\hby^n\Vert_2+C(s,\eps)\Dt u\sum_{j=1}^{k-1}\Vert\hbd_j\Vert_2,
	\end{equation}
	and prove, recursively, that 
	\begin{equation}\label{eq:boundmaxrk}
	\max_{k=1,\ldots,s}\Vert \bm{r}_k(\hby^n)\Vert_2\leq C(s,\eps)\Dt u\sum_{j=1}^{s-1}\Vert\hbd_j\Vert_2\leq \max_{k=1,\ldots,s-1}\Vert R_k(\Dt A)-I\Vert_2 \left((1+C(s,\eps)\Dt u)^{s-1}-1\right)\Vert\hby^n\Vert_2.
	\end{equation}
	\moda{For the hybrid scheme \cref{eq:hybrid-scheme} the proof is analogous. The only difference is that if $\Vert \hbd_j-c_j\Dt\bm{f}(\hby^n)\Vert_2\leq\Vert\hbd_j\Vert_2$ then in \cref{eq:recstaban} we have $\Delta A_{j-1}(\hbd_{j-1}-c_{j-1}\Dt\bm{f}(\hby^n))$ instead of $\Delta A_{j-1}\hbd_{j-1}$. After bounding $\Vert \hbd_j-c_j\Dt\bm{f}(\hby^n)\Vert_2\leq\Vert\hbd_j\Vert_2$, the rest of the proof remains unchanged.}
\end{proof}

\color{myred}
From \cref{thm:stabmpRKC} we can derive rough upper bounds for the constants $\Cdj$, $C_\Delta$ in \cref{ass:ddf}. Equations \cref{eq:exprhbdk,eq:boundmaxrk} yield
\begin{equation}
\Vert\hbd_j\Vert_2\leq \max_{k=1,\ldots,s-1}\Vert R_k(\Dt A)-I\Vert_2 (1+C(s,\eps)\Dt u)^{s-1}\Vert\hby^n\Vert_2
\end{equation}
and thus, using $\Vert R_k(\Dt A)-I\Vert_2\leq 2$,
\begin{equation}
\Cdj = 2 (1+C(s,\eps)\Dt u)^{s-1} \leq 2 e^{C(s,\eps)\Dt u(s-1)}\leq 2e^{2\bar{C}s^2u}.
\end{equation}
We deduce that $\Cdj$ is guaranteed to remain small provided that $s^2 u=O(1)$. However, we remark that the bounds in \cref{thm:stabmpRKC} are worst-case rounding error bounds that are very pessimistic, and we see that in practice the scheme remains stable also for $s^2 u\gg 1$, see \cref{sec:numerical_results}.

We now look at the effect of conditions \cref{eq:requirementRKC1,eq:requirementRKC2} onto $C_\Delta$. Under the first-order condition \cref{eq:requirementRKC1} we have
\begin{equation}
\Vert\hat\Delta\bm{f}_j-\Delta\bm{f}_j\Vert_2=\Vert \hat A\hbd_j-A\hbd_j\Vert_2
\leq  \barc\barm^2\rho u\Vert\hbd_j\Vert_2\leq \Cdj \barc\barm^2\rho u  \Vert\hby^n\Vert_2,
\end{equation}
hence $C_\Delta =\Cdj \barc\barm^2\rho u$ and a similar discussion follows as for $\Cdj$. On the other hand, under the second-order condition \cref{eq:requirementRKC2} we obtain instead 
\begin{equation}
\begin{aligned}
\Vert\hat\Delta\bm{f}_j-\Delta\bm{f}_j\Vert_2 &=\Vert \hat A(\hbd_j-c_j\Dt A\hby^n)-A(\hbd_j-c_j\Dt A\hby^n)\Vert_2\leq  \barc\barm^2\rho u\Vert\hbd_j-c_j\Dt A\hby^n\Vert_2 
\leq \barc\barm^2\rho u( \Cdj+c_j\Dt \rho)\Vert\hby^n\Vert_2
\end{aligned}
\end{equation}
and $C_\Delta = \barc\barm^2\rho u( \Cdj+c_j\Dt \rho)$. Under condition \cref{eq:requirementRKC2} $C_\Delta$ is therefore much larger than under condition \cref{eq:requirementRKC1} due to the term $\Dt\rho=O(s^2)$. This further confirms that condition \cref{eq:requirementRKC2} should only be enforced in a convergence regime, i.e. when $\Dt\rho=O(1)$, and explains why we resort to the hybrid scheme \cref{eq:hybrid-scheme}.

We remark that relation \cref{eq:ynpu} has a crucial difference with respect to the rounding error estimate given in \cite[Eq. (3.10)]{VerwerHundsdorfer1990RKC}. In \cite{VerwerHundsdorfer1990RKC} the perturbations at each stage were assumed to be independent, while here we are considering the propagation of previous perturbations. Indeed, in \cref{eq:defrk} each perturbation $\Delta A_k\hbd_k$ depends on $\hbd_k$, which in turn depends on $\Delta A_j\hbd_j$ for $j=1\ldots,k-1$ \cref{eq:exprhbdk}. The estimate found in \cite{VerwerHundsdorfer1990RKC} is for a standard RKC method \cref{eq:RKC} in which all operations are performed with the same precision $u$, and we can thus compare their result with the one we obtained for our mixed-precision scheme. 
\color{myred}
To do so, let $\psi_{\max}\! \coloneqq \max_{k=1,\ldots,s-1}\Vert R_k(\Dt A)-I\Vert_2$, from \cref{eq:ynpu}
\begin{equation}\label{eq:boundynpu}
\Vert\hby^{n+1}\Vert_2
\leq \left(1+\psi_{\max} \left((1+C(s,\eps)\Dt u)^{s-1}-1\right)\right) \Vert\hby^n\Vert_2.
\end{equation}
In the asymptotic regime, it holds that
\begin{equation}
(1+C(s,\eps)\Dt u)^{s-1}-1\approx C(s,\eps)\Dt u(s-1).
\end{equation}
Using the bound $\max_{j=0,\ldots,s-2}\Vert U_j(\oz I+\ou\Dt A)\Vert_2\leq \tilde C(s-1)$, with $\tilde C$ close to $1$, we have $C(s,\eps)\leq 3\tilde \bar{C} \barc\barm^2 \ou\rho(s-1)$, hence
\begin{equation}
C(s,\eps)\Dt u (s-1)\leq 3\tilde C \barc\barm^2 \ou\rho\Dt u(s-1)^2 \leq C \Dt\rho u. %\leq C\beta^p(s,\eps)u \leq 2Cs^2 u,
\end{equation}
Here we used the relation $\ou(s-1)^2\leq C_\omega$ where $C_\omega$ is small (cf.~\cref{eq:coeffRKC1,eq:coeffRKC2}), and we have set $\bar{C}=3C_\omega\tilde C \barc\barm^2$.

Since $\psi_{\max}=O(\Dt)$, the stability estimate of our mixed-precision RKC schemes \cref{eq:1-order-preserving-scheme,eq:hybrid-scheme} behaves as $1+C \rho u \Dt^{2}\leq 1+Cs^2 u\Dt$. In contrast, in \cite{VerwerHundsdorfer1990RKC} the authors find that the constant in the stability estimate behaves as $1+C s^2 u$ and is independent from $\Dt$. This difference stems from the fact that standard RKC schemes are not order preserving, i.e.~if they are run entirely in low precision they do not converge and their error stagnates (or grows like $O(u\Dt^{-1})$, cf.~\cite{CrociGilesSR2020}) as $\Dt\rightarrow 0$.

In a non-asymptotic regime $\psi_{\max}=O(1)$ thus the constant in estimate \cref{eq:boundynpu} could grow as quickly as $1+\bar{C}\rho u\Dt=1+\bar{C}s^{2}u$ and suggests that the scheme might become unstable whenever $s^{2} u$ is large. We investigate the stability of the mixed-precision RKC schemes in practice in \cref{sec:numerical_results} and verify that they remain stable for a very large number of stages (we stopped our experiments at $s=512$) even when the low-precision computations are performed in half-precision.
\color{black}

Unfortunately, we were unable to prove stability of the schemes analytically. The main difficulty stems from the fact that rounding errors affect all frequencies, and destroy any spectral relation between $\hby^n$ and $\hbd_j$, therefore \moda{impeding} any kind of stability analysis based on: 1) damping effects due to eigenvalues far from the origin, and 2) accuracy for those close to zero. Indeed, the accuracy properties of the stability polynomial would need to be taken into account to achieve better estimates: for $z$ close to zero we have $U_j(\oz+\ou z)=O(j+1)$, causing roundoff errors $\Delta A_j\hbd_j$ with low frequencies to be amplified (cf. \cref{eq:defrk}). If rounding errors preserved spectral relations, these errors would then be compensated by the fact that $R_k(z)-I\approx 0$ (cf. \cref{eq:exprhbdk}). However, frequencies of $\hby_n$ and $\Delta A_j\hbd_j$ are uncorrelated, making such an analysis impossible.  
Under the assumption that the smallest (in magnitude) eigenvalue of $\Dt A$ is sufficiently separated from the origin, we can prove that the schemes are stable by using damping properties only. This is possible thanks to the fact that $|U_j(\oz+\ou z)|\leq 2$ for $z$ sufficiently far from $0$. However, this assumption requires $A$ to have a small condition number, which is a very restrictive condition. Interestingly, lack of separation between the eigenvalues of $\Dt A$ and the origin does not seem to affect stability in practical experiments (cf.~\cref{sec:numerical_results}).
\color{black}

\section{Mixed-precision multirate RKC method}\label{sec:multirate}
In this section we consider a multirate differential equation of the type \cref{eq:mrode},
where $\fs$ is an expensive, but only mildly stiff term associated to relatively slow ($S$) time-scales and $\ff$ is a cheap, yet severely stiff term associated to fast ($F$) time scales. Typical applications are chemical kinetics problems with different reaction rates, electric circuits with active and latent components, and systems stemming from the spatial discretization of diffusion-dominated parabolic PDEs over graded meshes. In this latter case, $\ff$ and $\fs$ typically correspond to the discrete diffusion operator over the fine and coarse degrees-of-freedom respectively (i.e.~over refined and coarse mesh portions).

When an explicit stabilized scheme as RKC is applied to \cref{eq:mrode}, the number of stages $s$ is determined by the stiffness of $\ff$, even when $\ff$ has very few severely stiff degrees of freedom. Hence, the number $s$ of expensive $\fs$ evaluations depends on $\ff$ and this relation destroys the efficiency of the RKC scheme. In \cref{sec:mRKC} below we recall the mRKC scheme from \cite{AGR20}, where the evaluation of $\ff,\fs$ is decoupled and the number of $\fs$ evaluations depends solely on the mild stiffness of $\fs$ itself. Hence, the mRKC scheme is barely affected by few severely stiff terms and recovers the original efficiency of RKC methods without sacrificing accuracy.

\subsection{The multirate RKC method}\label{sec:mRKC}
The mRKC scheme is based on the modified equation
\begin{equation}\label{eq:modeq}
\ye'=\fe(\ye),\qquad \qquad \ye(0)=\by^0,
\end{equation}
for \cref{eq:mrode}. The modified right-hand side $\fe$, called \emph{averaged force}, depends on a free parameter $\eta\geq 0$ and is a good approximation to the exact $\bf=\ff+\fs$. Yet, for the right choice of $\eta$, the stiffness of $\fe$ depends on $\fs$ only and integration of \cref{eq:modeq} with an RKC scheme is cheaper than \cref{eq:mrode}. Evaluation of $\fe$ requires the solution of a stiff, yet cheap auxiliary problem, that is also approximated using an RKC scheme. 

\paragraph{The averaged force} 
\mbox{}\\
Before defining the mRKC scheme we introduce the averaged force $\fe$ and briefly discuss its properties. We refer to \cite{AGR20} for further details.
\begin{definition}\label{def:fe}
	For $\eta>0$, the averaged force $\fe:\R^n\rightarrow\R^n$ is defined as
	\begin{equation}\label{eq:deffedif}
	\fe(\by)=\frac{1}{\eta}(\bu(\eta)-\by),
	\end{equation}
	where {\it the auxiliary solution} $\bu:[0,\eta]\rightarrow \R^n$ is defined by {\it the auxiliary problem}
	\begin{align}\label{eq:defu}
	\bu'&=\ff(\bu)+\fs(\by), & 
	\bu(0)=\by.
	\end{align}
	For $\eta=0$, let $\bf_0=\bf$ (note that $\bf_0=\lim_{\eta\to 0^+}\fe$).
\end{definition}
Hence, an auxiliary problem \cref{eq:defu} with initial condition $\by=\ye(t)$ must be solved whenever $\fe(\ye(t))$ is evaluated in \cref{eq:modeq}.
Using \cref{eq:defu,eq:deffedif} we compute
\begin{equation}\label{eq:deffeint}
\fe(\by)=\frac{1}{\eta}\int_0^\eta \bu'(s)\text{ d}s = \fs(\by)+\frac{1}{\eta}\int_0^\eta \ff(\bu(s))\text{ d}s,
\end{equation}
thus $\fe$ evaluates $\fs$ exactly and computes an average of $\ff$ along the auxiliary solution $\bu$. This average has a damping effect on $\ff$ and reduces its stiffness. In the next lemma, proved in \cite{AGR20}, we show in a particular case the effects of the average and the size of $\eta$.

\begin{lemma}\label{lemma:phif}
	Let $\ff(\by)=A_F\,\by$ with $A_F\in\R^{n\times n}$. Then
	\begin{equation}\label{eq:deffephi}
	\fe(\by)=\varphi(\eta A_F)\bf(\by),
	\end{equation}
	where 
	\begin{equation}\label{eq:defphi}
	\varphi(z)=\frac{e^z-1}{z},\ \text{ for } z\neq 0,\qquad \text{and}\qquad \varphi(0)=1.
	\end{equation}
\end{lemma}
In \cref{eq:deffephi}, the $\varphi(\eta A_F)$ term has a damping effect on $\bf$ owing to the negative definiteness of the matrix $A_F$, and the exponential-like behaviour of $\varphi(z)$. In fact, $\varphi(z)$ satisfies $\lim_{z\to -\infty}\varphi(z)=0$ and $\varphi(z)\in (0,1)$ for all $z<0$. The free parameter $\eta$ in \cref{eq:deffephi} can be used to tune this damping effect. Let $\rhoe$, and $\rhos$ be the spectral radii of $\fe$, and $\fs$, respectively. It was shown in \cite{AGR20} that $\rhoe\leq\rhos$ already holds for $\eta$ relatively small, and therefore the stiffness of \cref{eq:modeq} does not depend on $\ff$ anymore, but solely on $\fs$. In \cite{AGR20}, the authors also proved that $\ye$ is an $O(\eta)$ approximation of $\by$ and that, in some cases, $\fe$ inherits the contractivity properties of $\bf$.

\paragraph{The mRKC scheme}
\mbox{}\\
The multirate RKC scheme is nothing else than an $s$-stage RKC scheme applied to \cref{eq:modeq}, with $s$ depending solely on $\rhos$, the spectral radius of $\fs$. Whenever $\fe$ must be evaluated, it is approximated by solving the auxiliary problem \cref{eq:defu} with an $m$-stage RKC method, where $m$ depends on $\rhof$, the spectral radius of $\ff$. However, integration of \cref{eq:defu} is cheap since $\fs$ is frozen at the initial value. The tuning parameter $\eta$ is chosen so that the approximation to $\fe$ is less stiff than $\fs$, and thus the $s$-stage RKC scheme remains stable. More precisely, the number of stages $s,m$ are the smallest integers satisfying
\begin{equation}\label{eq:defsmeta}
\Dt\rhos \leq \beta s^2, \qquad \eta\rhof \leq  \beta m^2, \qquad\text{with}  \qquad  
\eta = \frac{6\Dt}{\beta s^2} \frac{m^2}{m^2-1}
\end{equation}
and $\beta = 2-4\eps/3$ (see \cref{eq:defbetas}).
One step of the mRKC scheme is then given by
\begin{equation}\label{eq:mRKC}
\begin{dcases}
\bd_0 = \bm{0}, \quad \bd_1=\mu_1\Dt\bfe(\by^n),\\
\bd_j = \nu_j\bd_{j-1}+\kappa_j\bd_{j-2}+\mu_j\Dt \bfe(\by^n+\bd_{j-1}) \quad j=2,\ldots,s,\\
\by^{n+1}= \by^n+\bd_s,
\end{dcases}
\end{equation}
where the parameters $\mu_j,\nu_j,\kappa_j$ are those of the RKC1 scheme defined in \cref{eq:coeffRKC,eq:coeffRKC1} and $\bfe(\by)$ is
%\begin{equation}\label{eq:defbfe}
%	\bfe(y)=\frac{1}{\eta}(u_\eta-y)
%\end{equation}
a numerical approximation to $\fe(\by)=(\bu(\eta)-\by)/\eta$, cf.~\eqref{eq:deffedif}. Hence, in the mRKC scheme \cref{eq:mRKC}, at each evaluation of $\bfe(\by^n+\bd_j)$ an approximation of $(\bu(\eta)-\by)/\eta$ is computed, with $\bu(\eta)$ as in \cref{eq:defu}, and $\by=\by^n+\bd_j$. This is performed by integrating \cref{eq:defu} with one $m$-stage RKC step of size $\eta$ in which each stage is divided by $\eta$ itself:
\begin{equation}\label{eq:defbfe}
\begin{dcases}
\bh_0 = \bm{0}, \quad  \bh_1 = \alpha_1(\ff(\by)+\fs(\by)),\\
\bh_j = \beta_j \bh_{j-1}+\gamma_j \bh_{j-2}+\alpha_j (\ff(\by+\eta \bh_{j-1})+\fs(\by)) \quad j=2,\ldots,m,\\
\bfe(\by) = \bh_m.
\end{dcases}
\end{equation}
Here, the parameters $\alpha_j,\beta_j,\gamma_j$ of the $m$-stage RKC scheme \eqref{eq:defbfe} are given by \cite{AGR20}
\begin{align}\label{eq:defv01}
\vz&=1+\varepsilon/m^2, & \vu&= T_m(\vz)/T_m'(\vu), & a_j&= 1/T_j(\vz) & \mbox{ for }j=0,\ldots,m&
\end{align}
and $\alpha_1 = \vu/\vz$,
\begin{align}\label{eq:defabg} 
\alpha_j&= 2\vu  a_j/a_{j-1}, & 
\beta_j&= 2\vz a_j/a_{j-1},    & 
\gamma_j&=-a_j/a_{j-2}  &
\text{for }j&=2,\ldots,m.
\end{align}
To compute $m$ and $\eta$ in \eqref{eq:defsmeta}, we insert $\eta=6\Dt m^2/(\beta s^2(m^2-1))$ into $\eta\rhof\leq \beta m^2$, and first compute $m$, then $\eta$.
The mRKC method is given by \eqref{eq:defsmeta}--\eqref{eq:defbfe} and its stability and first-order accuracy were proved in \cite{AGR20}.

\subsection{The mixed-precision multirate RKC method}\label{sec:mpmRKC}
Roughly speaking, the mRKC scheme \eqref{eq:defsmeta}--\eqref{eq:defbfe} is obtained by applying an RKC1 scheme to \cref{eq:modeq} and a second RKC1 scheme to \cref{eq:defu} whenever the right-hand side needs to be evaluated. In our mixed-precision mRKC scheme we instead apply the \moda{mixed-precision RKC1 method \cref{eq:1-order-preserving-scheme} to \cref{eq:modeq} and \cref{eq:defu}.} The resulting method then only requires one evaluation of $\ff$ and $\fs$ in high precision (per timestep), with all the subsequent evaluations performed in low precision. We now present our mixed-precision mRKC scheme, and we analyze its accuracy in \cref{sec:convmpmRKC}.

\paragraph{The mixed-precision mRKC scheme}
\mbox{}\\
Let $s,m$ and $\eta$ be as in \cref{eq:defsmeta}. One step of the mixed-precision mRKC scheme is given by
\begin{equation}\label{eq:mpmRKC}
\begin{dcases}
\hbd_0 = \bm{0}, \quad \hbd_1=\mu_1\Dt\tbfe(\hby^n),\\
\hbd_j = \nu_j\hbd_{j-1}+\kappa_j\hbd_{j-2}+\mu_j\Dt( \tbfe(\hby^n)+\hat\Delta\fe[j-1]), \quad j=2,\ldots,s,\\
\hby^{n+1}= \hby^n+\hbd_s,
\end{dcases}
\end{equation}
where $\tbfe(\by)$ is given by
\begin{equation}\label{eq:deftbfe}
\begin{dcases}
\tbh_0 = \bm{0}, \quad  \tbh_1 = \alpha_1(\ff(\by)+\fs(\by)),\\
\tbh_j = \beta_j \tbh_{j-1}+\gamma_j \tbh_{j-2}+\alpha_j (\ff(\by)+\fs(\by)+\hat\Delta\ff[j-1]), \quad j=2,\ldots,m,\\
\tbfe(\by) = \tbh_m.
\end{dcases}
\end{equation}
The $\hat\Delta\ff[j]$ are computed in low precision and must satisfy
\begin{equation}\label{eq:requirementff}
\hat\Delta\ff[j] = \ff(\by+\eta \tbh_{j})-\ff(\by)+\moda{O(\epsilon\eta)},
\end{equation}
as $\eta\to 0$, where $\epsilon\geq0$ is a small constant. For the evaluation of $\hat\Delta\ff[j]$ we can again employ the techniques described in \cref{sec:mpRKC}, with $\bf$ and $\Dt$ replaced by $\ff$ and $\eta$, respectively. The low-precision $\{\hat\Delta\fe[j]\}_{j=1}^{s-1}$ terms in \cref{eq:mpmRKC} must again satisfy (cf. \cref{eq:requirementRKC1})
\begin{equation}\label{eq:requirementmRKC}
\hat\Delta\fe[j] = \Delta\fe[j]+\moda{O(\epsilon\Dt)},\quad\forall j,
\end{equation}
where $\Delta\fe[j]=\bfe(\hby^n+\hbd_j)-\bfe(\hby^n)$. Again, we can employ the strategies from \cref{sec:mpRKC}. For instance, one can use automatic differentiation or alternatively define
\begin{equation}\label{eq:defhDfe}
\hat\Delta\fe[j] = \delta^{-1}\left( \hbfe(\hby^n+\delta\,\hbd_j)-\tbfe(\hby^n)  \right), \qquad \delta = \frac{\sqrt{u}}{\Dt}
\end{equation}
with $\hbfe(\by)$ given by
\begin{equation}\label{eq:defhbfe}
\begin{dcases}
\hbh_0 = \bm{0}, \quad  \hbh_1 = \alpha_1(\hff(\by)+\hfs(\by)),\\
\hbh_j = \beta_j \hbh_{j-1}+\gamma_j \hbh_{j-2}+\alpha_j (\hff(\by+\eta\hbh_{j-1})+\hfs(\by)) \quad j=2,\ldots,m,\\
\hbfe(\by) = \hbh_m.
\end{dcases}
\end{equation}
We prove in \cref{lemma:hatDeltafe} that if $\Dt\leq\sqrt{u}$ then $\hat\Delta\fe[j]$ defined as in \cref{eq:defhDfe,eq:defhbfe} satisfies \cref{eq:requirementmRKC} with $\epsilon=\sqrt{u}$. Condition $\Dt\leq\sqrt{u}$ is very weak since the method is intended to be used when $\Dt$ is smaller or proportional to $u\ll\sqrt{u}$.
Note that the difference between $\tbfe$ and $\hbfe$ is that in \cref{eq:deftbfe} the functions $\ff,\fs$ are evaluated once in high precision while in \cref{eq:defhbfe} they are always evaluated in low precision. Hence, $\hbfe$ is a simple low-precision evaluation of $\bfe$ (compare \cref{eq:defhbfe,eq:defbfe}), while $\tbfe$ has the lowest-order term evaluated in high precision (exactly under Assumption \ref{assumption:high_precision_exact}). Again, we remark that our mixed-precision mRKC scheme only needs one evaluation of $\ff$, and $\fs$ in high precision per timestep.

\begin{remark}\label{rem:skmrock}
	\modc{The mixed-precision mRKC scheme can be extended as well to SDEs. It is sufficient to apply the same approach explained in \cref{rem:skrock} to the multirate method SK-mROCK in \cite{AbR22b}.}
\end{remark}

\subsection{Convergence analysis}\label{sec:convmpmRKC}
We compute here the Taylor expansion of the mixed-precision mRKC scheme, as we did in \cref{thm:convmpRKCbis} for the mixed-precision RKC schemes. For the sake of brevity, we omit the convergence analysis in the sense of \cref{thm:accuracympRKC}, and the rounding error propagation analysis of the mixed-precision mRKC scheme. The results are similar as for the mixed-precision RKC1 scheme, only with added technicalities in the proofs due to the use of embedded methods. Numerically, we observe that the mixed-precision mRKC scheme is more stable than the mixed-precision RKC1 scheme due to the reduced stiffness of the right-hand side, and the decreased number of stages. Therefore, in this section we only prove that the mixed-precision mRKC scheme \cref{eq:mpmRKC,eq:deftbfe,eq:requirementff,eq:requirementmRKC} is first-order preserving by performing a Taylor expansion of the numerical solution. 

In order to prove the main convergence result, \cref{thm:convmpmRKC}, we first need a technical lemma.
\begin{lemma}\label{lemma:errhbfe}
	Let $\by\in\R^n$, $\tbfe$ as in \cref{eq:deftbfe}, and $\bfe$ as in \cref{eq:defbfe}. Then $\tbfe(\by)=\bfe(\by)+\moda{O(\epsilon\eta)}$.
\end{lemma}
\begin{proof}
	From \cref{eq:requirementff} we have $\ff(\by)+\hat\Delta\ff[j]=\ff(\by+\eta\tbh_j)+\bm{r}_{j}$ with $\Vert\bm{r}_{j}\Vert_2\leq \moda{\hat C\epsilon\eta}$ and $\hat C$ depending on the definition of $\hat\Delta\ff[j]$. Hence, subtracting \cref{eq:defbfe} from \cref{eq:deftbfe} yields
	\begin{equation}
	\begin{aligned}
	\be_0 &= \bm{0}, \qquad  \be_1 =\bm{0},\\
	\be_j &= \beta_j \be_{j-1}+\gamma_j \be_{j-2}+\alpha_j (\ff(\by+\eta\tbh_{j-1})-\ff(\by+\eta\bh_{j-1})+\bm{r}_{j-1} ) \\
	&=\beta_j \be_{j-1}+\gamma_j \be_{j-2}+\alpha_j\ff'(\by)\eta\be_{j-1} +\alpha_j(\bm{r}_{j-1}+\bm{t}_{j-1}) \quad j=2,\ldots,m,\\
	\end{aligned}
	\end{equation}
	with $\Vert\bm{t}_{j}\Vert_2\leq C(\Vert\bh_j\Vert_2^2+\Vert\tbh_j\Vert_2^2)\eta^2$ and $C$ depending on $\ff''$.
	Using \cref{lemma:perturbations} follows $\tbfe(\by)-\bfe(\by)=\be_m=O(\epsilon\eta+\eta^2)$.
\end{proof}
We are now ready to prove the main theorem, which ensures that our mixed-precision mRKC method is indeed first-order preserving.
\begin{theorem}\label{thm:convmpmRKC}
	The mixed-precision mRKC scheme \cref{eq:mpmRKC,eq:deftbfe,eq:requirementff,eq:requirementmRKC} satisfies
	\begin{equation}\label{eq:locerrmpmRKC}
	\hby^{n+1}=\hby^n+\Dt(\ff(\hby^n)+\fs(\hby^n))+O(\epsilon\Dt^2+\Dt^2).
	\end{equation}
\end{theorem}
\begin{proof}
	We proceed similarly as in \cref{thm:convmpRKCbis}. By applying \cref{lemma:perturbations} to \cref{eq:mpmRKC}, with \moda{$A=0$} and $\bm{r}_{j}=\Dt(\tbfe(\hby^n)+\hat\Delta\fe[j-1])$ for $j=1,\ldots,s$ (with $\hat\Delta\fe[0]=\bm{0}$), we obtain
	\begin{equation}
	\begin{aligned}
	\hby^{n+1}
	&= \hby^n + \sum_{j=1}^s\frac{b_s}{b_j}U_{s-j}(\oz)\mu_j\Dt(\tbfe(\hby^n)+\hat\Delta\fe[j-1])
	=\hby^n +\Dt\tbfe(\hby^n)+\Dt\sum_{j=1}^s\frac{b_s}{b_j}U_{s-j}(\oz)\mu_j\hat\Delta\fe[j-1]\\
	&=\hby^n+\Dt(\ff(\hby^n)+\fs(\hby^n))+\moda{O( \epsilon\eta\Dt+\eta\Dt +\Dt^2+\epsilon\Dt^2)},
	\end{aligned}
	\end{equation}
	where we used \cref{lemma:errhbfe}, the relation $\bfe(\by)=\ff(\by)+\fs(\by)+O(\eta)$ \cite{AGR20}, and we applied \cref{lemma:perturbations} \ref{item:boundnorm} to the last summation. We conclude using the fact that $\eta\leq 8\Dt$ (usually $\eta\ll\Dt$).
\end{proof}

\section{Numerical experiments}
\label{sec:numerical_results}
In this section we test the algorithms and theory presented in the paper. We will often compare our order-preserving mixed-precision methods to some more na\"ive mixed-precision implementations that perform all function evaluations in low precision and only vector sums and multiplications in high precision; therefore  these schemes do not converge (not even under Assumption \ref{assumption:high_precision_exact}, cf.~\cref{def:order-preserving}). In this section we will refer to these na\"ive schemes as not order-preserving or as ``standard'' mixed-precision schemes.

\subsection{Test problems and computational setup}

\subsubsection{Problem 1: Nonlinear reaction-diffusion equation}
Problem 1 is a standard nonlinear reaction-diffusion equation in $d$-dimensions with Dirichlet boundary conditions:
\begin{align}
\label{eq:reaction-diffusion}
\left\{\begin{array}{llc}
\dot{\uu}(t,\bm{x}) = \mathcal{D}\Delta \uu - h(\uu) + f_1(\bm{x}), & \bm{x}\in D=(0,1)^d, & t\in[0,T],\\
\uu(0,\bm{x}) = 1 & \bm{x}\in D=(0,1)^d, & \\
\uu(t,\bm{x}) = 1 & \bm{x} \in \partial D, & t\in[0,T],
\end{array}\right.
\end{align}
where $d\in\{1,2,3\}$, $T=1$, $\mathcal{D}=100$, $h(\uu) = \uu^2$, and $f_1(x)$ is chosen so that the exact solutions in 1D, 2D, and 3D at steady-state are
\begin{gather}
\uu_{1D}(\infty,x)      = (4\,x(1-x))^2 + 1,\quad
\uu_{2D}(\infty,\bm{x}) = (16\,xy(1-x)(1-y))^2 + 1,\\
\uu_{3D}(\infty,\bm{x}) = (64\,xyz(1-x)(1-y)(1-z))^2 + 1.\notag
\end{gather}
We pick $D_h$, the mesh of $D$, to be uniform with $d!N^d$ cells, where $N\in\N$, to be given later.

\subsubsection{Problem 2: Heat equation on a graded L-shaped domain}
Problem 2 is the classic heat equation on an L-shaped 2D domain $D_L$ with a near-singular forcing term:
\begin{align}
\label{eq:L-shaped}
\left\{\begin{array}{llc}
\dot{\uu}(t,\bm{x}) = \Delta \uu + f_2(\bm{x}), & \bm{x}\in D_L, & t\in[0,T],\\
\uu(0,\bm{x}) = 1 & \bm{x}\in D_L, & \\
\uu(t,\bm{x}) = 1 & \bm{x} \in \partial D_L, & t\in[0,T].
\end{array}\right.
\end{align}
Here $T=1$, $f_2(x)=-10\log(\theta(x,y))$ where $\theta(x,y) = 2((x-0.501)^2 + (y-0.501)^2)$, and $D_L$ is the polygon delimited by the points $\{(0, 0)$, $(1, 0)$, $(1, 0.5)$, $(0.5, 0.5)$, $(0.5, 1)$, $(0, 1)\}$. We take $D_{L}^h$, the mesh of $D_L$, to be unstructured and graded near the re-entrant corner $(0.5,0.5)$. More specifically, $D_L^h$ is constructed so that the size of its cells is roughly given by $N^{-3/2} + N^{-1}\left(1 - \exp(-20\log(2)\theta(x,y))\right)$, where the value of $N\in\N$ will be given later. We will use Problem 2 to test the multirate RKC method with degrees-of-freedom splitting presented in \cite{AGR20}. In this case we split the matrix $A$ into $A=A_F + A_S$, ($\ff=A_F$, $\fs=A_S$, cf.~\cref{sec:mRKC}) so that its stiff part $A_F$ is given by the degrees of freedom with coordinates satisfying $\theta(x,y) < 1/50$. The mesh used for this problem and the degrees-of-freedom splitting is shown in Figure \ref{fig:1} (left), shown later.

\subsubsection{Problem 3: Brussellator model}
Problem 3 is the 1D Brussellator PDE model from Chapter IV.I of the book by Hairer and Wanner \cite{HairerWanner1996}:
\begin{align}
\label{eq:brussellator}
\left\{\begin{array}{llc}
\dot{\uu}(t,x) = \alpha\Delta\uu + \uu^2v - (b+1)\uu + a, & x\in D=(0,1), & t\in[0,T],\\
\dot{v}(t,x) = \alpha\Delta v - \uu^2v + b\uu, & x\in D=(0,1), & t\in[0,T],\\
\uu(t,0) = \uu(t,1) = a,\quad v(t,0) = v(t,1) = b, &
\uu(0,x) = a + \sin(2\pi x),& v(0,x) = b.
\end{array}\right.
\end{align}
Here $T=10$, $a=1$, $b=3$, $\alpha=1/50$, and we use the same unit interval mesh as for Problem 1.

\modb{
	\subsubsection{Problem 4: $p$-Laplace diffusion model}
	Problem 4 is a nonlinear diffusion equation with $4$-Laplace diffusion operator in 1D with Dirichlet boundary conditions:
	\begin{align}
	\label{eq:p-Laplace}
	\left\{\begin{array}{llc}
	\dot{\uu}(t,x) = \nabla\cdot(\Vert\nabla\uu\Vert_2^2 \nabla\uu) + f_4(x), & x\in D=(0,1), & t\in[0,T],\\
	\uu(0,x) = \uu(t,0) = \uu(t,1) = 1 & x\in D=(0,1), & t\in[0,T],
	\end{array}\right.
	\end{align}
	where $f_4(x) = 1 + 64\exp(4-(x(1-x))^{-1})$, and we use the same unit interval mesh as for Problem 1.
}

\subsubsection{Computational setup}
Unfortunately, half precision is still not widely supported on laptop CPUs, including our own. For this reason, in our experiments all low-precision computations are emulated in software via our custom-built C++/Python precision emulator, libchopping \footnote{This code was inspired by Higham and Pranesh's work \cite{HighamPranesh2019} and by Milan Kl\"ower's emulators in Julia \url{https://github.com/milankl?tab=repositories}.} \cite{libchopping}. Number format emulation is extremely expensive and our software relies on vectorization, OpenMP and MPI so as to improve efficiency. Nevertheless, emulated operations are slower than for native formats and we are thus unable to provide actual CPU timings for our algorithms. Consequently, we can only rely on the theoretical estimates of \cref{sec:costanalysis}.

We solve the test problems via the finite element method by using continuous piecewise-linear elements on simplices. We employ the open-source finite element software FEniCS \cite{LoggEtAl2012} for the assembly of the finite element matrices involved, and Python numpy \cite{numpy}, scipy.sparse \cite{scipy} and libchopping \cite{libchopping} linear algebra kernels for the computations. We use mass-lumping to avoid solving the mass-matrix linear system at every timestep, and we take the linear part of the discretized PDE, $A$, to be the stiffness matrix scaled on the left by the inverse lumped mass matrix. So as to better squeeze $A$ into the range of the low-precision format (cf.~Remark \ref{rem:matrix_squeezing}), we divide $A$ by its max norm $||A||_{\max} = \max_{ij}|A_{ij}|$ before rounding it (we multiply back by $||A||_{\max}$ in the high-precision format after each matrix-vector product). 

\begin{remark}[Matrix squeezing]
	\label{rem:matrix_squeezing}
	When doing computations in reduced precision one must be careful about underflow/overflow, especially when working with formats with a small range such as fp16 (cf.~Table \ref{tab:precision}). We remark that there exist matrix-squeezing algorithms \cite{Higham2019MatrixSqueezing} that first rescale and then round a matrix in such a way that the available range is fully exploited. These algorithms typically work by applying a two-sided diagonal scaling to a matrix $A$ so as to obtain a new matrix $\tilde{A}=D_1AD_2$ (here $D_1$, $D_2$ are diagonal matrices) that better fits into the available range. The advantage of working in mixed precision is that it is possible to compute matrix-vector products in low precision using $\tilde{A}$ and then rescale the result back, e.g.~as $(D_1^{-1})\tilde{A}(D_2^{-1}\bm{b})$, where $\tilde{A}$ is applied in low precision and the remaining (linear-cost) operations are performed so that the result is stored in high precision. Similar techniques are also available for nonlinear terms, see e.g.~\cite{klower2021fluid} for an application of these techniques to weather simulation.
\end{remark}

\subsection{Numerical results}

\subsubsection{Stability}\label{sec:numericalstability}
We start by looking at the numerical stability of our mixed-precision methods. As previously mentioned, establishing any theoretical stability result is extremely complicated since rounding errors disrupt both the smoothness of the solution and the spectrum of $A$. For instance, whenever $\Dt A$ has small nonpositive eigenvalues, these can be perturbed by rounding errors and made positive, thus amplifying the error. At the same time, a solution affected by noise due to rounding errors loses its smoothness, which prevents us from obtaining sharp \emph{a priori} error bounds.

We remark that these theoretical issues arise even when computations are performed in high precision. However, we know that when computations are performed in double precision the situation in practice is much different, and numerical methods for ODEs work as they should. We now demonstrate that the same holds in practice for our mixed-precision methods, and that the low-precision computations we use do not impact numerical stability. For this purpose, we take Problem 1 in 2D with $h(\uu),f_1\equiv 0$ (i.e.~the standard heat equation with no forcing), homogeneous Dirichlet boundary conditions, and $\uu(0,\bm{x})=(16\,xy(1-x)(1-y))^2$, and we investigate how the ratio $||\hby^n||_2/||\by^0||_2$ evolves as the mesh size is refined for fixed $\Dt$ for different values of $s$ across a larger timespan of $T=8$. We look at order-preserving mixed-precision RKC implementations, and at a simpler not order-preserving version in which all function evaluations are performed in low precision. \moda{We use schemes \cref{eq:1-order-preserving-scheme,eq:hybrid-scheme}, and employ a double/bfloat16 format combination. We fix $\mathcal{D}=50$, we take $s=2^{5+i}$ and $N=2^{2+i}$, for $i=0,\dots,4$, and we set $\Dt=s^2/\rho$ for RKC1 and $\Dt=\frac{1}{2}\beta^2(s,\frac{2}{13})/\rho$ (cf.~\eqref{eq:defbetas}) for RKC2. Results are shown in \cref{fig:0}. Clearly, both the standard (dotted lines) and the new (dashed lines) mixed-precision algorithms are stable in practice. In fact, they are as stable as the high-precision implementation results (which we are not showing). These results show that our mixed-precision schemes do not seem to impact stability, even for large $s$ and $\Dt$.}

\begin{figure}[h!]
	\centering
	\begin{subfigure}{0.49\textwidth}
		\centering
		\includegraphics[width=0.9\textwidth]{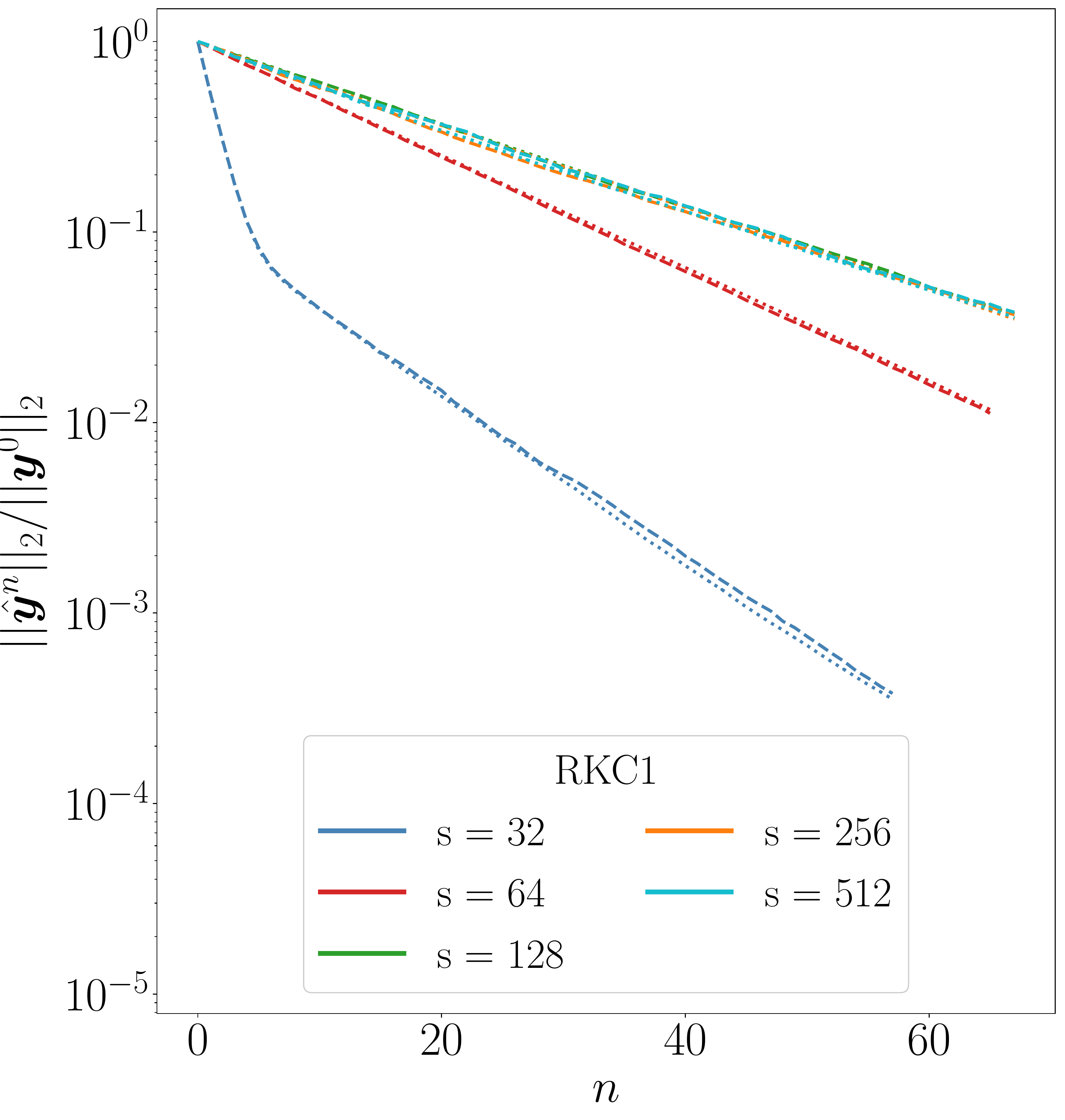}
	\end{subfigure}
	\begin{subfigure}{0.49\textwidth}
		\centering
		\includegraphics[width=0.9\textwidth]{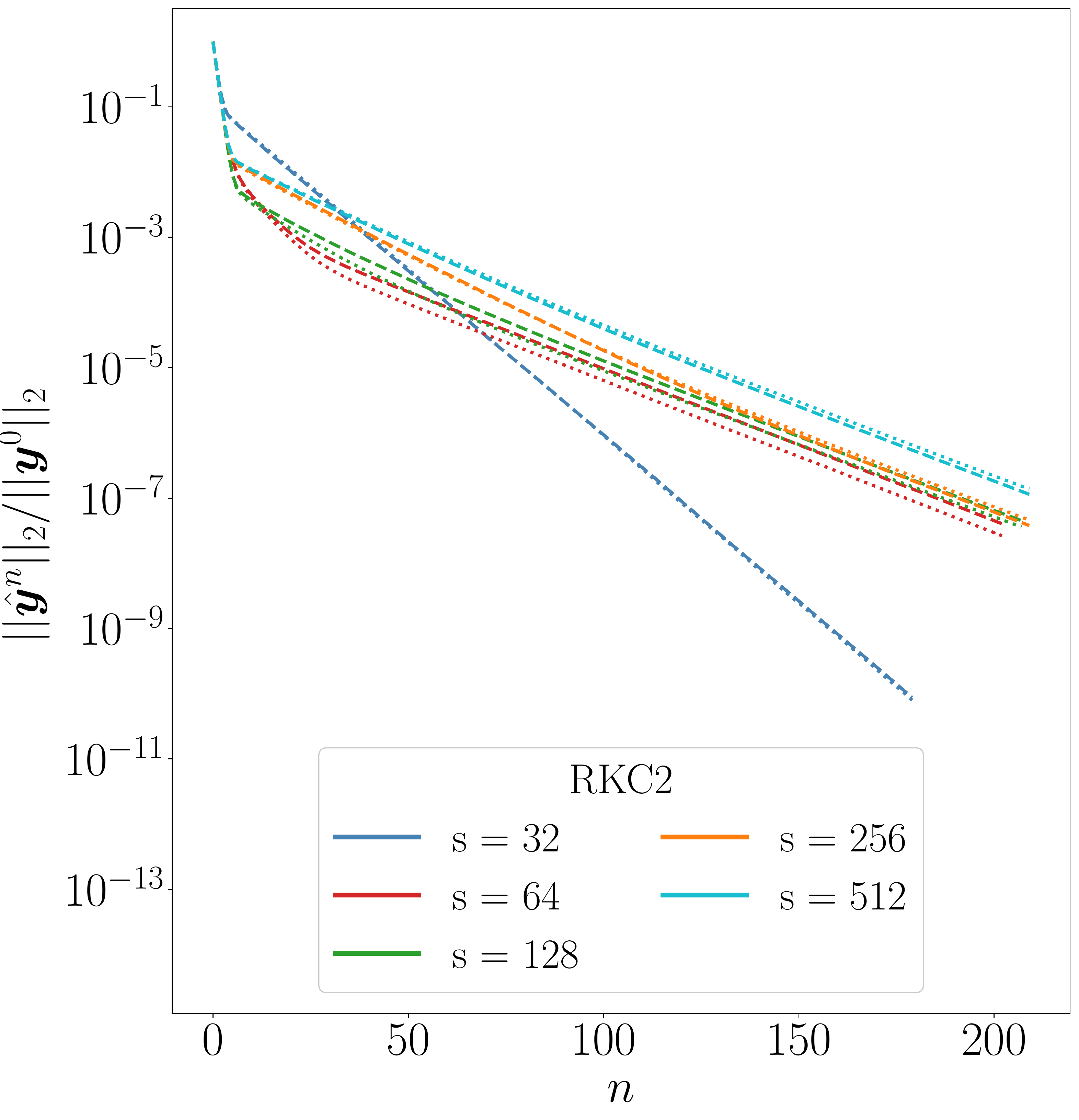}
	\end{subfigure}
	\caption{\textit{Behaviour of the $2$-norm of the numerical solution of the heat equation in mixed precision with RKC vs number of timesteps for different values of $s$. Dotted lines correspond to results obtained using a non-order-preserving implementation, while dashed lines correspond to our mixed-precision algorithms \eqref{eq:1-order-preserving-scheme} and \eqref{eq:hybrid-scheme}. A decaying trend follows the behavior of the true solution of the PDE and denotes stability.}}
	\label{fig:0}
\end{figure}

\begin{figure}[h!]
	\centering
	\begin{subfigure}{0.49\textwidth}
		\centering
		\includegraphics[width=0.9\textwidth]{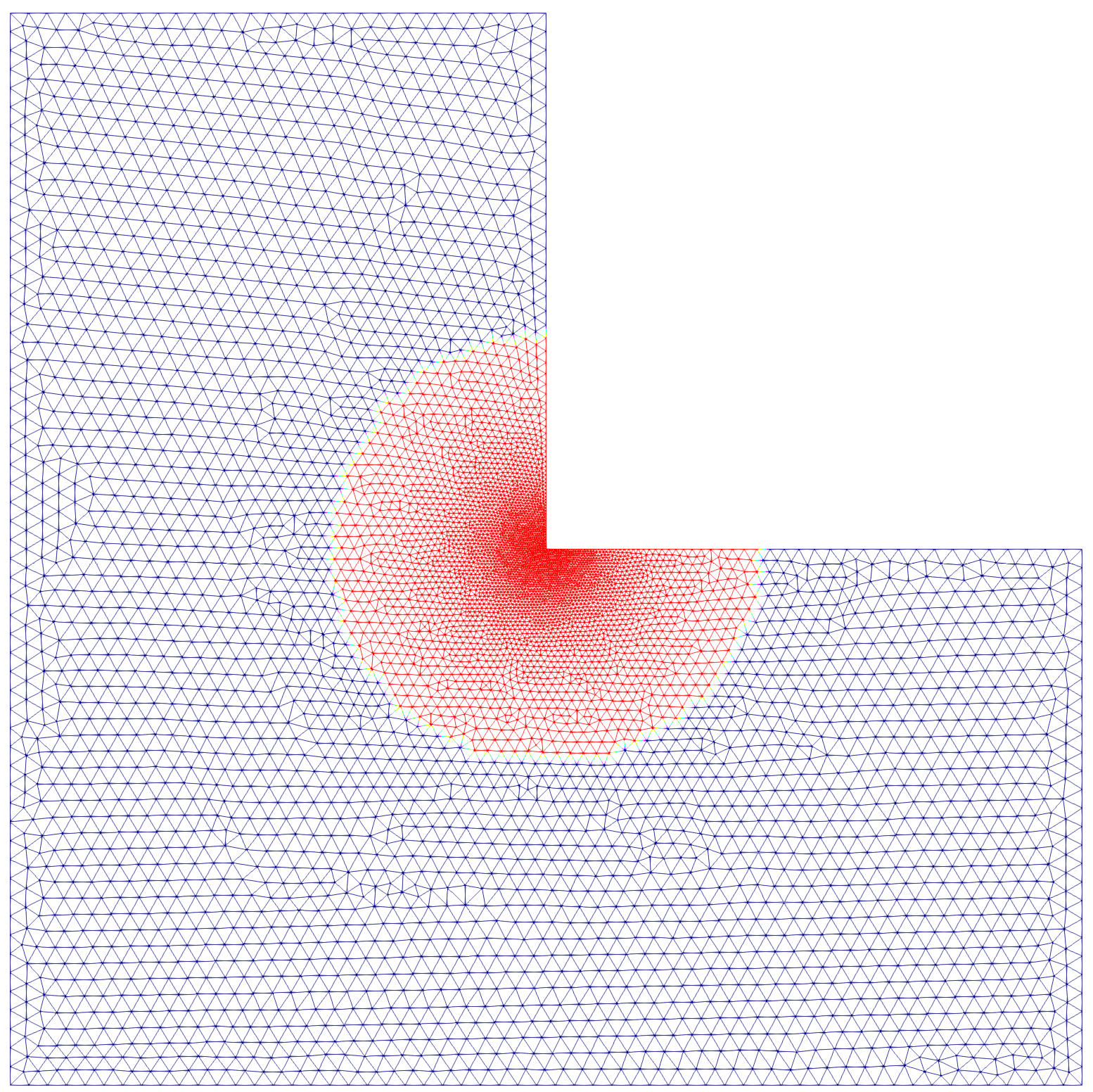}
	\end{subfigure}
	\begin{subfigure}{0.49\textwidth}
		\centering
		\includegraphics[width=0.9\textwidth]{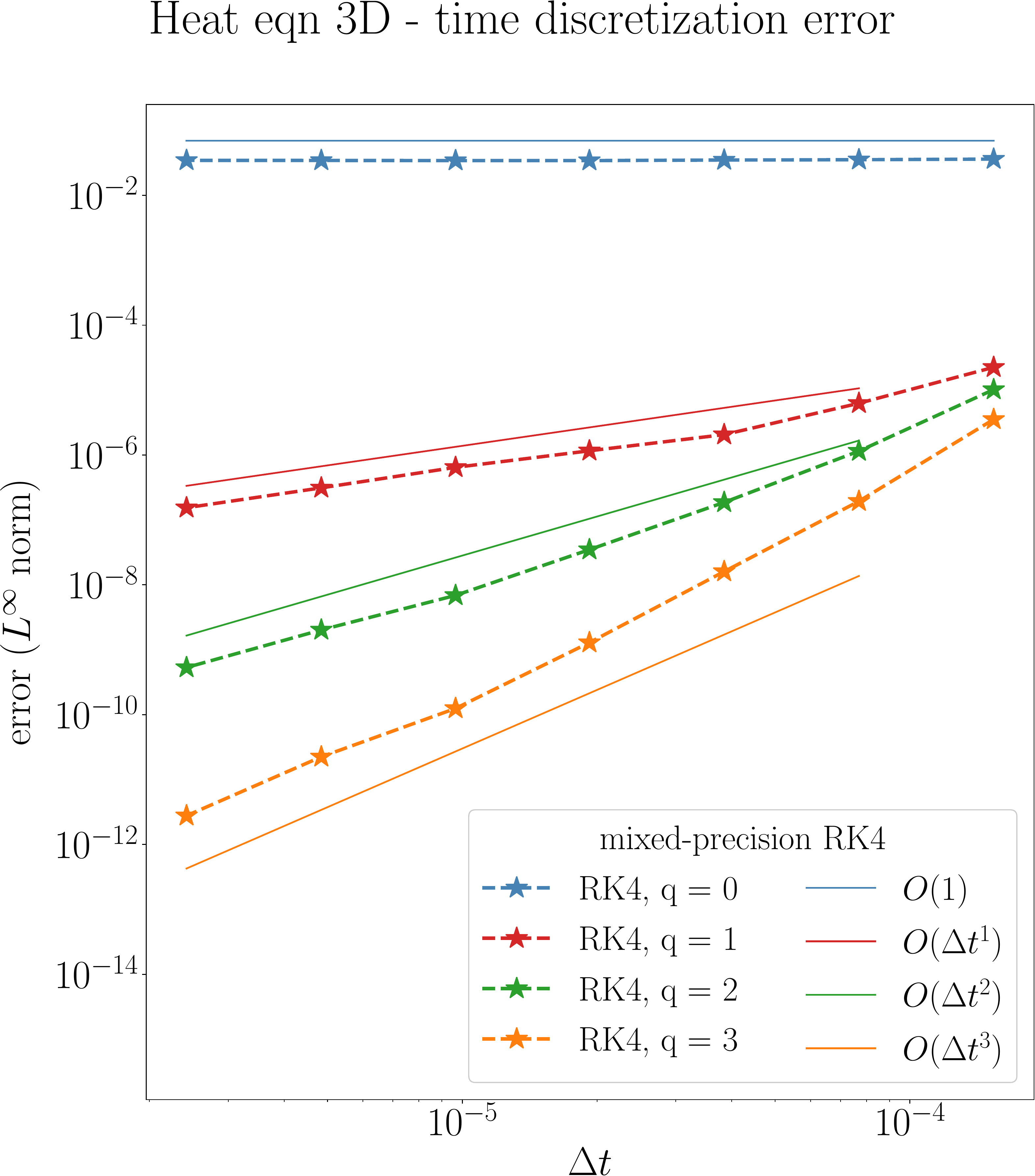}
	\end{subfigure}
	\caption{\textit{On the left, the graded mesh used for Problem 2. For this problem, the stiff part of $A$ is given by the entries corresponding to the dofs near the re-entrant corner (colored in red). On the right, the convergence behaviour of the $q$-order preserving mixed-precision RK4 for the nonlinear heat equation in 3D as $q$ varies.}}
	\label{fig:1}
\end{figure}

\subsubsection{Convergence}
In order to sanitize our results from spatial discretization errors we compare the numerical solutions $\hat{\uu}^n_h\approx \uu|_{t=n\Delta t}$ and $\hat{v}^n_h\approx v|_{t=n\Delta t}$ against the much more accurate solutions $\bar{\uu}^n_h$ and $\bar{v}^n_h$ obtained by using the same spatial discretization, but in exact arithmetic and with the classic fourth-order method RK4 with a much smaller timestep $\Delta t_{\text{ref}} = \min(2\rho^{-1}, \Delta t/4)$.

We first verify that our methods are indeed order-preserving by estimating what their order of convergence is in practice. For this purpose, we take the maximum $L^\infty$ norm over time, defined as (for Problems 1, 2, and 3 respectively)
\begin{align}
\label{eq:error_measures}
\max_{n}||\hat{\uu}^n_h - \bar{\uu}^n_h||_{L^\infty(D)},\qquad \max_{n}||\hat{\uu}^n_h - \bar{\uu}^n_h||_{L^\infty(D_L)},\qquad \max_{n}\max\left(||\hat{\uu}^n_h - \bar{\uu}^n_h||_{L^\infty(D)},\ ||\hat{v}^n_h - \bar{v}^n_h||_{L^\infty(D)}\right).
\end{align}
We also consider relative errors computed by dividing the quantities in \eqref{eq:error_measures} by the roundoff unit $u$ of the low-precision format.

\paragraph{Linear problems}
We begin by considering a linear problem and investigating the effect of changing $q$, the number of high-precision matvecs. For this purpose, we take Problem 1 with $h(\uu)\equiv0$ (i.e.~the standard heat equation) in 3D with $N=2^5$ (i.e.~a mesh of $196608$ tetrahedra), which we solve using the $q$-order-preserving RK4 method constructed following \eqref{eq:q-order-pres-RK}. We choose $q\in\{0,1,2,3\}$ and we show the results in Figure \ref{fig:1} (right). As we can see, taking $q$ high-precision matrix-vector products as in \eqref{eq:q-order-pres-RK} is sufficient to recover $q$-th order convergence. Furthermore, for $q>0$ we can distinguish a pre-asymptotic regime in which the mixed-precision method converges with full order (i.e.~$4$th) before tailing off to the lower rate. The length of this regime appears to be growing with $q$. This suggests that even when $q\geq 1$ is much smaller than the full order $p$ it is still be possible to reduce the error significantly. Indeed we observe that even with only $q=1$ the error is already reduced by up to $4$ orders of magnitude with respect to the non-order preserving method.

\begin{remark}
	In some scenarios an $O(u)$ limiting accuracy is sufficient. However, the limiting error constant is problem-dependent, and in the worst-case it might be comparable to $u^{-1}$. In this case we suggest that a simple 1-order-preserving mixed-precision method would be enough to avoid losing all accuracy. We remark that for some problems it is possible to ensure near-$O(u)$ limiting accuracy without ever resorting to higher precision. Techniques such as compensated summation \cite{higham1993accuracy,klower2021fluid} or stochastic rounding \cite{croci2021stochastic,CrociGilesSR2020,ConnollyHighamMary2020} can be used for this purpose.
\end{remark}

\paragraph{Nonlinear problems}
We now solve the nonlinear Problems \moda{1-4} with our methods and estimate their empirical convergence order. We consider the order preserving RKC schemes \eqref{eq:1-order-preserving-scheme} and \eqref{eq:hybrid-scheme} with $s=16$ \moda{(Problems 1-3), and $s=32$ (Problem 4),} and the nonlinear terms evaluated according to Scenario 1 (high-precision evaluations of $\bm{g}$) and Scenario 2 (low-precision evaluations of the Jacobian), and the order-preserving multirate RKC method \eqref{eq:mpmRKC}. \moda{Problem 4 does not have a linear term, and we therefore only evaluate the Jacobian according to Scenario 2, and we do not use the multirate RKC scheme for this problem.} We investigate the behaviour of the time-discretization error as the timestep is refined by taking bfloat16 and double precision as the low- and high-precision formats respectively. With these methods and formats, we solve Problem 1 in 2D with $N=2^6$ (Figure \ref{fig:2}), Problem 2 with  $N=2^5$ (Figure \ref{fig:3}), Problem 3 with $N=2^6$ (Figure \ref{fig:4})\moda{, and Problem 4 with $N=2^5$ (Figure \ref{fig:4bis})}. In Figures \ref{fig:2}, \ref{fig:3}, \ref{fig:4}\moda{, and \ref{fig:4bis},} we plot the relative error (the error measures in \eqref{eq:error_measures} divided by $u$) versus $\Delta t$. We note how the order-preserving schemes successfully ensure that the full order of the method is preserved even when almost all function evaluations are performed in low precision. On the other hand, we observe that the error of the non order-preserving schemes stagnates at roughly $10u$ (two digits of accuracy) for Problems 1 and 2, $100u$ (less than one digit of accuracy!) for Problem 3\moda{, and $0.1u$ for Problem 4}. The order-preserving methods are up to 2-8 orders of magnitude more accurate.

\begin{figure}[h!]
	\begin{subfigure}{0.32\textwidth}
		\centering
		\includegraphics[width=\textwidth]{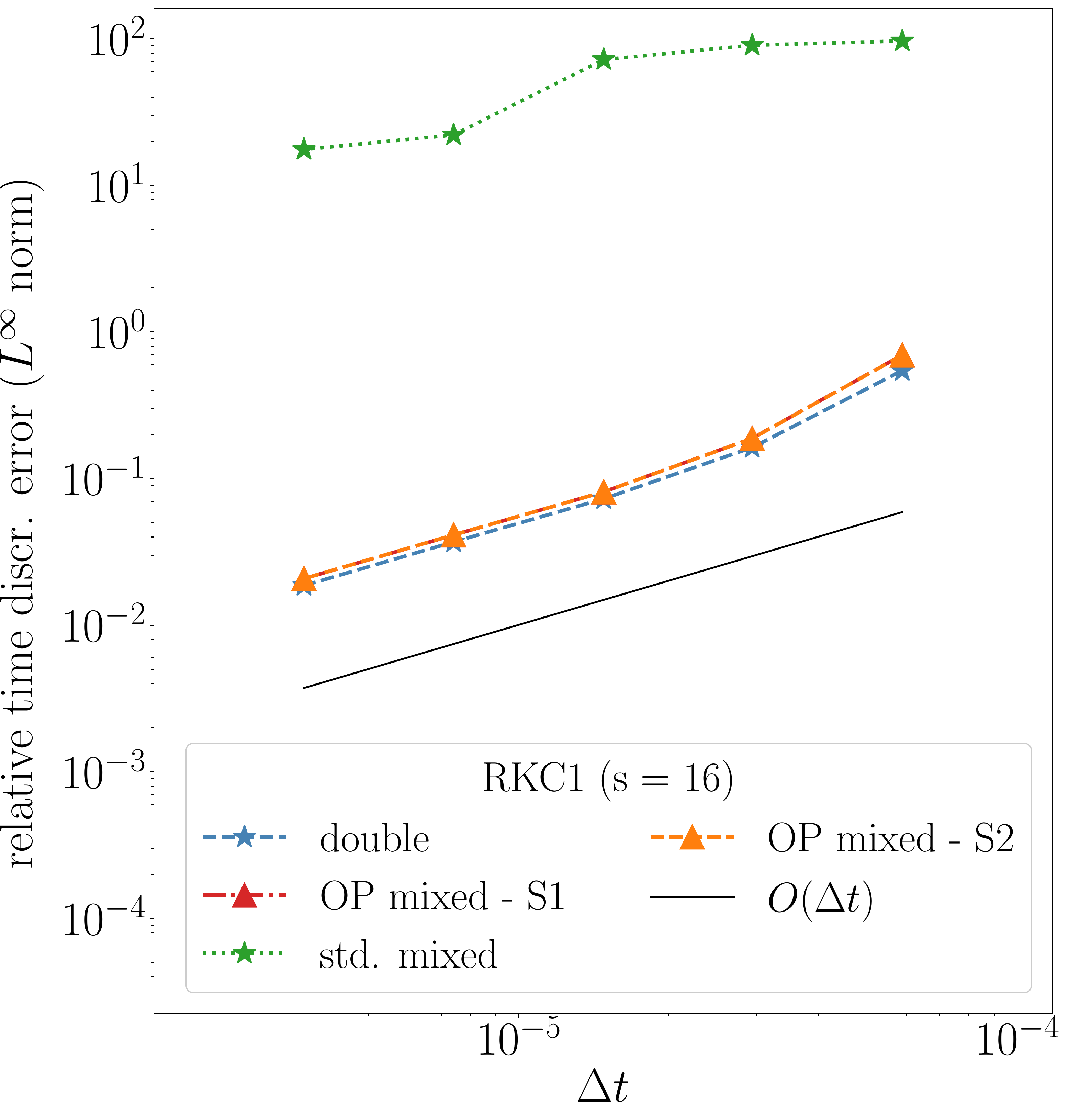}
	\end{subfigure}
	\begin{subfigure}{0.32\textwidth}
		\centering
		\includegraphics[width=\textwidth]{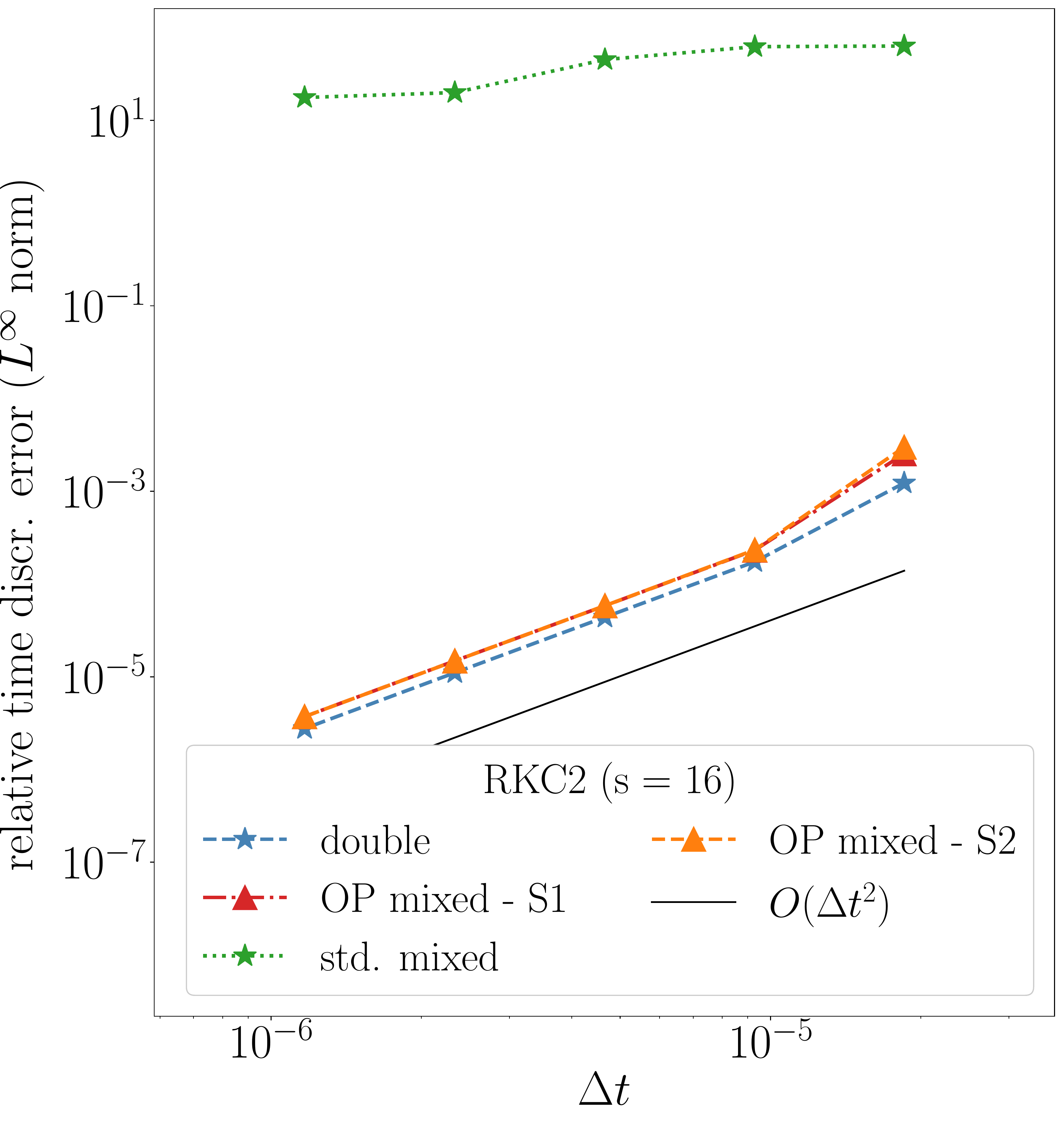}
	\end{subfigure}
	\begin{subfigure}{0.32\textwidth}
		\centering
		\includegraphics[width=\textwidth]{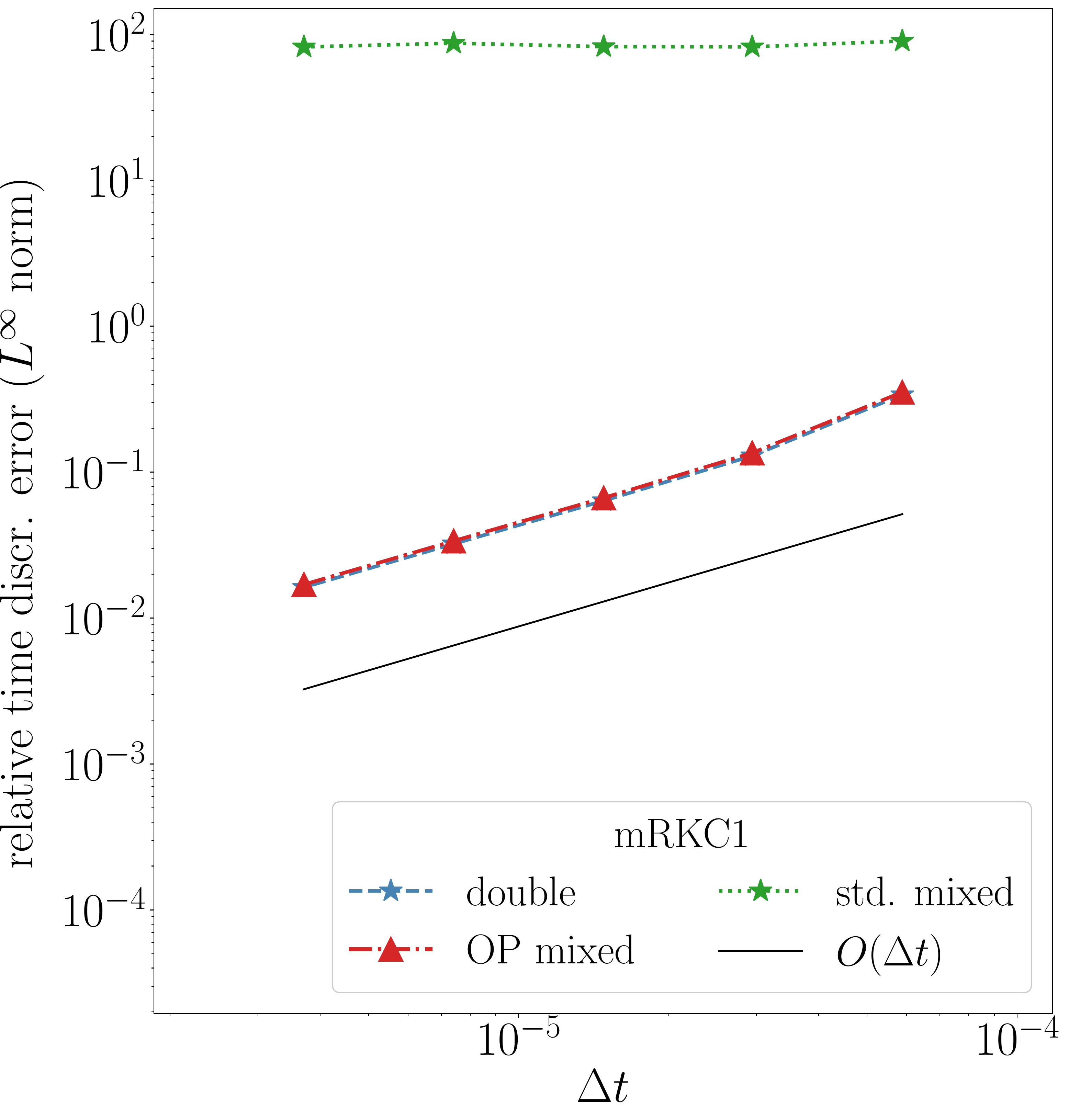}
	\end{subfigure}
	\caption{\textit{Mixed-precision RKC: convergence for the nonlinear heat equation in 2D. ``OP'' stands for order-preserving and S1 and S2 stand for Scenario 1 and 2 respectively, while ``std.~mixed'' indicates a standard non-order preserving mixed-precision implementation.}}
	\label{fig:2}
\end{figure}

\begin{figure}[h!]
	\begin{subfigure}{0.32\textwidth}
		\centering
		\includegraphics[width=\textwidth]{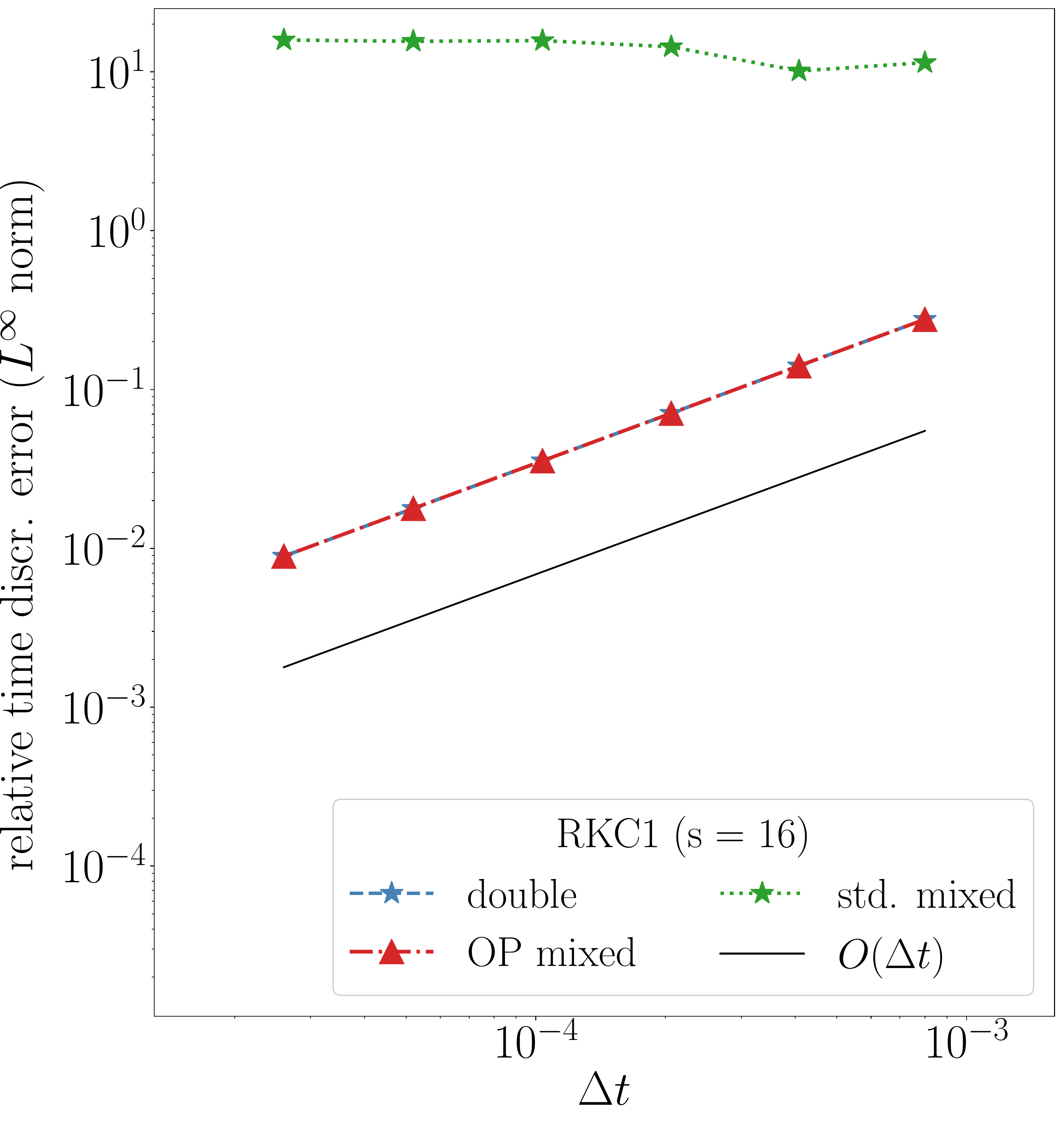}
	\end{subfigure}
	\begin{subfigure}{0.32\textwidth}
		\centering
		\includegraphics[width=\textwidth]{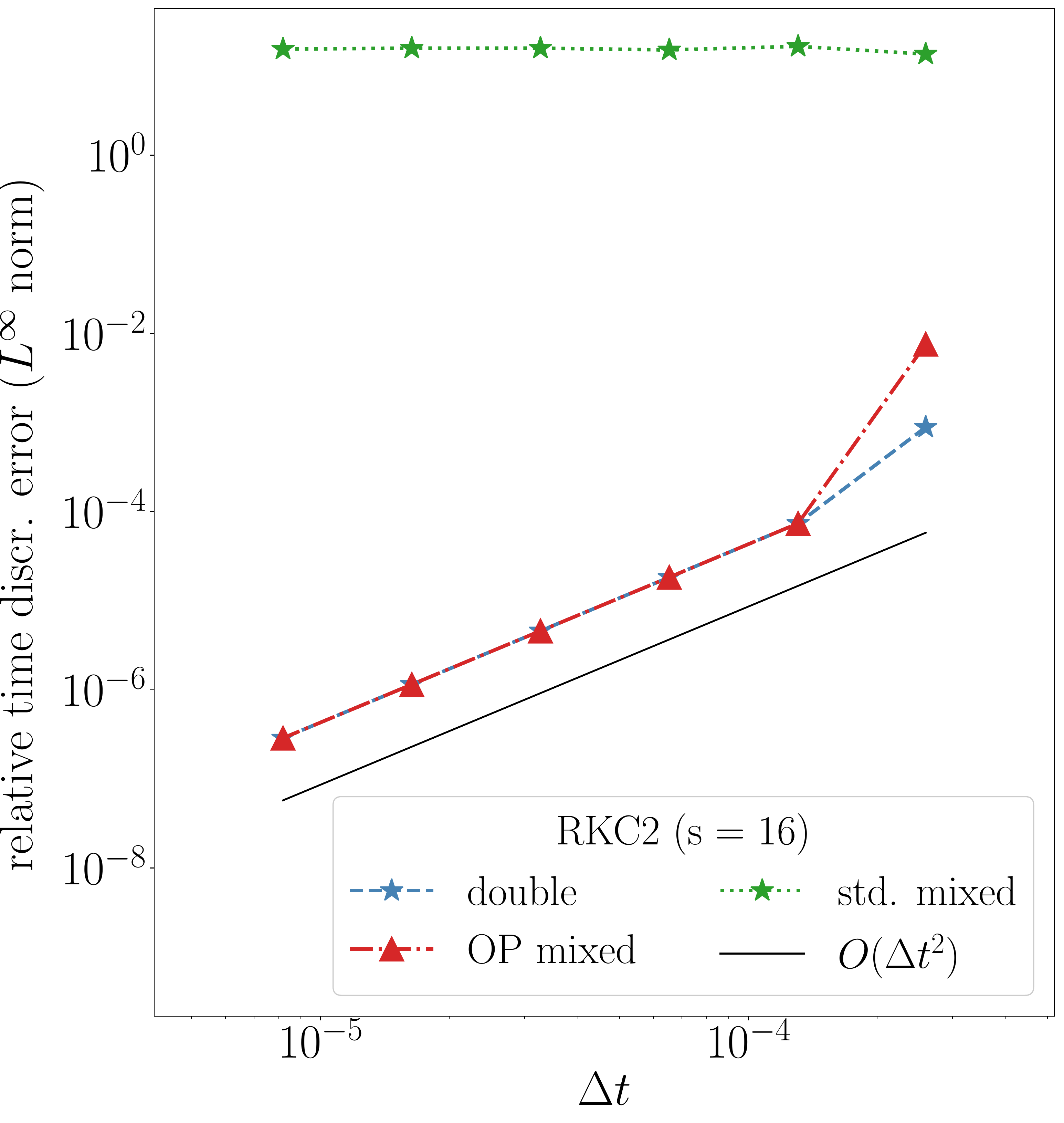}
	\end{subfigure}
	\begin{subfigure}{0.32\textwidth}
		\centering
		\includegraphics[width=\textwidth]{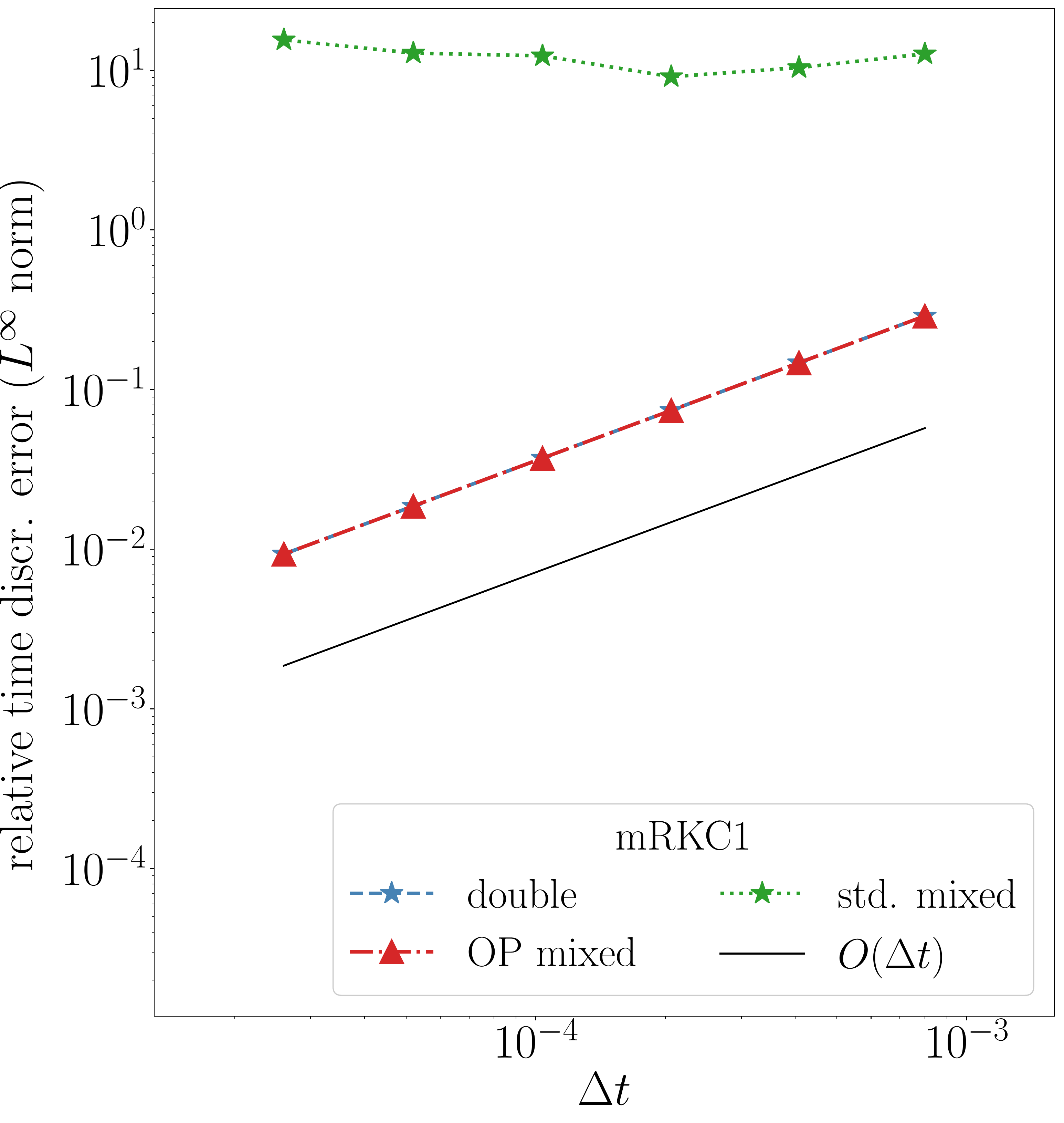}
	\end{subfigure}
	\caption{\textit{Mixed-precision RKC: time-discretization error convergence for the heat equation in the L-shaped domain. ``OP'' stands for order-preserving and ``std.~mixed'' indicates a standard non-order preserving mixed-precision implementation.}}
	\label{fig:3}
\end{figure}

\begin{figure}[h!]
	\begin{subfigure}{0.32\textwidth}
		\centering
		\includegraphics[width=\textwidth]{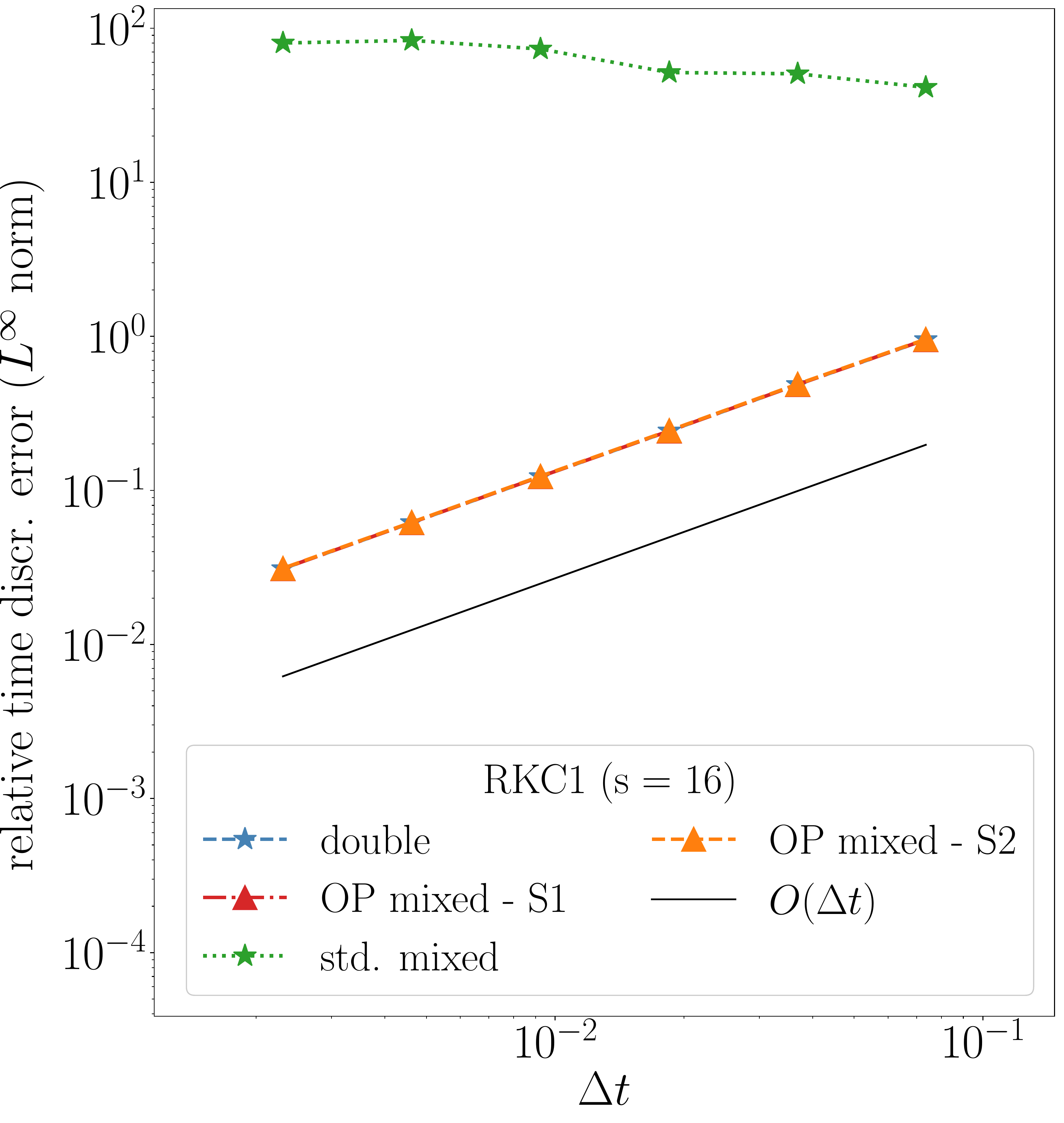}
	\end{subfigure}
	\begin{subfigure}{0.32\textwidth}
		\centering
		\includegraphics[width=\textwidth]{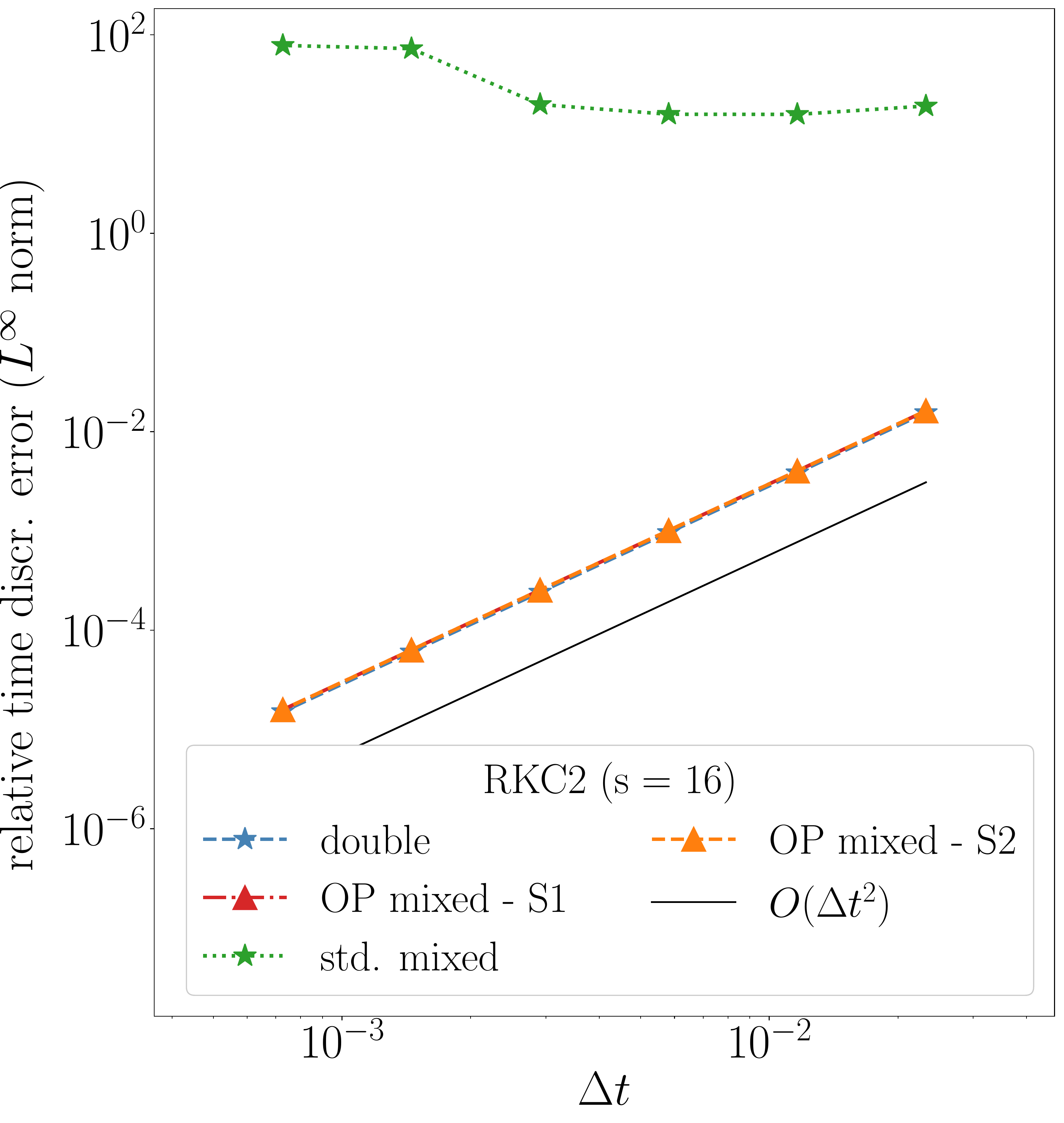}
	\end{subfigure}
	\begin{subfigure}{0.32\textwidth}
		\centering
		\includegraphics[width=\textwidth]{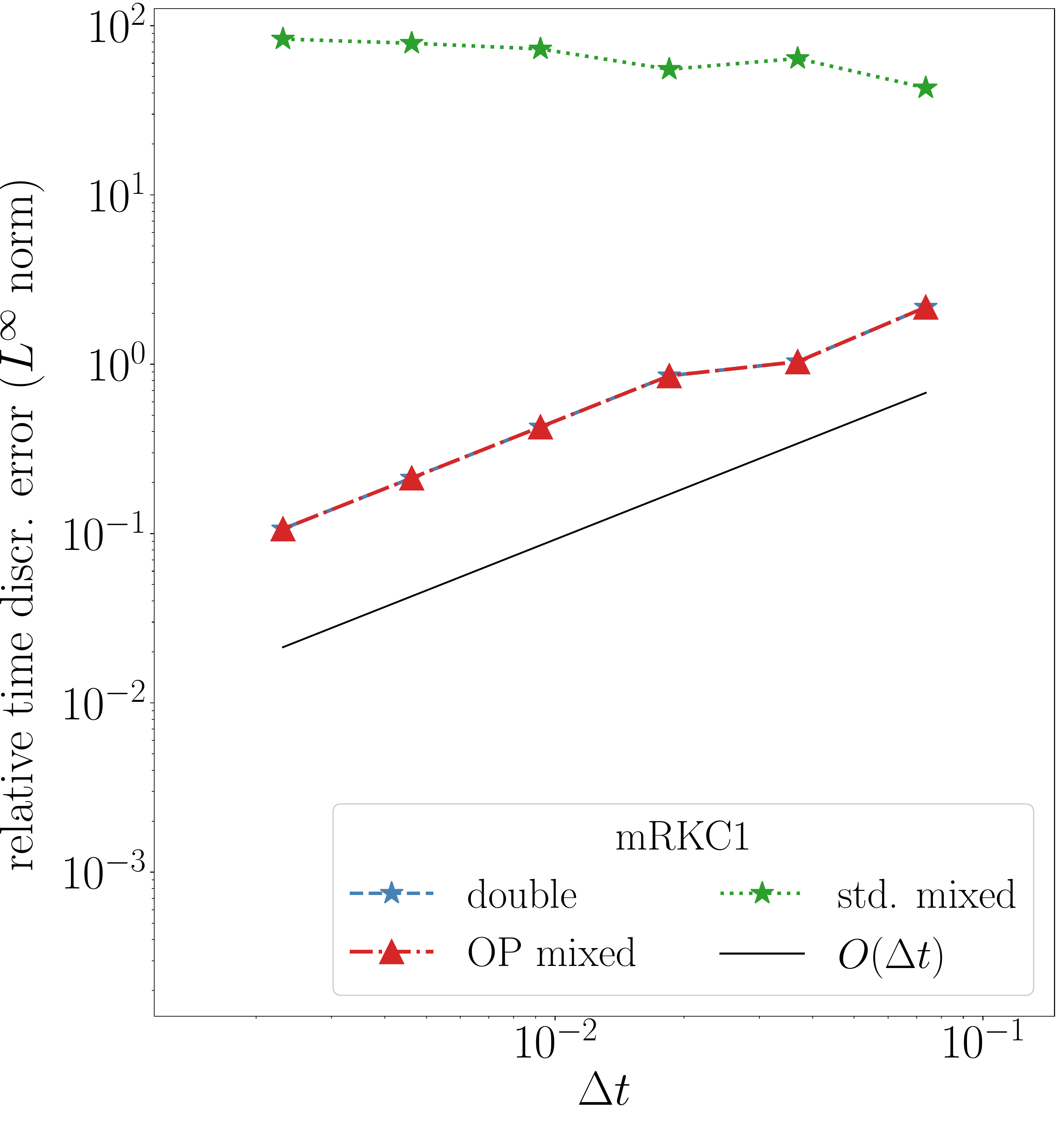}
	\end{subfigure}
	\caption{\textit{Mixed-precision RKC: time-discretization error convergence for the Brussellator. ``OP'' stands for order-preserving and S1 and S2 stand for Scenario 1 and 2 respectively, while ``std.~mixed'' indicates a standard non-order preserving mixed-precision implementation.}}
	\label{fig:4}
\end{figure}

\begin{figure}[h!]
	\centering
	\begin{subfigure}{0.36\textwidth}
		\centering
		\includegraphics[width=\textwidth]{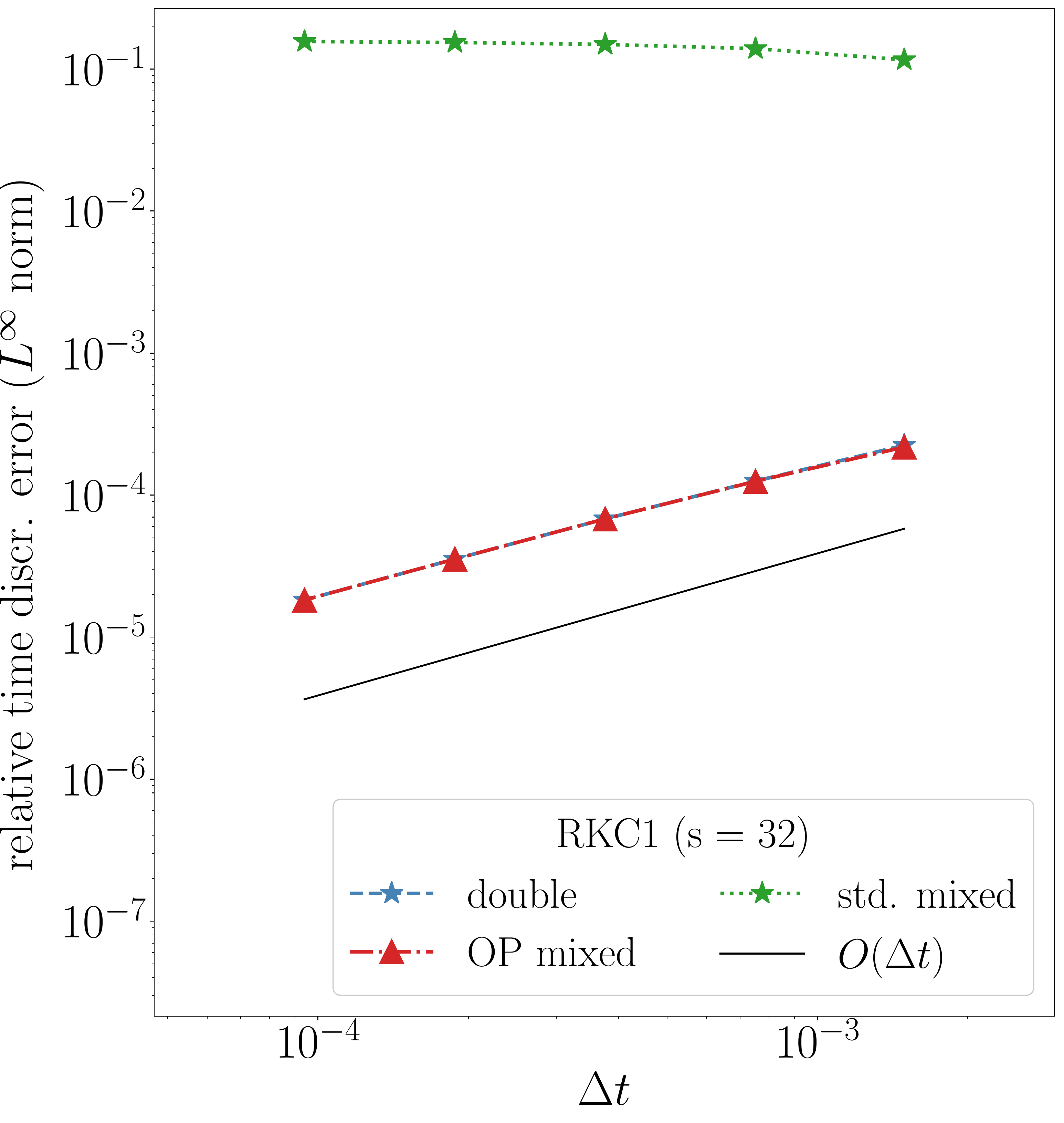}
	\end{subfigure}
	\begin{subfigure}{0.36\textwidth}
		\centering
		\includegraphics[width=\textwidth]{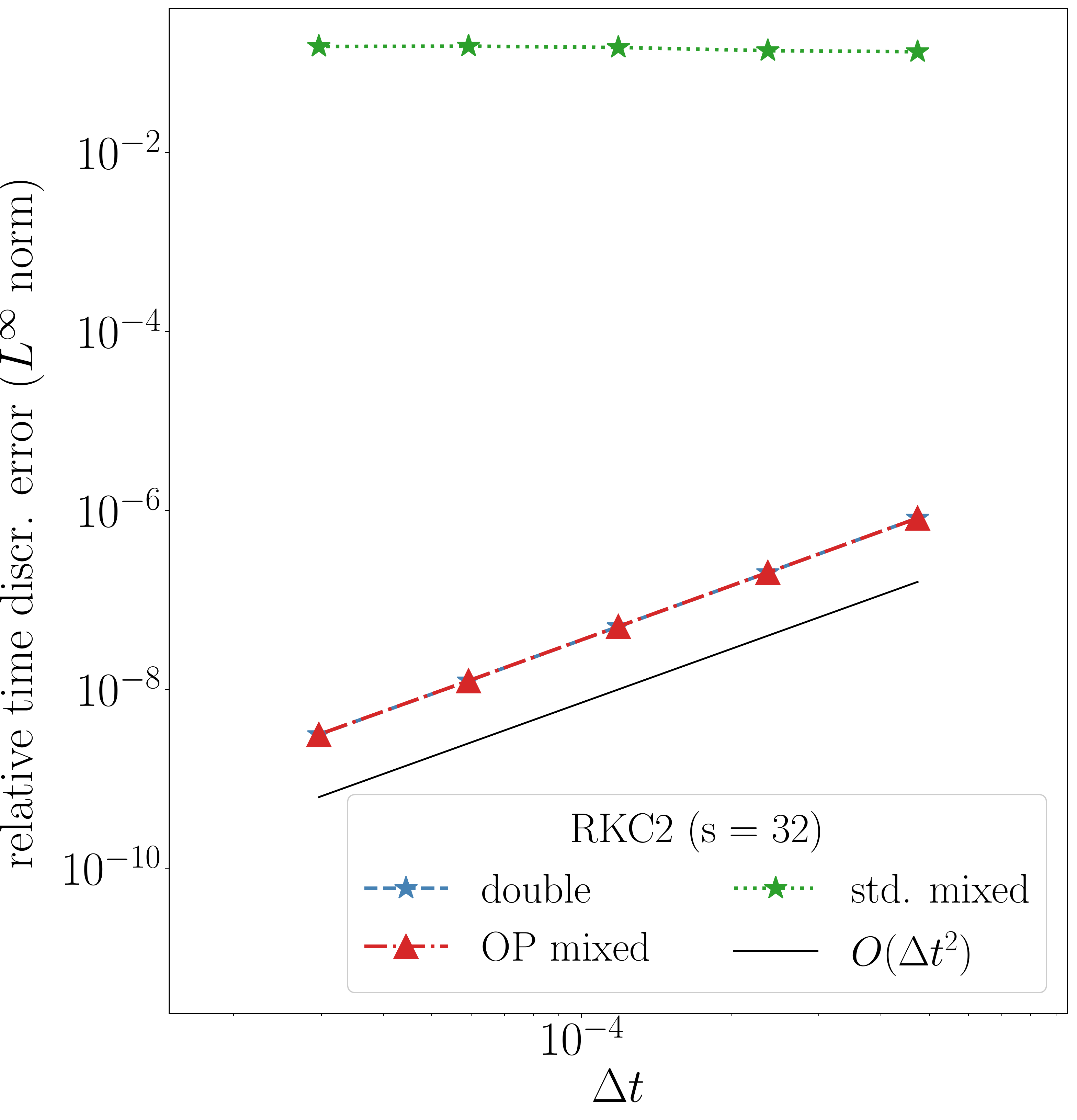}
	\end{subfigure}
	\caption{\textit{\moda{Mixed-precision RKC: time-discretization error convergence for the $4$-Laplace diffusion problem. ``OP'' stands for order-preserving, while ``std.~mixed'' indicates a standard non-order preserving mixed-precision implementation.}}}
	\label{fig:4bis}
\end{figure}

\subsubsection{Number of stages vs error}
We now investigate the stability of our methods as the number of stages increases. More specifically, we look at how the number of stages $s$ affects the global rounding error in the mixed-precision schemes. We only consider Problem 1 in 2D with $N=2^4$, Problem 3 with $N=2^8$, \moda{and Problem 4 with $N=2^6$,} and we fix $\Dt\rho = s^2$ for RKC1 and $\Dt\rho = \frac{1}{2}\beta^2(s,\frac{2}{13})$ for RKC2. We estimate how the global rounding error compares to the time-discretization error of the schemes when run in exact arithmetic. Namely, we look at the ratios:
\begin{align}
\dfrac{\max_{n}||\hat{\uu}^n_h -  \uu^n_h||_{L^{\infty}(D)}}{\max_{n}||\uu^n_h -  \bar{\uu}^n_h||_{L^{\infty}(D)}},\qquad \max\left(\dfrac{\max_{n}||\hat{\uu}^n_h -  \uu^n_h||_{L^{\infty}(D)}}{\max_{n}||\uu^n_h -  \bar{\uu}^n_h||_{L^{\infty}(D)}},\ \dfrac{\max_{n}||\hat{v}^n_h -  v^n_h||_{L^{\infty}(D)}}{\max_{n}||v^n_h -  \bar{v}^n_h||_{L^{\infty}(D)}}\right),
\end{align}
for Problems 1 \moda{and 4, and Problem} 3 respectively. Here $\uu^n_h$ and $v^n_h$ are obtained by running the same numerical scheme as for $\hat{\uu}^n_h$ and $\hat{v}^n_h$, only in fully high precision. \moda{We consider the mixed-precision (bfloat16/double) RKC schemes, and we evaluate the nonlinear terms according to both Scenario 1 (high-precision evaluations of $\bm{g}$) and Scenario 2 (low-precision evaluations of the Jacobian), except for Problem 4 for which we can only follow Scenario 2.} The purpose of this test is to assess the magnitude of rounding errors vs discretization errors and to validate the results in \cref{sec:mixed-precision-ESRK} by estimating in practice the range of values of $s$ for which our mixed-precision schemes are stable.

\begin{figure}[b!]
	\centering
	\begin{subfigure}{0.36\textwidth}
		\centering
		\includegraphics[width=\textwidth]{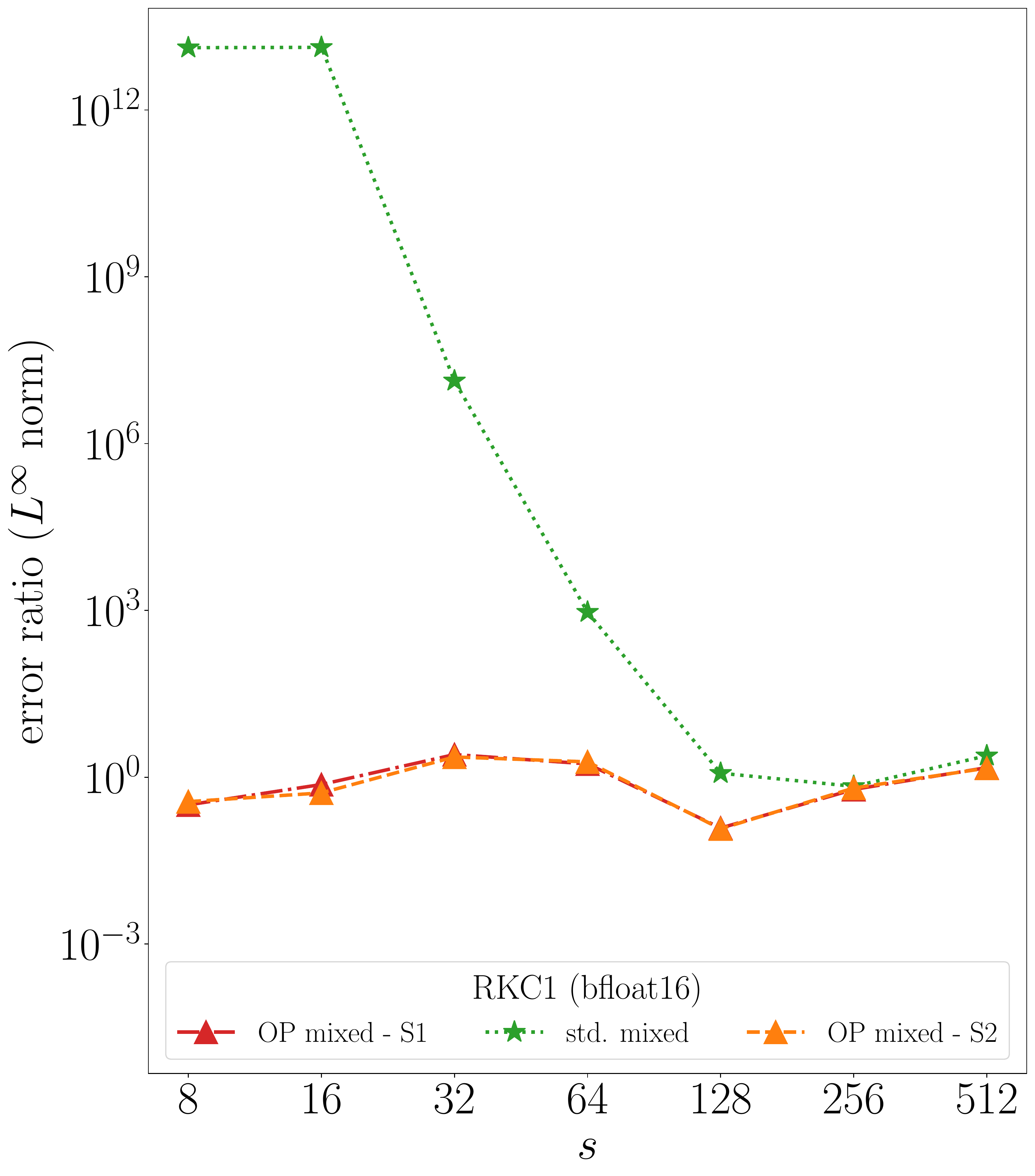}
	\end{subfigure}
	\hspace{6pt}
	\begin{subfigure}{0.36\textwidth}
		\centering
		\includegraphics[width=\textwidth]{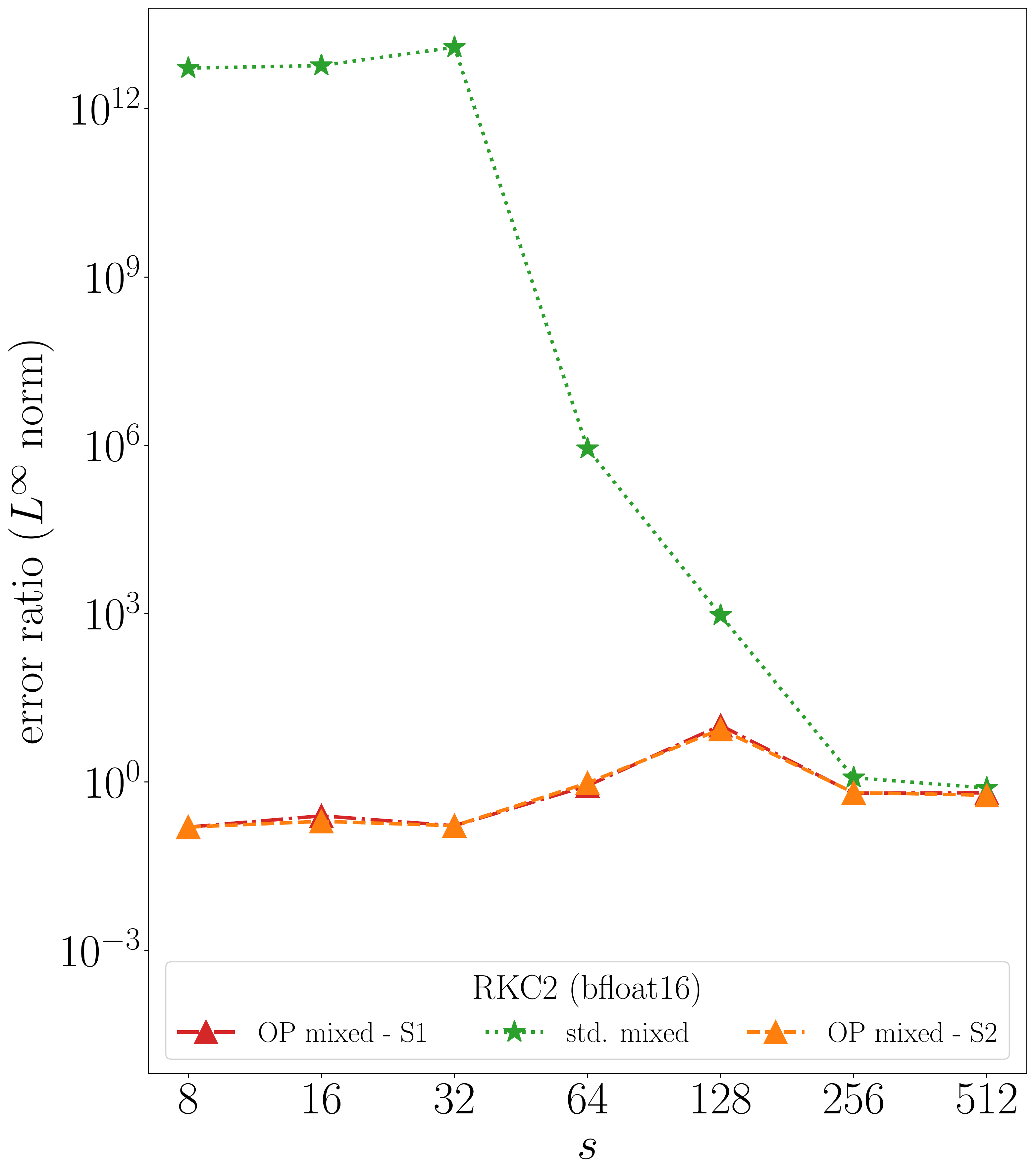}
	\end{subfigure}
	\caption{\textit{Mixed-precision RKC: ratio between rounding error and time-discretization error vs number of stages for the nonlinear heat equation in 2D. ``OP'' stands for order-preserving and S1 and S2 stand for Scenario 1 and 2 respectively, while ``std.~mixed'' indicates a standard non-order preserving mixed-precision implementation.}}
	\label{fig:5}
\end{figure}

Results are shown in Figure \ref{fig:5} (Problem 1), Figure \ref{fig:6} (Problem 3)\moda{, and Figure \ref{fig:7} (Problem 4)}. We observe that while the rounding error of the mixed-precision schemes is of roughly the same order of the time-discretization error (or even smaller), the fully low-precision scheme is orders of magnitude larger for small to moderate values of $s$. Nevertheless, the timestep, and consequently the discretisation error, increase with $s$ and eventually the accuracy of the fully low-precision scheme and its mixed-precision counterpart become comparable. \moda{We remark that in these experiments our mixed-precision RKC methods were stable for all values of $s$, suggesting that the mixed-precision RKC schemes are actually more robust than our theory predicts.}

\begin{figure}[h!]
	\centering
	\begin{subfigure}{0.36\textwidth}
		\centering
		\includegraphics[width=\textwidth]{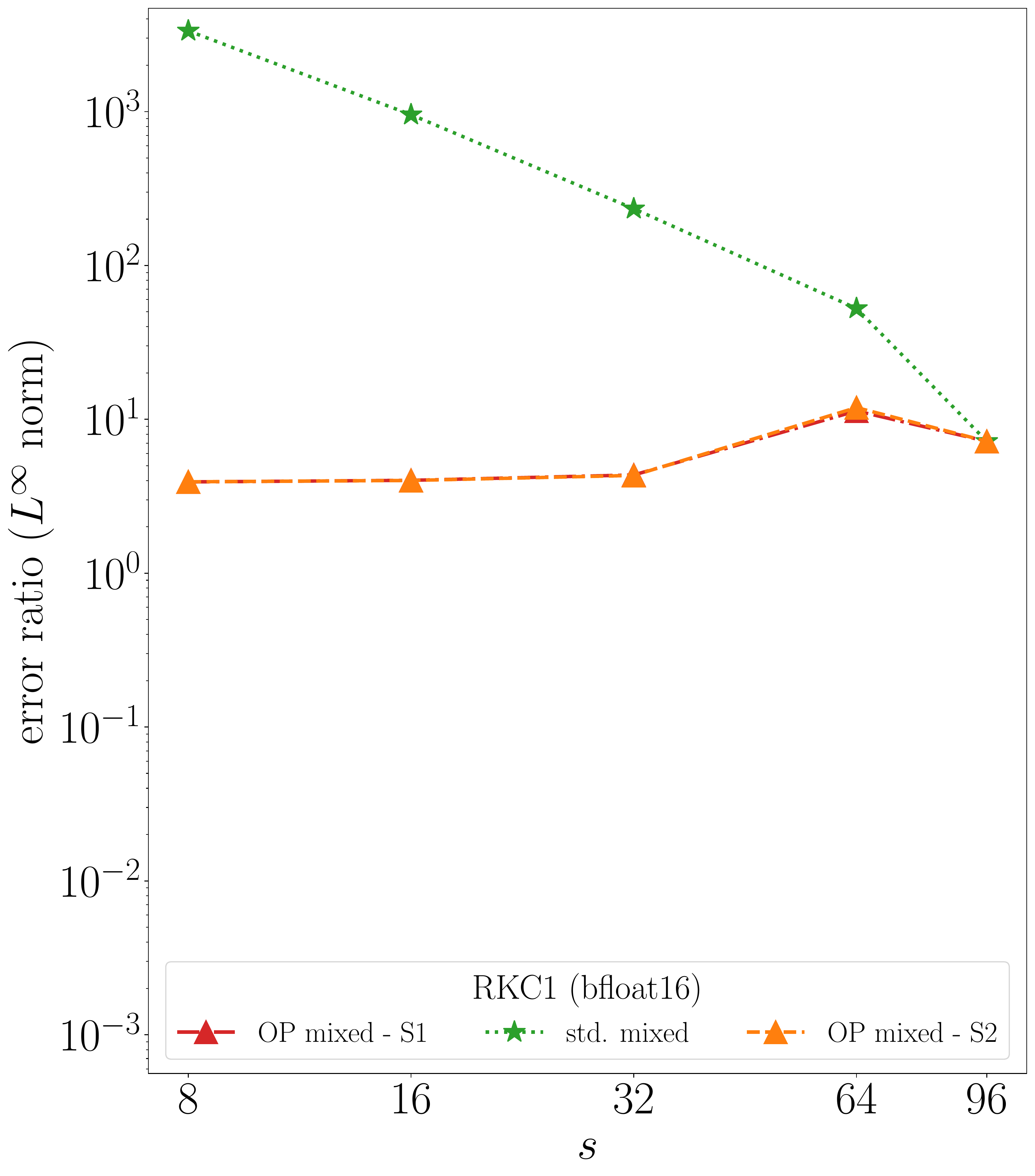}
	\end{subfigure}
	\hspace{6pt}
	\begin{subfigure}{0.36\textwidth}
		\centering
		\includegraphics[width=\textwidth]{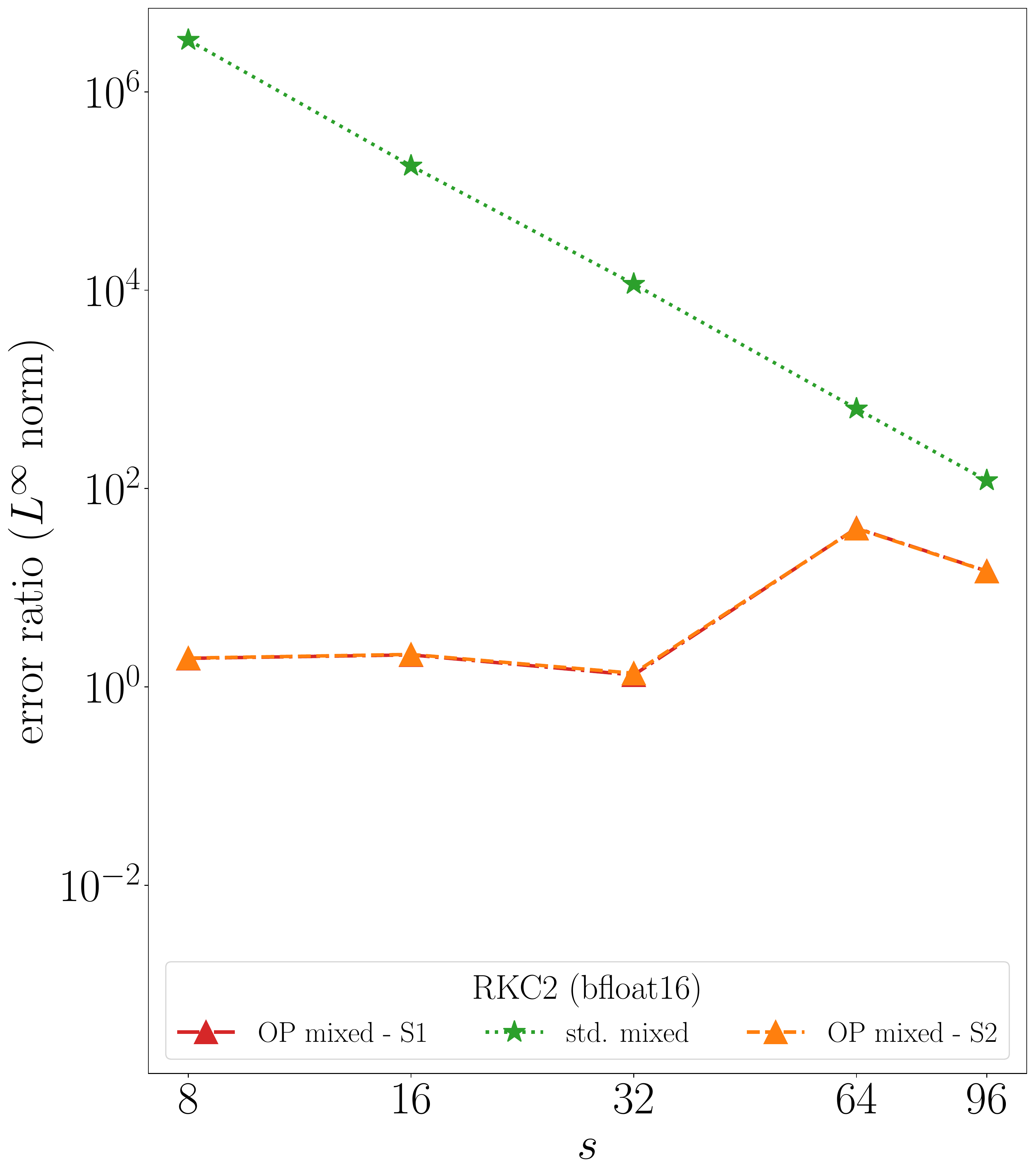}
	\end{subfigure}
	\caption{\textit{Mixed-precision RKC: ratio between rounding error and time-discretization error vs number of stages for the Brussellator. ``OP'' stands for order-preserving and S1 and S2 stand for Scenario 1 and 2 respectively, while ``std.~mixed'' indicates a standard non-order preserving mixed-precision implementation.}}
	\label{fig:6}
	\vspace{-12pt}
\end{figure}
\begin{figure}[h!]
	\centering
	\begin{subfigure}{0.36\textwidth}
		\centering
		\includegraphics[width=\textwidth]{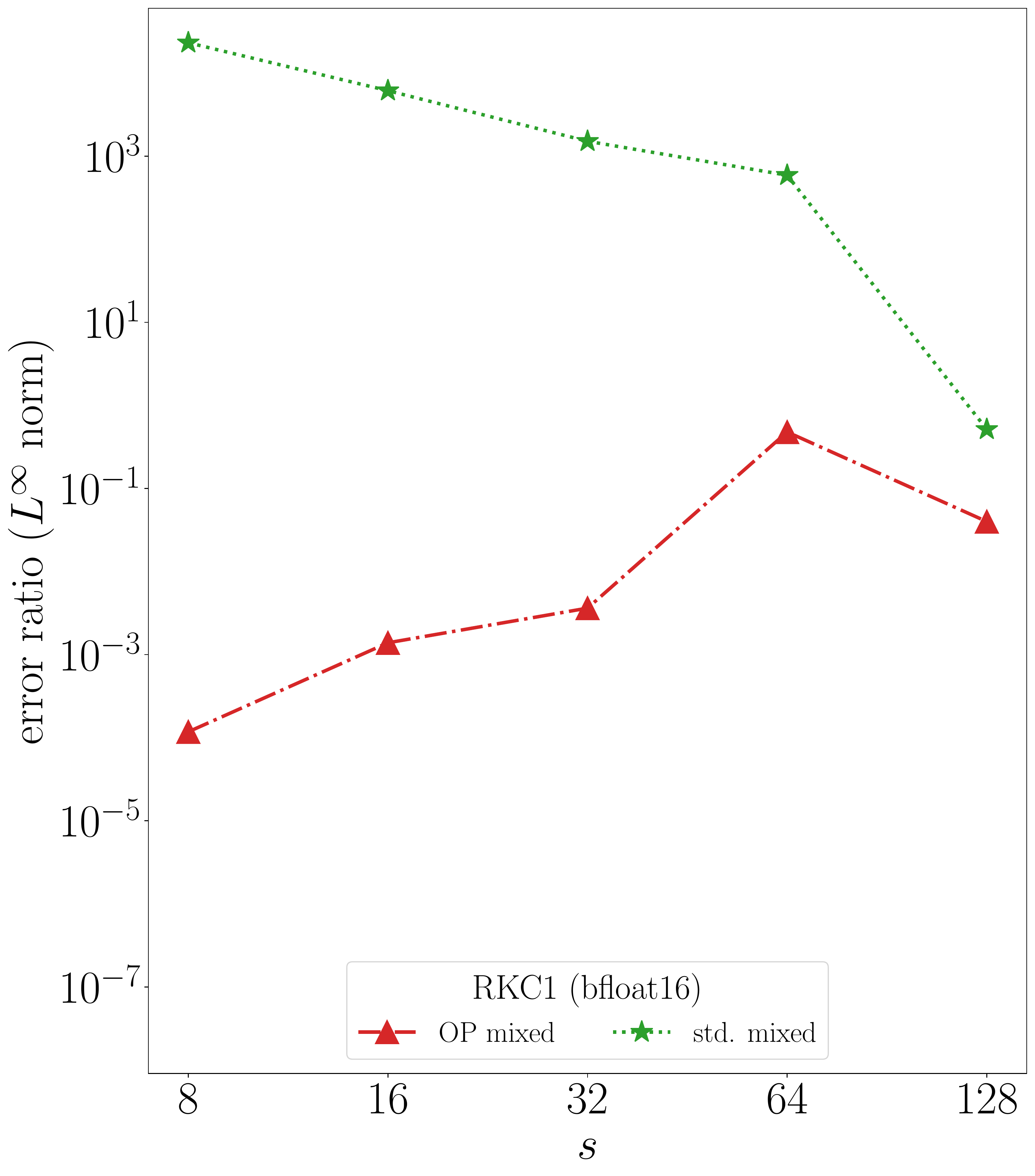}
	\end{subfigure}
	\hspace{6pt}
	\begin{subfigure}{0.36\textwidth}
		\centering
		\includegraphics[width=\textwidth]{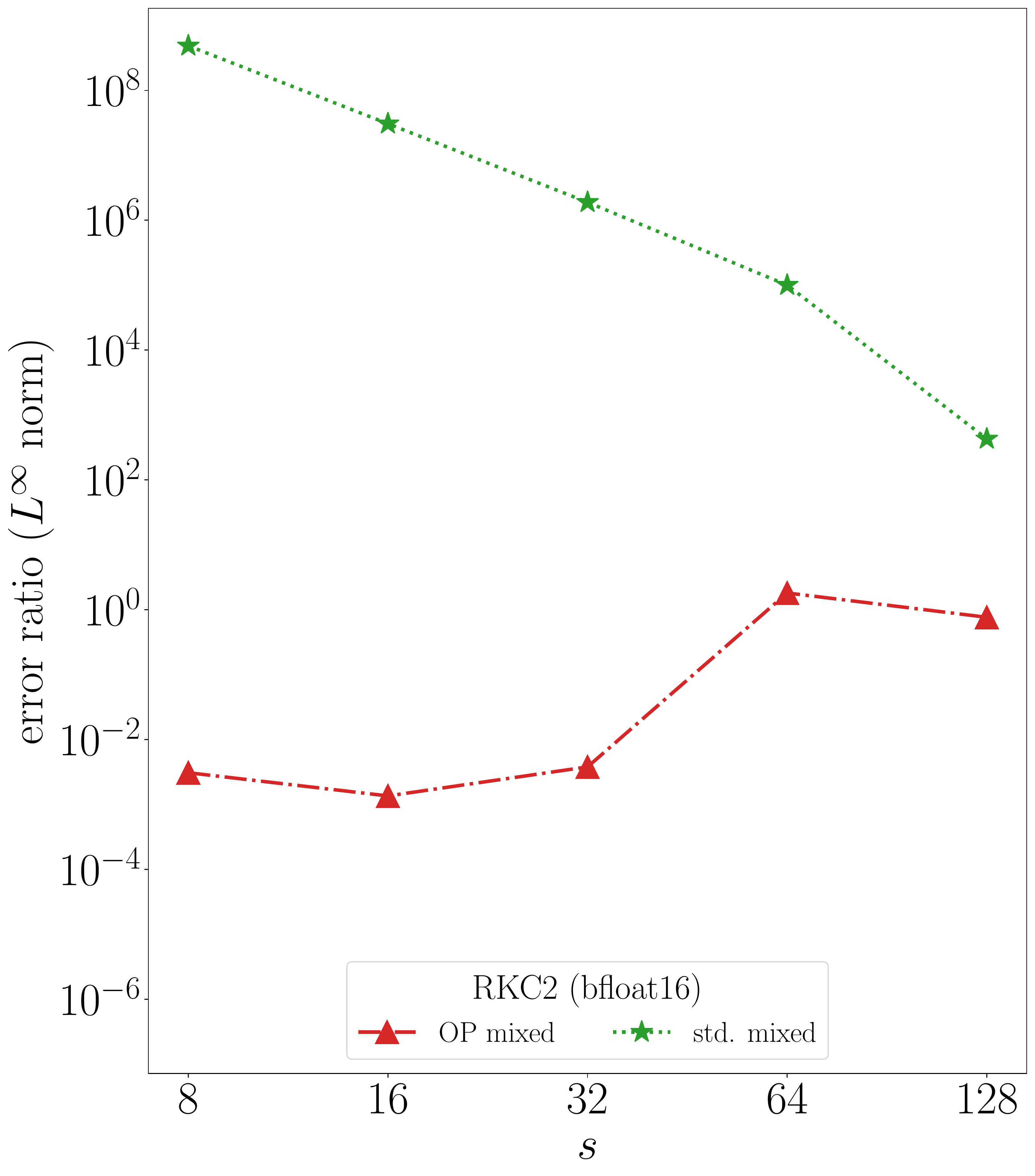}
	\end{subfigure}
	\caption{\textit{\moda{Mixed-precision RKC: ratio between rounding error and time-discretization error vs number of stages for the $4$-Laplace diffusion problem. ``OP'' stands for order-preserving, while ``std.~mixed'' indicates a standard non-order preserving mixed-precision implementation.}}}
	\label{fig:7}
	\vspace{-12pt}
\end{figure}

\subsubsection{Space-time convergence}
We conclude the section by testing the convergence in time and space of our mixed-precision methods. \moda{We only consider Problems 1 and 4, and we measure the error according to
	\begin{align}
	u^{-1}\max\limits_{n}||\hat{\uu}_h^n - \uu(t^n,\cdot)||_{L^\infty(D)},\quad\quad\quad u^{-1}||\hat{\uu}_h^{T/\Dt} - \uu(\infty,\cdot)||_{L^\infty(D)},
	\end{align}
	where we use the error measure on the left for Problem 1 and the one on the right for Problem 4. Problem 1 is the only problem for which an exact solution is available so we can actually compute the maximum error across timesteps. For Problem 4 we instead only compute the steady-state error by comparing the mixed-precision solution with a very accurate ($N=2^{12}$) steady-state solution obtained by solving the steady-state $4$-Laplacian problem in double precision. For RKC1 we fix $s=16$ and $\Dt \rho=s^2$, i.e.~$\Dt=O(N^{-2})$, while for RKC2 we vary $N_i=2^{i}$, $s_i=\lceil8\sqrt{N_i}\rceil$, and $\Dt\rho(N_i)=\frac{1}{2}\beta_2(s_i,\frac{2}{13})$ (cf.~\eqref{eq:defbetas}) for $i=2,\dots,6$ (for Problem 1), and for $i=6,\dots,10$ (for Problem 4), i.e.~$\Dt=O(N^{-1})$. Results are shown in Figures \ref{fig:9} and \ref{fig:10}.} The convergence behavior of our order-preserving methods under Strategies 1 and 2 is the same as for the schemes run fully in double precision. On the other hand, the non-order preserving methods stagnate and are unable to reduce the total error below a given limiting threshold.  

\begin{figure}[h!]
	\centering
	\begin{subfigure}{0.38\textwidth}
		\centering
		\includegraphics[width=\textwidth]{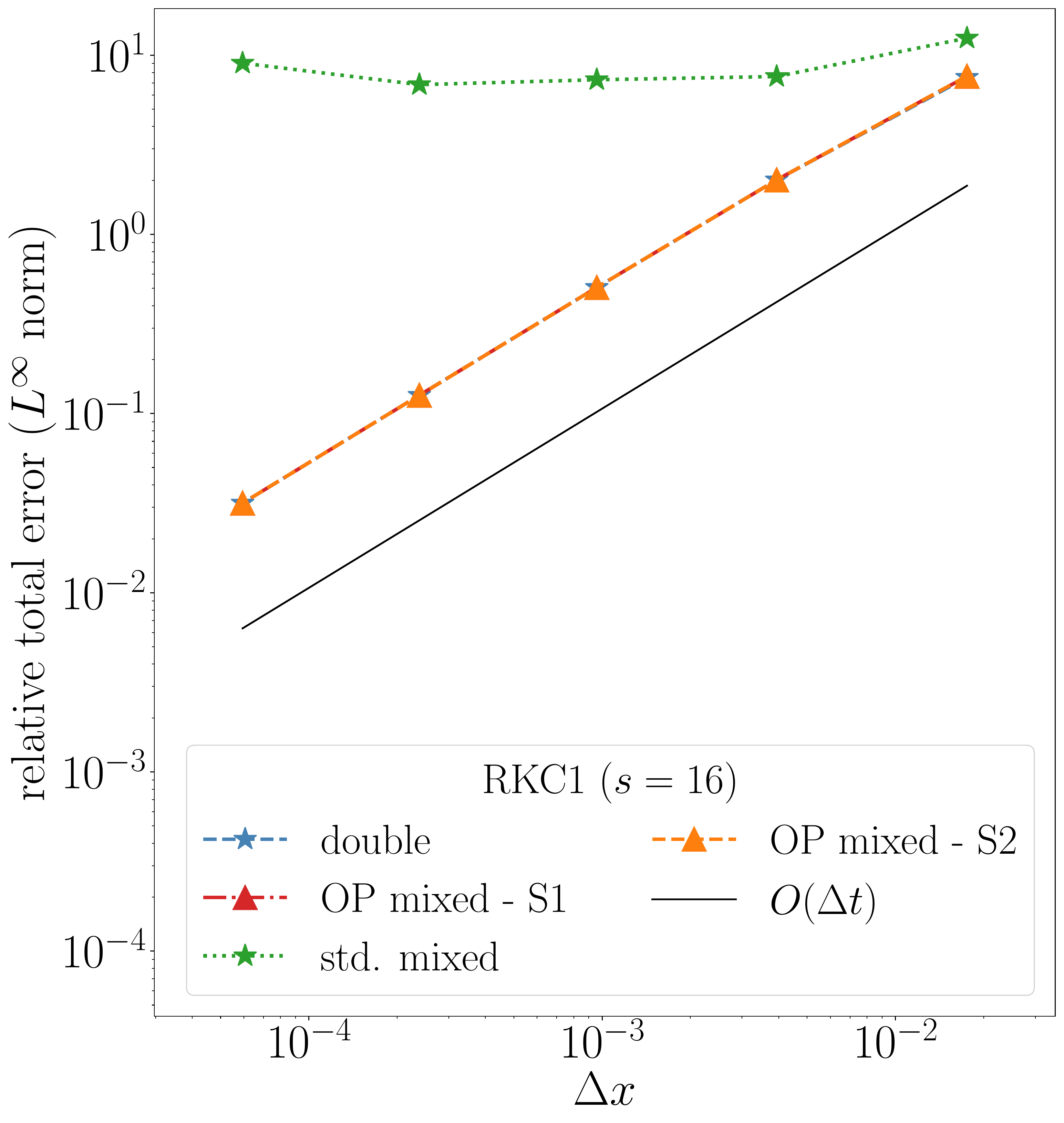}
	\end{subfigure}
	\hspace{6pt}
	\begin{subfigure}{0.38\textwidth}
		\centering
		\includegraphics[width=\textwidth]{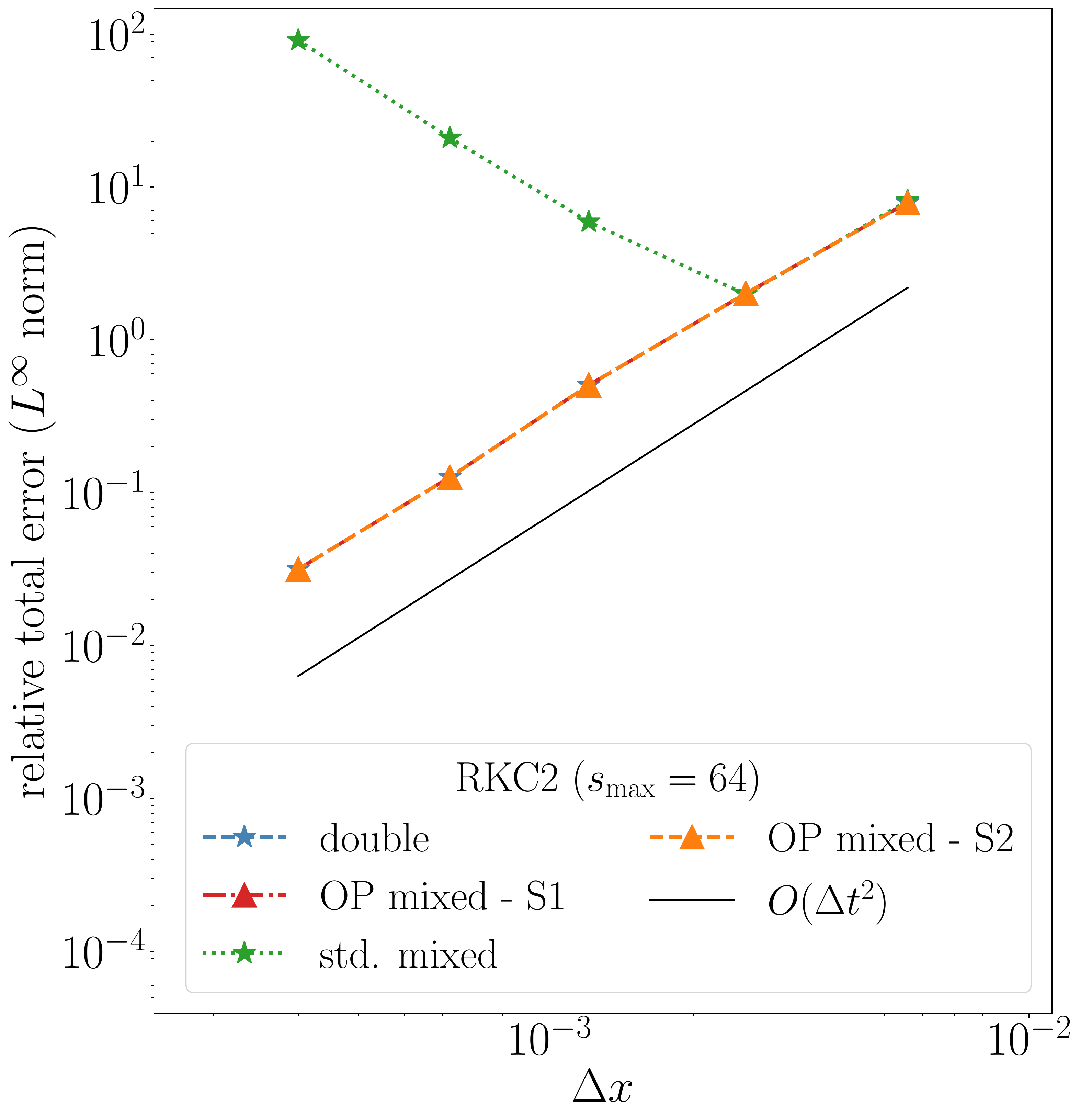}
	\end{subfigure}
	\caption{\textit{Mixed-precision RKC: total error vs $\Dt$ for the nonlinear heat equation in 2D. ``OP'' stands for order-preserving and S1 and S2 stand for Scenario 1 and 2 respectively, while ``std.~mixed'' indicates a standard non-order preserving mixed-precision implementation. While we kept $\Dt = O(N^{-2})$ for RKC1, we instead used $\Dt =O(N^{-1})$ for RKC2 and increased $s$ accordingly up to $s=64$. The convergence orders are as predicted by the theory.}}
	\label{fig:9}
\end{figure}
\begin{figure}[h!]
	\centering
	\begin{subfigure}{0.38\textwidth}
		\centering
		\includegraphics[width=\textwidth]{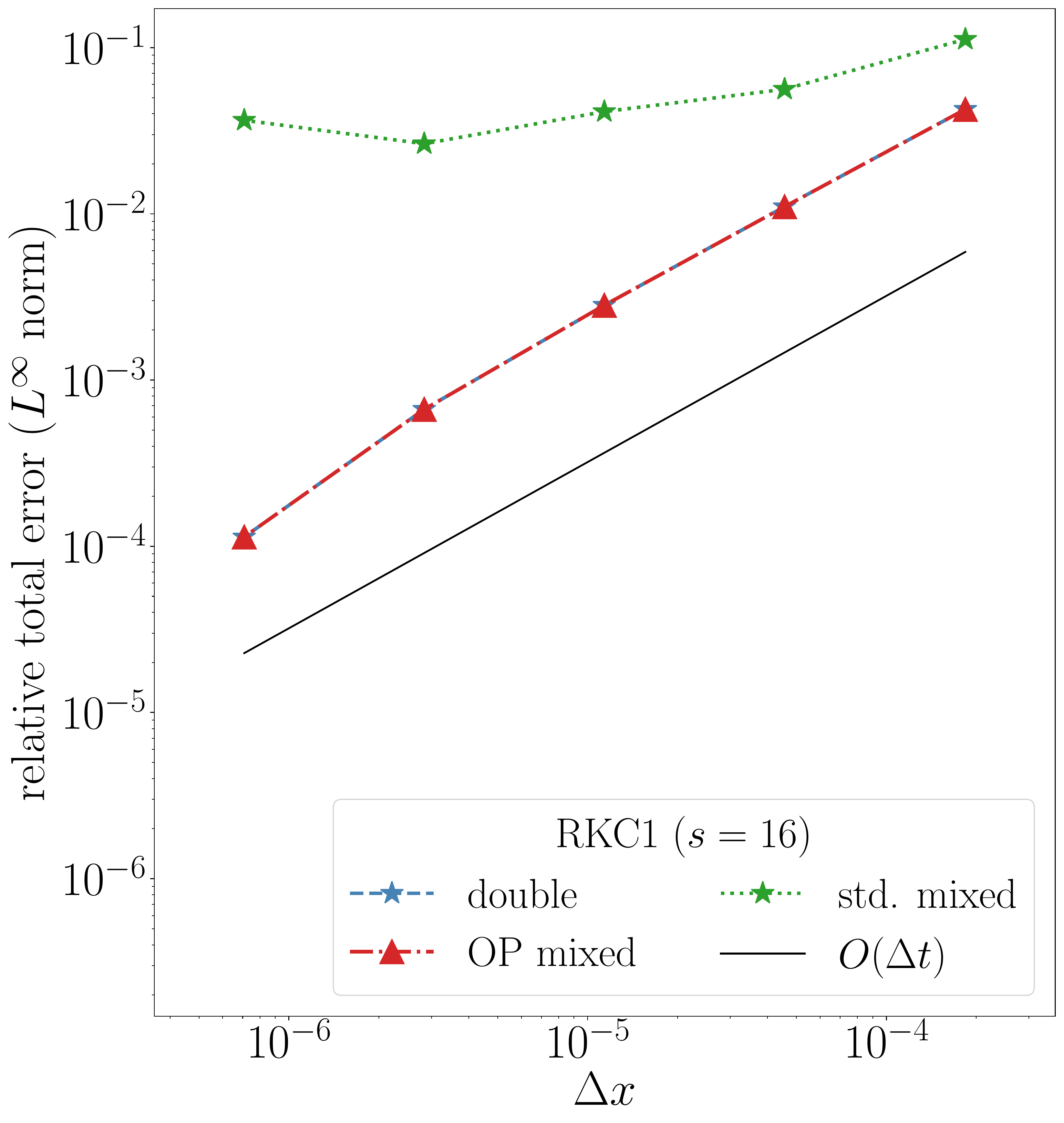}
	\end{subfigure}
	\hspace{6pt}
	\begin{subfigure}{0.38\textwidth}
		\centering
		\includegraphics[width=\textwidth]{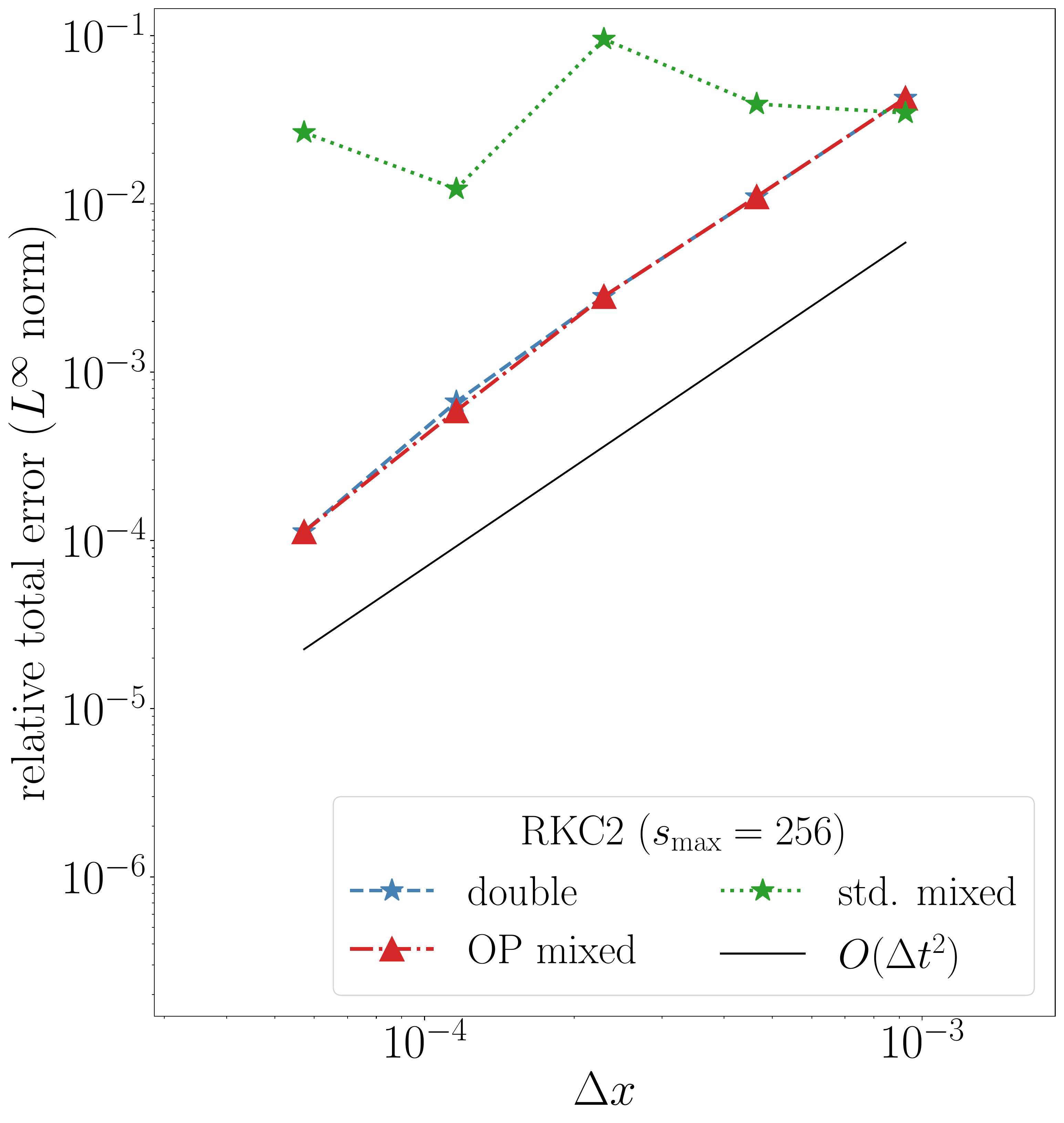}
	\end{subfigure}
	\caption{\moda{\textit{Mixed-precision RKC: total error vs $\Dt$ for the $4$-Laplace diffusion problem. ``OP'' stands for order-preserving, while ``std.~mixed'' indicates a standard non-order preserving mixed-precision implementation. While we kept $\Dt = O(N^{-2})$ for RKC1, we instead used $\Dt =O(N^{-1})$ for RKC2 and increased $s$ accordingly up to $s=256$. The convergence orders are as predicted by the theory.}}}
	\label{fig:10}
\end{figure}

\begin{remark}
	As an example of how advantageous can be using mixed-precision RKC methods, we also solved Problem 1 in 3D with $N=2^5$ and $\Delta t\rho=s^2$, $s=33$, using implicit Euler timestepping using the PETSc software library \cite{balay2014petsc} implementation of Newton's method and the preconditioned conjugate gradient method. As a preconditioner, we used the BoomerAMG algebraic multigrid routines of the Hypre library \cite{hypre}. Overall, the solution of Problem 3 required  on average roughly $6$ preconditioned conjugate gradient iterations and $3$ Newton iterations per time step. Assuming that the cost of 1 multigrid cycle is twice the cost of one high-precision matvec (see Section 5 in \cite{briggs2000multigrid}), we obtain that the number of high-precision matvecs required by implicit Euler is roughly $54$ per timestep, compared to only one high-precision matvec and $32$ half-precision matvecs for a bfloat16/double order-preserving mixed-precision RKC1.
\end{remark}

\section{Conclusions}
\label{sec:conclusions}
We presented new mixed-precision explicit stabilized schemes for stiff differential equations, considering both stiff and multirate problems. While the error of na\"ively implemented mixed-precision methods stagnates due to rounding errors, the mixed-precision schemes we proposed preserve the full order of convergence of the original high-precision methods. This order preservation is achieved by performing only one or two (for first- and second-order methods respectively) high-precision evaluations of the right-hand side, while the remaining function evaluations are only needed to preserve stability, and can be performed in low precision.
Our order-preserving mixed-precision schemes were constructed by linearizing the original methods and carefully evaluating the Jacobian of the right-hand side in low precision. For this purpose, we proposed different strategies for accurate low-precision Jacobian evaluations.

We showed that the mixed-precision methods preserve the original order of convergence, see \cref{thm:convmpRKCbis,thm:convmpmRKC}, and we studied their stability properties in \cref{thm:stabmpRKC}. We remark that our worst-case rounding error analysis does not take into account roundoff cancellation effects, which explains why our schemes behave better in practice than in theory. Since rounding errors disrupt all smoothness and spectral properties of the solution we were unable to prove stability in the standard ODE sense. However, extensive numerical experiments show that our \moda{methods remain stable}. Through our numerical experiments we also confirmed that the mixed-precision schemes preserve the full order of convergence, and that their error is barely distinguishable from the error of the original high-precision schemes.

Our work naturally extends to other explicit stabilized methods based on orthogonal polynomials, and an extension to strong-stability-preserving RK methods is in preparation. Possible other future extensions to this work include the design of mixed-precision explicit stabilized methods for stiff stochastic differential equations, which are often run on chips supporting low-precision arithmetic.
\modb{Another possible research direction would be to investigate the relaxation of \cref{assumption:high_precision_exact}. We expect the rounding errors introduced by high-precision computations to be always negligible (as seen in our numerical experiments). However, a more detailed rounding error analysis could provide us with the insight needed to construct RK schemes that exploit a multi-precision or a fully low-precision implementation to further improve efficiency or obtain smaller error constants. We believe that the use of compensated summation \cite{higham1993accuracy} or stochastic rounding \cite{croci2021stochastic,CrociGilesSR2020,ConnollyHighamMary2020} could especially be beneficial in reduced-precision schemes.}
Open questions remain the design of stable order-preserving mixed-precision strategies for high-order stabilized methods, and the development of a stability theory that is able to circumvent the analytical obstacles deriving from rounding errors, namely loss of smoothness and destruction of spectral properties.

\section*{Acknowledgements}
We would like to thank Giacomo Garegnani for introducing us and making this project possible, and Milan Kl\"ower for the useful discussions and his help in making our low-precision emulator faster. We are also extremely grateful to Assyr Abdulle and Michael B.~Giles for giving us the freedom to pursue our own independent research.

\printbibliography

\begin{appendices}
\renewcommand{\thesection}{\Alph{section}}

\section{Technical results for the mixed-precision RKC1 and RKC2 methods}
\label{sec:sample:appendix}
Here we prove \cref{lemma:jacapp}. We indicate with $C$ a generic positive constant that only depends on $\bf$ and not on $s,u,\Dt$. The actual value of $C$ might change from line to line.

\begin{lemma}\label{lemma:jacapp}
	Let $\hat\Delta\bm{f}_j$ as in \cref{eq:defDf2}, then $\hat\Delta\bm{f}_{j} = \bm{f}(\hby^n+\hbd_j)-\bf(\hby^n) + O(\sqrt{u}\Dt+\Dt^2)$.
\end{lemma}
\begin{proof}
	We have
	\begin{equation}
	\hat{\bg}(\hby^n+\delta\hbd_j) -\bg(\hby^n) = \bg(\hby^n+\delta\hbd_j)-\bg(\hby^n)+\bm{r} = \bg'(\hby^n)\delta\hbd_j+\bm{r}+\bm{t},
	\end{equation}
	where $\bm{r}$ and $\bm{t}$ represent rounding and truncation errors, respectively.
	It holds $\Vert\bm{r}\Vert_2\leq C u$ and $\Vert \bm{t}\Vert_2\leq C \delta^2\Vert\hbd_j\Vert_2^2\leq C\delta^2\Dt^2\leq C u$, where we used $\Vert\hbd_j\Vert_2\leq C\Dt$. Hence
	\begin{equation}\label{eq:devhatJ}
	\begin{aligned}
	\delta^{-1}(\hat{\bg}(\hby^n+\delta\hbd_j)-\bg(\hby^n))
	&= \bg'(\hby^n)\hbd_j+\delta^{-1}(\bm{r}+\bm{t})=\bg'(\hby^n)\hbd_j+O(\sqrt{u}\Dt).
	\end{aligned}
	\end{equation}
	Therefore, using $\hat A \hbd_j=A\hbd_j+\Delta A_j\hbd_j$, and \cref{eq:defDf2}, it holds
	\begin{equation}\label{eq:err2}
	\begin{aligned}
	\hat\Delta\bm{f}_j
	&= A\hbd_j +\bg'(\hby^n)\hbd_j+O(\sqrt{u}\Dt)+\Delta A_j\hbd_j
	= \bm{f}'(\hby^n)\hbd_j + O(\sqrt{u}\Dt),
	\end{aligned}
	\end{equation}
	where we used $\Vert \Delta A_j \hbd_j\Vert_2\leq \barc \barm^2\Vert A\Vert_2 u\Vert\hbd_j\Vert_2\leq C\barc \barm^2\Vert A\Vert_2 u \Dt$. We conclude by Taylor expanding $\bm{f}(\hby^n+\hbd_j)-\bf(\hby^n)$.
\end{proof}

\section{Technical results for the mixed-precision mRKC method}\label{app:proofslemmasmpmRKC}
Now we prove \cref{lemma:hatDeltafe}. In what follows, we denote with $C$ a generic positive constant depending on $\ff$ and $\fs$, and not on $s,m,\eta,u,\Dt$. The actual value of $C$ might change from line to line.
Before proving \cref{lemma:hatDeltafe} we first need another auxiliary lemma:
\begin{lemma}\label{lemma:hatbfe}
	Let $\by\in\R^n$, $\hbfe$ as in \cref{eq:defhbfe}, and $\bfe$ as in \cref{eq:defbfe}. Then $\hbfe(\by)=\bfe(\by)+O(u+\eta^2)$.
\end{lemma}
\begin{proof}
	In \cref{eq:defhbfe} we replace $\hff(\by+\eta\hbh_j)+\hfs(\by)=\ff(\by+\eta\hbh_j)+\fs(\by)+\bm{r}_{j}$, with $\Vert\bm{r}_{j}\Vert_2\leq C u$. By subtracting \cref{eq:defbfe}, we then obtain
	\begin{equation}
	\be_0 =\bm{0},\qquad
	\be_1=\alpha_1\bm{r}_0,\qquad 
	\be_j = \beta_j\be_{j-1} +\gamma_j\be_{j-2}+\alpha_j(\ff(\by+\eta\hbh_j)-\ff(\by+\eta\bh_j))+\alpha_j\bm{r}_{j-1}.
	\end{equation}
	We then conclude the proof by first using a Taylor expansion of $\ff$, and then invoking \cref{lemma:perturbations}.
\end{proof}
\begin{lemma}\label{lemma:hatDeltafe}
	Assume $\Dt\leq\sqrt{u}$. For $\hat\Delta\fe[j]$ as in \cref{eq:defhDfe} it holds $\hat\Delta\fe[j]  = \bfe'(\hby^n)\hbd_j+O\left((\sqrt{u}+\epsilon)\Dt+\Dt^2\right)$, and thus \cref{eq:requirementmRKC} holds with $\epsilon$ replaced by $\sqrt{u}+\epsilon$.
\end{lemma}
\begin{proof}
	Note that the Jacobian of $\bfe$ exists, as it can be obtained by simply differentiating \cref{eq:defbfe}. Using \cref{eq:defhDfe}, the relation $\delta^{-1}=\Dt/\sqrt{u}$, and \cref{lemma:errhbfe,lemma:hatbfe}, we obtain
	\begin{equation}
	\hat\Delta\fe[j]  = \delta^{-1}\left(\bfe(\hby^n+\delta\hbd_j)-\bfe(\hby^n)\right)+\bm{r} = \bfe'(\hby^n)\hbd_j+\bm{r}+\bm{t},
	\end{equation}
	where again $\bm{r}$, $\bm{t}$ represent rounding and truncation errors, respectively, and satisfy
	\begin{equation}
	\Vert\bm{r}\Vert_2\leq C\delta^{-1}(u+\epsilon\eta+\eta^2)\leq C(\sqrt{u}\Dt+\epsilon\eta\Dt/\sqrt{u}+\eta^2\Dt/\sqrt{u}),\qquad 
	\Vert\bm{t}\Vert_2\leq C\delta^{-1}\Vert\delta\hbd_j\Vert_2^2\leq C\delta\Dt^2\leq C\sqrt{u}\Dt.
	\end{equation}
	Using $\eta=O(\Dt)$ (see \cite{AGR20}), and $\Dt\leq\sqrt{u}$, we obtain $\Vert\bm{r}\Vert_2+\Vert\bm{t}\Vert_2=O(\sqrt{u}\Dt+\epsilon\Dt+\Dt^2)$, which concludes the proof.
\end{proof}

\end{appendices}

\end{document}